\documentclass[10pt]{amsart}
\usepackage{a4}
\usepackage{amsmath,amssymb,amsfonts,amsthm}
\usepackage{color}
\usepackage{xargs}

\usepackage{ifpdf}
\ifpdf
    \usepackage[pdftex]{graphics,graphicx}
    \DeclareGraphicsExtensions{.png,.pdf}
\else
   \usepackage[dvips]{graphics,graphicx}
   \DeclareGraphicsExtensions{.eps}
\fi
\graphicspath{{.}{figures/}}

\theoremstyle{plain}

\newtheorem{theorem}{Theorem}[section]
\newtheorem{lemma}[theorem]{Lemma}
\newtheorem{proposition}[theorem]{Proposition}
\newtheorem{corollary}[theorem]{Corollary}
\theoremstyle{definition}
\begingroup
\newtheorem{remark}[theorem]{Remark}
\newtheorem{definition}[theorem]{Definition}
\endgroup

\numberwithin{equation}{section}

\def\ds{\displaystyle}

\def\R{\mathbb{R}}

\def\N{\mathbb{N}}
\def\dist{\textup{dist}}
\def\osc{\textup{osc}}
\def\H{\mathcal{H}}

\def\E{\mathcal{E}}
\def\F{\mathcal{F}}
\def\M{\mathcal{M}}
\def\RT{\R^N\times[0,T]}
\def\QT{\R^N\times(0,T)}
\def\Jet{\mathcal{J}}
\def\PJet{\mathcal{P}}
\def\half{\frac{1}{2}}
\def\ov{\overline}
\def\oc{\ov{c}}
\def\uc{\underline{c}}

\def\Ins{\mathfrak{M}}
\def\Reg{\mathfrak{C}}
\newcommand{\regu}[1][]{{\ell #1,\beta}}
\def\Regm{\mathfrak{C}^{1,1}}
\def\la{\left\langle}
\def\ra{\right\rangle}

\newcommand{\SymN}{\mathbb M^{N\times N}_{\rm sym}}

\newcommand{\f}{\varphi}
\newcommand{\e}{\varepsilon}
\newcommand{\de}{\delta}
\newcommand{\cp}{\mathfrak{Cap}_p}
\newcommand{\inter}[1]{\mathring{#1}}
\newcommand{\pa}{\partial}

\newcommand{\medint}{-\kern -,375cm\int}
\newcommand{\medintinrigo}{-\kern -,315cm\int}

\newcommand{\kto}{\stackrel{\H}\to}

 \title[Nonlocal curvature flows]{Nonlocal curvature flows}
 \author[A. Chambolle]
 {Antonin Chambolle}
 \address[Antonin Chambolle]{CMAP, Ecole Polytechnique, CNRS, France}
 \email[A. Chambolle]{antonin.chambolle@cmap.polytechnique.fr}

 \author[M. Morini]
 {Massimiliano Morini}
 \address[Massimiliano Morini]{Dip.~di Matematica, Univ.~Parma, Italy}
 \email[M. Morini]{massimiliano.morini@unipr.it}

 \author[M. Ponsiglione]
 {Marcello Ponsiglione}
 \address[Marcello Ponsiglione]{Dip.~di Matematica,
 Univ.~Roma-I ``La Sapienza'', Roma, Italy}
 \email[M. Ponsiglione]{ponsigli@mat.uniroma1.it}

\begin{document}
\vskip .2truecm

\begin{abstract}
\small{
This paper aims at building a unified framework to deal with a wide class of local and nonlocal translation-invariant geometric flows. 
First, we  introduce a class of generalized curvatures, and  prove the existence and uniqueness for the level set formulation of  the corresponding geometric flows. 

We then introduce a class of generalized perimeters, whose first variation is an admissible generalized curvature.  Within this class, we implement a minimizing movements scheme and we prove that it approximates the viscosity solution of the corresponding level set PDE.

We also describe several examples and applications.  Besides recovering and presenting in a unified way  existence, uniqueness, and approximation results for several geometric motions already studied and scattered in the literature, the theory developed in this paper allows us to  establish also new  results.


}
\end{abstract}

\maketitle
\tableofcontents
 \bibliographystyle{plain}

\section{Introduction}
In this paper we present a unified approach to deal  with a large class of possibly nonlocal
geometric flows; i.e., evolutions of sets $t\mapsto E(t)$ governed by a law of the form   
\begin{equation}
\label{oee} 
V(x,t) = -\kappa(x,E(t)),
\end{equation}
where $V(x,t)$ stands for the (outer) normal velocity of the boundary $\pa E(t)$ at $x$ and the function $\kappa(\cdot, E)$ will be referred to as a  {\em generalized curvature} 
of $\pa E$, in analogy with the classical theory.  

If the function $\kappa$ depends only on how $\pa E$ looks   around $x$, then the flow is local in nature. This is of course  the case of the classical {\it mean curvature flow}, where $\kappa(\cdot, E)$ is nothing but the mean curvature of $\partial E$, i.e.  the first variation of the standard perimeter functional at $E$. On the other hand, for some of the relevant flows that have been intensively studied in recent years the generalized curvature $\kappa$ is truly nonlocal and depends on the global shape of the evolving set $E(t)$ itself. It happens for instance for  {\em fractional mean curvature flows}, where the corresponding curvatures are  defined as the first variation of the so-called {\em fractional perimeters}, and hence represent the natural nonlocal counterparts of the classical mean curvature in the fractional framework (\cite{CRS}, \cite{CS},  \cite{I}, \cite{V}).

As already made clear by  the aforementioned examples,  a relevant class of curvatures is given by those that can be seen as the first variation of  some  {\em generalized perimeters}; we refer to such a class  as {\it variational curvatures}. 
It is important to observe that when $\kappa$ is variational, then \eqref{oee}  can be interpreted as  the gradient flow of the corresponding perimeter, with respect to a suitable $L^2$ Riemannian structure. In the case of the classical mean curvature flow, this observation underpins
the {\em minimizing movements} algorithm implemented by Almgren-Taylor-Wang  in their pioneering work \cite{ATW} (see also \cite{LS}). 

The strong formulation of the motion \eqref{oee}, which requires smoothness,  faces the possible formation of singularities in finite time. Thus, the evolution  can only be defined locally in time, which is clearly unsatisfactory from the applications point of view. On the other hand,  Brakke~\cite{Br}
proposed a weak  formulation  for motion by mean curvature that  resulted in deep regularity results but had the disadvantage of producing a lack of uniqueness. These uniqueness issues are often overcome by the more recent  notion of {\em generalized motion} that is associated to the so-called {\em level set approach}.  Such an approach is based on representing the evolving set as  the zero super-level set of a function $u(x,t)$,  which is defined for all times as  viscosity solution to the  (degenerate) parabolic partial differential equation
\begin{equation}\label{lsfi}
u_t(x,t) + |D u(x,t)| \kappa(x, \{y: \, u(y,t) > u(x,t)\} ) = 0.
\end{equation}
 The level-set method  was proposed in \cite{OS},  analytically validated in  \cite{EvansSpruckI} for the motion by mean curvature and in \cite{CGG} for more general local motions. In the case of the classical mean curvature (and of several different local curvatures) viscosity solutions to \eqref{lsfi} with a prescribed initial datum are unique. Note also that \eqref{lsfi} prescribes that all the super-level sets of $u$ evolve according to \eqref{oee}.

The paper is divided into two parts: The main focus of Part 1  is to develop a general level set approach for the  geometric motions \eqref{oee}, while Part 2  is aimed at  implementing  a general minimizing movements scheme \`a la  Almgren-Taylor-Wang for a large class of variational curvature motions,  and at exploring the connections between the two approaches. In carrying out the program of Part 1 and Part 2, we recover and present in a unified way  existence, uniqueness, and approximation results for several geometric motions already studied and scattered in the literature, but we also establish new  results (see the end of this Introduction).

We now describe the content of  the paper  in more details. In Part 1 we  introduce the class of generalized curvatures we deal with and then
we set up the viscosity theory for the corresponding generalized level set equation \eqref{lsfi}. 

To be more specific, a generalized curvature  $\kappa$ is a function defined on the admissible pairs $(x, E)$, where $E$ is a  $C^2$-set\footnote{As is often
the case in viscosity solution approaches, we may in fact assume that the
curvature is \textit{a priori} defined only for smoother sets, and will later on consider also stronger regularities.} with compact boundary and $x\in \pa E$,  and such that $\kappa(x, \cdot)$  is monotone non-increasing with respect to the  inclusion between sets touching at $x$ and  continuous  with respect to $C^2$-convergence of sets.  We also assume the translation invariance, i.e., $\kappa(x, E)=\kappa(x+y, E+y)$ for all admissble pairs $(x,E)$ and for all $y\in \R^N$. 

In order to fully exploit the second order viscosity solutions formalism, we need to extend the definition of  the right-hand side of \eqref{lsfi} to non-smooth sets. This 
is achieved by considering suitable lower and upper semi-continuous envelopes $\kappa_*$ and $\kappa^*$ of $\kappa$ that are then employed to define viscosity subsolutions and supersolutions, respectively. The domain of definition of the relaxed curvatures $\kappa_*$ and $\kappa^*$
 is made up of the elements of  the form $(x, p, X, E)$, where $E$ is now any measurable set, $x\in \pa E$, and $(p, X)$ belongs to the second order super-jet (as far as $\kappa_*$ is concerned) or sub-jet (as far as $\kappa^*$ is concerned) of $E$ at $x$. By construction, $\kappa_*$ and $\kappa^*$ turn out to be lower and upper semicontinuous, respectively, with respect to the Hausdorff convergence of sets and a suitable notion of uniform convergence of super-  and sub-jets (see Definition \ref{defuniconv}). 

Remarkably, the above semicontinuity property is weaker than the semicontinuity with respect to $L^1$-convergence, which is one of the main hypotheses in the approach of \cite{S}. This significantly increases  the class of admissible curvatures we can treat. For instance,  the aforementioned fractional curvatures  are not semicontinuous with respect to the $L^1$-convergence and the corresponding $L^1$-relaxation would be useless (equal to $-\infty$ for every closed set). Let us also notice that $\kappa_*$ and $\kappa^*$ are defined only on ``geometrically meaningful'' objects and this represents a further difference from \cite{S}, while the relaxation procedure used to define the semicontinuous envelopes of $\kappa$  is reminiscent of the approach of Cardaliaguet and co-authors (see \cite{C-NLI, C-NLII, CL07, CL08, CardaliaguetRouy}). 
Section~\ref{sec:2} is entirely devoted to  setting up the viscosity formalism. The latter task being accomplished,  a general 
 existence theorem for the level set formulation of \eqref{oee} and for the class  of generalized translation invariant curvatures we specified before is easily obtained through an application of the standard Perron method (see Theorem~\ref{thexun}).  

The drawback of such a generality is that the classical 
  strategy to prove the comparison principle and, in turn, the uniqueness of viscosity solutions may fail. Uniqueness is the main focus of Section~\ref{sec:uniqueness}, where we provide two different treatments, distinguishing between {\em first order} and {\em second order} geometric flows. 
  
  Roughly speaking, we say that a geometric flow is of first order if the envelopes $\kappa_*$ and $\kappa^*$ do not depend on the 
  second derivative variables (see condition (FO) at the beginning of Subsection~\ref{subsec:1storder}). Fractional mean curvature motions and the {\em shape flow generated by the $p$-capacity} in $\R^N$ are relevant examples of first order geometric motions, see below.  Uniqueness for such motions follows from the Comparison Principle provided by Theorem~\ref{th:CP1st}. Let us mention that the main technical tool used in the first order uniqueness theory is represented by the well-known Ilmanen Interposition Lemma (see \cite{IL,C-NLI,C-NLII}).

The uniqueness theory for second order flows is harder and developed in Subsection~\ref{subsec:2ndorder}.  In order to understand the source of difficulty, notice that semicontinuity properties  of $\kappa_*$ and $\kappa^*$ are not sufficient to conclude that the subsolution and the supersolution inequalities extend to elements of the {\em closure} of the {\em parabolic} super- and sub-jets, respectively. This means that the usual machinery to establish uniqueness for second order equations, which is   based on the celebrated Ishii's Lemma (see \cite{CIL}),  cannot be applied.  Indeed,   Ishii's Lemma states that if $u$ is upper semicontinuous, $v$ is lower semicontinuous, and $u-v$ attains a (local) maximum at $(x, t)$, then there exists at least one element belonging to both the {\em closure} of the parabolic super-jet of $u$ and  the {\em closure} of the parabolic sub-jet of $v$. Such a separating element is obtained through a limiting procedure, which involves  regularizations via inf- and sup-convolutions, the Alexandrov theorem on the a.e. second order differentiability of semi-convex functions,  and a perturbation argument due to Jensen (\cite[Lemma~A.3]{CIL}). Since we lack the proper semicontinuity properties, we need to avoid ``passing to the limit''.
Our proof of  the second order Comparison Principle (see Theorem~\ref{th:CP2nd}) still uses all the aforementioned tools but combines them 
with some new insight, allowing us to avoid the limiting procedure. The little price we have to pay is a reinforced continuity assumption on $\kappa$ (see beginning of Subsection~\ref{subsec:2ndorder}), which is nevertheless satisfied by all the relevant examples we have in mind.

In Part 2 of the paper we take on a  variational approach to geometric flows based on the minimizing movements. To this aim, we introduce a class of functionals referred to as {\it generalized perimeters}. More precisely, denoting by $\Ins$ the class 
of Lebesgue-measurable sets, we call a generalized perimeter any {\em translation invariant} set function $J:\Ins \to [0,+\infty]$, which is insensitive  to modifications on negligible sets,   finite on  $C^2$-sets with compact boundary,   lower semicontinuous with respect to $L^1_{loc}$-convergence, and satisfying the following {\em submodularity condition}:
For any masurable sets $E, \, F\subset \R^N$,
\begin{equation}\label{eq:submoi}
J(E\cup F)\,+\, J(E\cap F)\ \le\ J(E)\,+\,J(F)\ .
\end{equation}
It turns out that the latter condition is  a convexity condition in the following sense: Extend  $J$  to   $L^1_{loc}(\R^N)$ by enforcing the {\em generalized coarea formula} \eqref{visintin}(see  \cite{V}). Then,  the submodularity property of $J$ is equivalent to the convexity of the extended functional (see \cite{ChamGiacoLuss}). 
A first important consequence of  submodularity is that if $J$ admits a first variation $\kappa$, then such a curvature is monotone. More in general, if $J$ is smooth enough, then its first variation is an admissible generalized curvature. 

As mentioned before, the main achievement of Part 2 is the implementation of a {\em generalized Almgren-Taylor-Wang minimizing movements scheme} to approximate  geometric motions associated with variational generalized curvatures. We recall that, given an initial set $E_0$  and a time step $h>0$,  such a scheme provides a discrete-in-time evolution obtained by  solving iteratively suitable incremental minimum problems. The energy to be minimized at each discrete time  is the sum of the generalized perimeter and of a suitable dissipation that penalizes the $L^2$-distance from the boundary of the set obtained at the previous step. It turns out, once again due to submodularity,  that the discrete evolutions satisfy the comparison principle.    Adopting the point of view introduced in \cite{Chambolle}, we combine the minimizing movements scheme with the level set framework. More precisely, given an initial function $u_0$, we let all its super-level sets evolve according to our discrete scheme, selecting at each time step the minimal solution (the maximal one would work as well). In light of the discrete comparison principle, the evolving sets are themselves the super-levels sets of  a discrete-in-time  evolving function $u_h(\cdot, t)$. 

We then study the limiting behavior of $u_h$ as the time step $h\searrow 0$ and in one of the main results of this paper  we establish the following general {\em consistency principle}: For all variational generalized curvatures, the discrete evolutions $u_h$ provided by the minimizing movements scheme converge 
(up to subsequences) to a viscosity solution of the level set equation. In particular, under the additional assumptions that guarantee uniqueness, the whole sequence converges (to the unique viscosity solution). Let us mention that in the case of the mean curvature motion  the consistency between the level set and the minimizing movements approach has been established in \cite{Chambolle}
(see also~\cite{EGK}). From the technical point of view, the convergence analysis combines several ingredients, among which we mention some careful estimates on the speed of propagation of the support of the initial function $u_0$ and a subgradient inequality involving the generalized curvature, which is crucial in obtaining the limiting sub(super)-solution property. The subgradient inequality is a consequence of the analysis developed in Section~\ref{secfirstvar} (see \eqref{intW}). 
Note that since we don't have a regularity theory for the minimizers of the incremental problems, all our argument are necessarily  variational in nature.

A relevant  consequence of our general consistency principle is that the generalized perimeter  of the super-level sets for which no fattening occurs during the evolution is non-increasing in time (note that the no fattening condition is satisfied by almost all super-level sets).  Moreover, we show that a suitable inner regularization of the generalized perimeter $J$, defined on open sets $A$ as the $\inf- \liminf$ of $J$ along sequences of open sets approximating $A$ from the interior (see Definition \ref{relJ}), is always non-increasing in time (see Subsection~\ref{subsec:pd}).

We conclude this introduction by highlighting  some relevant examples and applications of our general theory that are presented in the paper (see Section~\ref{sec:examples}). Such examples are by no means exhaustive and  serve indicating  the scope of our theory:
 
 --\,{\it Local motions:} The theory of local generalized  mean curvature motions established in \cite{CGG} fits into our theory as  particular case.

--\,{\it Fractional mean curvatures motions:} As mentioned before, these are the first order  geometric flows associated with  the so-called fractional perimeters. Existence and uniqueness of viscosity solutions to the corresponding level set equation were already established in \cite{I} and are recovered here. On the other hand, the convergence of the minimizing movements scheme provided by our theory is  new for these motions and furnishes an 
 approximation algorithm that is alternative to the threshold-dynamics-based one studied in \cite{CS}.
 
 --\,{\it Capacity flows:} Given $1<p<N$, consider the generalized perimeter that  coincides with the $p$-capacity in $\R^N$ on bounded sets of class $C^2$. It can be shown that the associated curvature $\kappa(x,E)$ is given by $|Dw_E(x)|^p$, where $w_E$ denotes the {\em capacitary potential} of $E$. Thus, the associated geometric motion is somewhat related to the Hele-Show type flows studied in \cite{CL07, CL08, CardaliaguetRouy} (see also \cite{C-NLI, C-NLII}).  Our general results yield  existence, uniqueness, and approximation via minimizing movements also for this shape flow, which turns out to be of first order according to our terminolgy (see Subsection~\ref{subsec:capacity}).

 --\,{\it Second order nonlocal motions}: Our theory includes all the  generalized curvatures treated in \cite{S}, that are in addition translation invariant. As already mentioned, compared to our approach the theory of \cite{S} requires stronger continuity assumptions on the Hamiltonians and more restrictive growth conditions that rule out singular behavior  of generalized curvatures along shrinking balls. 
 As a further example, which  is not covered by the theory developed in  \cite{S} while fitting into ours, we mention here the generalized perimeter
  introduced in \cite{BKLMP} in the framework of two-phase Image Segmentation. We will refer to it  as  {\em regularized pre-Minkowski content} of a set since it consists  in a suitable 
  regularization  of the Lebesgue measure of the $\rho$-neighborhood of the essential boundary, for some fixed $\rho>0$.  The corresponding geometric motion was studied in \cite{CMP0}. In the latter work we computed the associated generalized curvature and  proved  convergence of the minimizing movements scheme to a viscosity solution of the level set equation, but we were unable to establish uniqueness. The theory of the present paper allows us to recover the existence and approximation   results of \cite{CMP0} and, in addition, yields the  uniqueness of the geometric motion.

A final remark regarding the translation invariance and the continuity assumptions on $\kappa$ is in order. Concerning the former, we believe that it  
can be removed at the expense of  some additional technical effort  but within the main theoretical framework introduced in this paper. 
On the other hand, the  continuity assumptions that guarantee the stability property and  the comparison principle are of a more subtle and essential nature. They exclude from our theory some relevant irregular perimeters, as crystalline perimeters (that also would require a different specific viscosity level set formulation).
It is not clear at the present which is (if any) the weakest continuity assumption on $\kappa$ yielding  the uniqueness of the geometric flow.

\newpage
\part{Nonlocal Curvature Flows}\label{part:1}
In this part we introduce the class of generalized curvatures we deal with. Then, we introduce the notion of viscosity solutions for the level set equation of the geometric flow, 
and we prove the existence and uniqueness of the solution. 
\section{Viscosity solutions: definition, properties, existence}\label{sec:2}

We begin this section by introducing the class of generalized curvatures we will deal with. 
\subsection{Axioms of a nonlocal curvature}\label{sec:curvature}
Let $\Reg$ be the class of  subsets of $\R^N$, which can be obtained as the closure of an open set with  compact $C^{\regu}$ boundary, and  let $\Ins$ be the class of all measurable subsets of $\R^N$. Throughout the paper
$\ell\ge 2$ and $\beta\in [0,1]$ will be fixed, the reader may simply
assume $\ell=2$, $\beta=0$.

In this part, we are given for every set $E\in\Reg$  a function $\kappa(x,E)\in\R$ defined at each point $x\in\partial E$, and referred to as the ``curvature'' of the set $E$ at $x$. This curvature will satisfy the following axioms:
\begin{itemize}
\item[A)] Monotonicity: If 
$E, F\in\Reg$ with $E\subseteq F$, and if $x\in \partial F\cap \partial E$,
then $\kappa(x,F)\le\kappa(x,E)$;
\item[B)] Translational invariance: for any $E\in\Reg$, $x\in\partial E$,
$y\in\R^N$, $\kappa(x,E)=\kappa(x+y,E+y)$;
\item[C)] Continuity:
If $E_n\to E$ in $\Reg$ and $x_n\in \partial E_n \to x\in \partial E$, then $\kappa(x_n,E_n)\to \kappa(x,E)$.
\end{itemize}
Axioms A), B) and C) are enough to prove the existence of a generalized solution of the geometric flow~\ref{oee}. 
For axiom C), we say that $E_n\to E$ in $\Reg$ if the boundaries converge
in the $C^{\regu}$ sense (meaning for instance
that they can be described locally
with a finite number of graphs of functions which converge in $C^{\regu}$).

By assumption C), for any $\rho>0$ we can define the quantities
\begin{equation}\label{defbarc0}
\oc(\rho) \ :=\ 
\max_{x\in\partial B_\rho} \max\{\kappa(x,B_\rho),-\kappa(x,\R^N\setminus B_\rho)\}\,,
\end{equation}
\begin{equation}\label{defbarc}
\uc(\rho) \ :=\ 
\min_{x\in\partial B_\rho} \min\{\kappa(x,B_\rho),-\kappa(x,\R^N\setminus B_\rho)\}\,,
\end{equation}
which are continuous functions of $\rho>0$.
Assumption D) below will guarantee that the curvature flow starting from
a bounded set remains bounded at all times.
\begin{itemize}
\item[D)] Curvature of the balls: 
There exists $K>0$ such that
\begin{equation}\label{lwrbdkappa}
\uc(\rho)\ >\ -K \ >\ -\infty\,.
\end{equation}
\end{itemize}

Without assuming  D)
most of the results in this paper  remain true, except
that the flow starting from a given set { with compact boundary}  will be defined possibly
up to some time $T^*<+\infty$ where { its boundary } becomes unbounded
and the framework of this paper cannot be applied anymore. The
time $T^*$ can be estimated from $\uc(\rho)$. Observe that thanks
to the monotonicity axiom A), 
the functions $\rho\mapsto\oc(\rho)$
and $\rho\mapsto\uc(\rho)$ are nonincreasing.

\subsection{Viscosity solutions}
Here we introduce the level set formulation of the 
geometric evolution problem 
$V = -\kappa$, 
where $V$ represents the normal velocity of the boundary of the evolving sets $t \mapsto  E_t$, and 
we give a proper notion of viscosity solution. We refer to~\cite{GigaBook}
for a general introduction of this approach for geometric evolution problems.  
The level set approach consists in solving the following parabolic
Cauchy problem 
\begin{equation}
\label{levelsetf}
\begin{cases}
\ds \partial_t u(x,t) + |D u(x,t)| \kappa(x,  \{ y: \, u(y,t)\ge u(x,t)\}) \,=\,0 \\ 
 u(0,\cdot) = u_0.
\end{cases}
\end{equation}
Here and  in the following, $D$ and $D^2$ stand for the spatial gradient and the spatial Hessian matrix, respectively.
Notice that if the superlevel sets of $u$ are not smooth, the meaning of \eqref{levelsetf} is unclear. For this reason, it is necessary to use a definition
based on appropriate smooth test functions whose curvature of level sets are
well defined. The appropriate setting is of course the framework of
viscosity solutions.
%
Let us first introduce a class of test functions appropriate for this problem.

Following \cite{IS} {(see also \cite{CMP0})}, we introduce the family  $\F$ of functions {$f\in C^\infty([0,\infty))$}, such that $f(0) = f'(0)= f''(0)=0$, $f''(r) >0$ for all $r$ in a neighborhood of $0$, $f$ is constant in $[M, +\infty)$ for some $M>0$, and 
\begin{equation}\label{IS}
\lim_{\rho\to 0^+} f'(\rho)\, \overline c(\rho) = 0,
\end{equation}
where $\overline c(\rho)$ is the function introduced in \eqref{defbarc0}.
We refer to  \cite[p. 229]{IS} for the proof  that the family $\F$ is not empty.
Note that \eqref{IS} {(recall also \eqref{lwrbdkappa})} implies
\begin{equation}\label{ISvera}
\lim_{\rho\to 0^+} f'(\rho)\, \overline c(f^{-1}(\rho)) = 0,
\end{equation}
since $f^{-1}(\rho)>\rho$ for small values of $\rho$ and $\overline c$ is decreasing.

{We fix $T>0$, and look for geometric evolutions in the time interval $[0,T]$.
In the following, with a small abuse of language, we will say that a function $g:\R^N\times A\to \R$, with $A\subseteq [0, T]$, is constant outside a compact set  if
there exists a compact set $\mathcal K\subseteq \R^N$ such that $g(\cdot,t)$ is constant in $ (\R^N\setminus \mathcal K)$ for every $t\in  A$ (with the constant possibly depending on $t$).
}

\begin{definition}\label{defadmissible}
Let $\hat z=(\hat x,\hat t)\in  \R^N\times (0, T)$ and let $A\subseteq (0, T)$ be any open interval containing $\hat t$.   We will say that $\f\in C^{0}( \R^N\times \overline A)$ is
{\em admissible at the point $\hat z=(\hat x,\hat t)$} if it is { 
of class $C^2$ in a neighborhood of $\hat z$, 
if it is constant out of a compact set, 
  and,}
in case $D\f(\hat z)=0$, the following holds: 
there exists   $f\in\F$  and { $\omega\in  C^\infty([0,\infty))$ with
 $\omega'(0)=0$, $\omega(r)>0$ for $r\neq 0$}
such that
$$
|\f(x,t) - \f(\hat z) - \f_t(\hat z)(t-\hat t)|\le f(|x-\hat x|) + \omega(|t -\hat t|)
$$
for all $(x,t)$ in $\R^N\times A$.
\end{definition}

{
\begin{definition}\label{defviscoC2}
An upper semicontinuous function $u:\R^N\times [0,T]\to \R$ (i.e. $u\in USC(\R^N\times [0,T])$),
constant outside a compact set, is
a viscosity subsolution of the Cauchy problem \eqref{levelsetf} if $u(0,\cdot) \le u_0$, and for all 
$z:= (x,t) \in \R^N\times (0,T)$ and for all $C^\infty$  test functions $\f$
such that $\f$ is admissible at $z$ and $u-\f$ has a
maximum at $z$ (in the domain of definition of $\f$)  the following holds:
\begin{itemize}
\item[i)] If $D\f(z) = 0 $, then $\f_t(z) \le 0$;
\item[ii)]  If the level set $\{\f(\cdot,t) = \f(z)\}$ is noncritical, then 
$$\f_t(z)+ |D\f(z)| \, \kappa\left(x,\{y: \f(y,t)\ge \f(z)\}\right)
\le 0.
$$ 
\end{itemize}
A lower semicontinuous function $u$ (i.e. $u\in LSC(\R^N\times [0,T])$), constant outside a compact set,
is a viscosity supersolution of the Cauchy problem \eqref{levelsetf} if $u(0,\cdot) \ge u_0$, and for all $z\in\R^N\times (0,T)$ 
and for all $C^\infty$  test functions $\f$
such that $\f$ is admissible at $z$ and $u-\f$ has a
minimum at $z$ (in the domain of definition of $\f$)  the following holds:
\begin{itemize}
\item[i)] If $D\f(z) = 0 $, then $\f_t(z) \ge 0$;
\item[ii)] If the level set $\{\f(\cdot,t) = \f(z)\}$ is noncritical, then 
$$
\f_t(z)+|D\f(z)| \, \kappa\left(x,\{y: \f(y,t) > \f(z)\}\right)
\ge 0.
$$ 
\end{itemize}
Finally, a function $u$ is a viscosity solution of the Cauchy problem \eqref{levelsetf} if its upper semicontinuous envelope is a subsolution and 
 its lower semicontinuous envelope is a supersolution of~\eqref{levelsetf}.
 \end{definition}
 }

In this whole paper, we will use (with a small abuse of terminology)  the terms subsolutions and supersolutions (omitting the locution ``of the Cauchy problem \eqref{levelsetf}'') also for functions which do not 
satisfy the corresponding inequalities at time zero.   

While Definition~\ref{defviscoC2} is quite natural, it has the
drawback that the family of possible test functions is too restrictive
to be handy.
As usual in the viscosity theory, we will introduce suitable
lower and upper semicontinuous extensions of $\kappa$, in particular
to general measurable sets. This will allow to give definitions
equivalent to Definition~\ref{defviscoC2}, based on less smooth
test functions {as in Definition \ref{defadmissible} (See Definition \ref{defvisco})},  or on the notion of sub/superjets.

\subsection{Convergence of sets with uniform superjet}

\begin{definition}\label{jets}
Let $E\subseteq \R^N$, $x_0\in\partial E$, $p\in\R^N$, and  $X\in \SymN$. We say that $(p,X)$  is in the superjet $\Jet^{2,+}_E(x_0)$ of $E$ at $x_0$  if for every $\delta>0$ there exists a neighborhood $U_\delta$ of $x_0$ such that, for every $x\in E\cap U_\delta$   
\begin{equation}
 (x-x_0)\cdot  p  + \frac12 (X+\delta  I )  (x-x_0)\cdot (x-x_0)  \ge 0.
\end{equation}
Moreover, we say that $(p,X)$ is in the subjet $ \Jet^{2,-}_E(x_0)$ of $E$ at $x_0$  if
$(-p,-X)$  is in the superjet $\Jet^{2,+}_{\R^N \setminus E}(x_0)$ of $\R^N \setminus E$ at $x_0$. Finally, 
we say that $(p,X)$ is in the jet $\Jet^{2}_E(x_0)$ of $E$ at $x_0$  if
$(p,X)\in  \Jet^{2,+}_{E}(x_0) \cap \Jet^{2,-}_{E}(x_0)$.
\end{definition}
The   above definition of  superjet of sets is consistent  with the classical notion of  superjet of u.s.c. functions. Indeed, if  $u$ is an u.s.c. function, then  $(p,X)$ is in the superjet of $u$ at $x_0$ if and only if $(p,  X )$ is in the superjet of the superlevel set $\{u\ge u(x_0)\}$, according to Definition \ref{jets}.   
In the same way, the above definition of subjet is consistent  with the classical notion of  subjet of l.s.c. functions. 
\begin{remark}
\rm
It can be checked that  the condition $(p,X)\in \Jet^{2,+}_E(x_0)$ is equivalent to $(\lambda  p, \lambda  X+\mu\, p \otimes p)\in \Jet^{2,+}_E(x_0)$ for all $\lambda>0$ and for all $\mu\in \R$. Thus,  $(p,X)\in \Jet^{2,+}_E(x_0)$ if and only if $(\frac{p}{|p|}, \frac{1}{|p|}\pi_{p^\perp}  X \pi_{p^\perp})\in \Jet^{2,+}_E(x_0)$, where  $\pi_{p^\perp}$ denotes the projection operator on $p^\perp:= \{ v\in\R^N :  p\cdot  v = 0\}$. Note that if $E=\{u\ge u(x_0)\}$, with $u$ of class $C^2$ and $D u\neq 0$ on $\partial E$, then, setting $p:=D u(x_0)$ and $X:=D^2 u(x_0)$, $ \frac{p}{|p|}$ is the inner normal to 
$\{u\ge u(x_0)\}$ at $x_0$, while $\frac{1}{|p|}\pi_{p^\perp}  X \pi_{p^\perp}$ represents the the second fundamental form of $\partial \{u\ge u(x_0)\}$ (oriented with the external normal) at $x_0$.
\end{remark}

Let us  introduce the notion of uniform superjet. 
\begin{definition}\label{defuniconv0}
Let $E_n\subseteq \R^N$ and $x_0\in\partial E_n$. We say that the $(p_n,X_n)$'s are in the superjet $\Jet^{2,+}_{E_n}(x_0)$ uniformly, if 
for every positive $\delta >0$ there exists a neighborhood $U_\delta$ of $x_0$ (independent of $n$) such that, { for all $n\in N$,} 
\begin{equation}
  (x-x_0)\cdot p_n + \frac12 (X_n+\delta  I )  (x-x_0)\cdot (x-x_0)  \ge 0 \text{ for every }  x\in E_n \cap U_\delta.
\end{equation}
\end{definition}
{In the following, given a set $E\subset \R^N$, $E^c:=\R^N\setminus E$ denotes its complement.} 
We also recall that a sequence of closed sets $C_n$ converges to a closed set $C$ in the Hausdorff metric ($C_n\kto C$) if 
$$
\max \left\{ \sup_{x\in C_n} \text{dist}(x,C), \sup_{x\in C} \text{dist}(x,C_n)\right\}\to 0 \qquad \text{ as } n\to\infty.
$$

\begin{definition}\label{defuniconv}
We say that $(p_n,X_n,E_n)$ converge to $(p,X,E)$  with uniform superjet at $x_0$ if 
$\ov E_n\to \ov E$ in the Hausdorff sense,
the $(p_n,X_n)$'s are in the superjet $\Jet^{2,+}_{E_n}(x_0)$ uniformly and $(p_n,X_n)\to (p, X)$ as $n\to\infty$. 

{Moreover, we say that $(p_n,X_n,E_n)$ converge to $(p,X,E)$  with uniform subjet at $x_0$ if  $(-p_n,-X_n,E^c_n)$ converge to $(-p,-X,E^c)$  with uniform superjet.}
\end{definition}

\subsection{Semi continuous extensions of $\kappa$}

We now introduce two suitable lower and upper semicontinuous extensions of $\kappa$, which will be instrumental in developing the level set formulation of the geometric flow. This is reminiscent of the approach in~\cite{C-NLI,C-NLII} for
evolution of ``tubes'' by geometric motions
(see also~\cite{CL07,CL08,CardaliaguetRouy}).
For every $F\subseteq \R^N$ with compact boundary and $(p,X)\in \Jet^{2,+}_{ F}(x)$, we define 
\begin{equation}\label{defkappal}
\kappa_*(x,p,X,F) :=\ \sup\left\{
\kappa(x,E)\,:\, E \in \Reg\,, E\supseteq F\,,
(p,X)\in \Jet^{2,-}_{{E}}(x)
\right\}
\end{equation}
Analogously, for any $(p,X)\in \Jet^{2,-}_{ F}(x)$ we set
\begin{equation}\label{defkappau}
\kappa^*(x,p,X,F)\ =\ \inf\left\{
\kappa(x,E)\,:\, E \in \Reg\,, \inter  E\subseteq  F\,,
(p,X)\in \Jet^{2,+}_{ {E}}(x)
\right\}.
\end{equation}
(It is clear that $\kappa_*$ only depends on the closure of $F$ while
$\kappa^*$ depends on its interior, in practice the first one will be
evaluated at superlevels of usc functions, while the second one at strict
superlevels of lsc functions.)

Clearly, it follows from the monotonicity property A) 
that if $E\in \Reg$, and $(p,X)\in \Jet^2_E(x)$, 
then $\kappa_*(x,p,X, {E})=\kappa^*(x,p,X, {E})
=\kappa(x,E)$.
Notice that the monotonicity of $\kappa$ 
clearly extends to $\kappa_*$ and $\kappa^*$.
More precisely, $\kappa_*(x,p,X, {E}) \ge \kappa_*(x,p,X, {F})$ (resp. $\kappa^*(x,p,X, {E}) \ge \kappa^*(x,p,X, {F})$) whenever
$E\subseteq F$ and $(p,X)\in \Jet^{2,+}_E(x)\cap \Jet^{2,+}_F(x)$ (resp. $(p,X)\in \Jet^{2,-}_E(x)\cap \Jet^{2,-}_F(x)$). 

In the next Lemma we show that $\kappa_*$ and $\kappa^*$ are the l.s.c. and the u.s.c envelope of $\kappa$ with respect to the convergence defined in 
Definition \ref{defuniconv}, respectively. 
\begin{lemma}\label{lemmalsc}
Let $F\subseteq \R^N$ with compact boundary. 
Then,
$$
\kappa_*(x,p,X,F)= \inf  \liminf_n \kappa(x,E_n)  
$$
where the infimum is over  all $(p_n,X_n,E_n)\to (p,X,F)$ with uniform superjet at $x$; 
$$
\kappa^*(x,p,X,F) = \sup \limsup_n \kappa(x,E_n) 
$$
where the supremum is over all {$(p_n,X_n, E_n)\to (p,X,F)$  with uniform subjet at $x$.}
\end{lemma}
\begin{proof}
We will prove the statement only for $\kappa_*$, the other case being  analogous. 
We start by showing that there exists a sequence $(p, X_n, E_n)$ such that $(p, X_n, E_n) \to (p, X, F)$ with uniform superjet at $x$  and $\kappa_*(x,p,X,F)= \lim_n \kappa(x,E_n)$. To this aim, recall that  
by definition of $\kappa_*(x,p,X,F)$, for every $n\in \N$
there exists a set $\tilde E_n\in\Reg$ with $F\subseteq \tilde E_n$, $(p,X)\in \Jet^{2,-}_{{\tilde E_n}}(x)$ and such that $ 0\le \kappa_*(x,p,X,F) - \kappa(x, \tilde E_n) \le \frac{1}{n}$.  
By the monotonicity of $\kappa$
we may  also assume that $\tilde E_n\to \ov F$ in the Hausdorff metric.   Moreover,  by the continuity assumption C) we can suitably modify the sequence $\tilde E_n$ so that, in addition to the previous properties, we have 
$(p,X +\delta_n I) \in \Jet^{2,-}_{{\tilde E_n}}(x)$, $\partial \tilde E_\e\cap \partial F = \{x\}$ and 
$ 0\le \kappa_*(x,p,X,F) - \kappa(x, \tilde E_n) \le \frac{2}{n}$ for some suitable $\delta_n\searrow 0$. 
Now we construct the sequence $E_n$ according to the following inductive procedure. Assume that $E_1, \ldots , E_n$ have been defined with the following properties:
\begin{itemize}
\item[1)] $ F \subseteq E_n \subseteq E_{n-1} \subseteq \ldots \subseteq E_1$,
\item[2)] $E_i\subseteq \tilde E_i$ for $i=1,\ldots n$,
\item[3)] $\partial  E_i \cap \partial F = \{x\}$ for every $i=1,\ldots n$
\item[4)] $ (p,X+\frac{\delta_i}{2} I) \in \Jet^2_{E_i}(x)  $ 
for every $i=1,\ldots n$.
\end{itemize}
 Since 
 $$ 
 X+\frac{\delta_{n+1}}{2} I <  X+ \frac{\delta_n}{2} I,
 \qquad
 X+\frac{\delta_{n+1}}{2} I <  X+ \delta_{n+1} I,
 $$
 recalling that  $ (p,X+\frac{\delta_n}{2} I) \in \Jet^2_{E_n}(x)  $ and $(p,X+\delta_{n+1} I) \in \Jet^{2,-}_{\tilde E_{n+1}}(x)$ we can easily construct $E_{n+1}\in \Reg$ such that 
 $$
(p,X+\frac{\delta_{n+1}}{2} I) \in \Jet^2_{E_{n+1}}(x), \,\,  F \subseteq E_{n+1}\subseteq \tilde E_{n+1}, \,\, E_{n+1}\subseteq E_n,\,\,  \partial  E_{n+1} \cap \partial F = \{x\}.  
 $$
The sequence $E_n$ just constructed still converges to $\ov F$ in the sense of Hausdorff. Moreover, since $E_n$ is  monotone decreasing, we have $(p,X+\frac{\delta_n}{2}I, E_n)\to (p,X,F)$ with uniform superjet.  Note also that since $(p, X)\in \Jet^{2,-}_{E_n}(x)$, by \eqref{defkappal} we immediately have
$\kappa_*(x, p, X, F)\geq \kappa(x, E_n)$ for all $n$. Finally, by the monotonicity of $\kappa$,
$$
\kappa_*(x, p, X, F)\geq\limsup_n \kappa(x,E_n)\ge\liminf_n \kappa(x,E_n)\ge
\lim_n \kappa(x,\tilde E_n) = \kappa_*(x,p,X,F).
$$
Now, let   $(p_n,X_n,F_n)\to (p,X,F)$ with uniform superjet  at $x$. 
Let $j\in\N$ and let $E_j$ the set constructed above. Notice that $F_n$ is definitively contained in $x+ R_n (E_j - x)$, where $R_n$ is any rotation such that $R_n(p)= p_n$. By  the monotonicity assumption A) 
 and the continuity assumption C) on $\kappa$ we deduce 
 \begin{equation}\label{interb}
 \kappa(x,E_j) = \lim_n\kappa(x,x+ R_n (E_j - x))  \le  \liminf_n \kappa(x,F_n).
 \end{equation}
We conclude that
 $$
 \liminf_n \kappa(x,F_n) \ge \lim_j \kappa(x,E_j) = \kappa_*(x,p,X,F).
 $$
 \end{proof}

\begin{lemma}\label{scif}
Let $\f_n, \f \in C^0(\R^N)$ be constant outside a compact set $\mathcal K$ (independent of $n$).  Assume that  $\f_n\to \f$ uniformly and 
$\f_n\to \f$ in $C^2(B(x,\delta))$ for some $\delta>0$ and $x\in\R^N$ with $D \f(x)\neq 0$.
If $x_n\to x$, then
\begin{multline}\label{eq:sciC2}
\kappa_*(x,D\f(x),D^2\f(x),\{\f\ge \f(x)\})
\\
\ \le\ \liminf_{n} \kappa_*(x_n,D\f_n(x_n),D^2\f_n(x_n),\{\f_n\ge \f_n(x_n)\})
\end{multline}
and
\begin{multline}\label{eq:scsC2}
\kappa^*(x,D\f(x),D^2\f(x),\{\f> \f(x)\})
\\
\ \ge\ \limsup_{n} \kappa^* (x_n,D\f_n(x_n),D^2\f_n(x_n),\{\f_n> \f_n(x_n)\}).
\end{multline}
\end{lemma}
\begin{proof}
Up to a subsequence, we can assume that $\mathcal K_n:=x-x_n +\{\f_n\ge \f_n(x_n)\}$ converge in the sense of Hausdorff to some closed set $\tilde {\mathcal K}$ contained in  $\mathcal K:= \{\f \ge \f(x)\}$. Since $D \f(x)\neq 0$, then  $(D \f_n(x_n), D^2 \f_n(x_n), \mathcal K_n) \to (D \f(x),D^2 \f(x), \tilde{\mathcal K})$ with uniform superjet at $x$.
Thus, by Lemma \ref{lemmalsc} and the monotonicity of $\kappa_*$ 
we may conclude
\begin{multline*}
\kappa_*(x,D \f(x),D^2 \f(x),  \mathcal K)\leq  \kappa_*(D \f(x),D^2 \f(x), \tilde {\mathcal K})\\
\le \liminf_n \kappa_*(x,D \f_n(x_n), D^2 \f_n(x_n),  \mathcal K_n), 
\end{multline*}
which proves  \eqref{eq:sciC2}. The proof of \eqref{eq:scsC2} is identical. 
\end{proof}
\begin{remark}\label{mds}
{\rm 
In \cite{S},  a class of nonlocal Hamiltonians  $H(x,p,X,F)$ that are lower semicontinuous with respect to the $L^1$ convergence of sets and the standard convergence of the other variables is considered. Notice that such Hamiltonians are also lower semicontinuous  with respect to the convergence with uniform superjets introduced in Definition \ref{defuniconv}. Indeed, 
 if  $(p_n,X_n,F_n)\to (p,X,F)$ with uniform superjet  at $x$, then one can show that
  { $F_n\cup 
 \overline F \to \overline F$} in $L^1_{loc}$. Thus,
$$
{H(x,p,X,\overline F)\le \liminf_n H (x,p_n,X_n,F_n\cup \overline F) \le \liminf_n H (x,p_n,X_n,F_n) ,} 
$$

where the last inequality follows from the { monotonicity assumption on $H$ with respect to the set variable.}
}
\end{remark}

\subsection{Equivalent definition of the viscosity solutions}

We now can give a second definition of viscosity solutions of~\eqref{levelsetf}. It is seemingly more restrictive than the previous Definition~\ref{defviscoC2},
but we will check later on that it is equivalent. 
{
\begin{definition}\label{defvisco}
An upper semicontinuous function $u:\R^N\times [0,T]\to \R$,
constant outside a compact set, is
a viscosity subsolution of the Cauchy problem \eqref{levelsetf} if $u(0,\cdot) \le u_0$, and for all 
 $z:= (x,t) \in \R^N\times (0,T)$ and
all   $\f$ admissible at $z$, such that $u-\f$ has a
maximum at $z$ (in the domain of definition of $\f$) we have 
\begin{itemize}
\item[i)] If   $D\f(z) =  0$, then $\f_t(z) \le 0$;
\item[ii)] If  $D\f(z)\neq 0$, then
\begin{equation}\label{eqsubsol}
\f_t(z)+ |D\f(z)| \, \kappa_*\left(x,D\f(z),D^2\f(z),\{y: \f(y,t)\ge \f(z)\}\right)
\le 0.
\end{equation}
\end{itemize}
The definition of viscosity supersolutions and of viscosity solutions are given accordingly in the obvious way. 
 \end{definition}
}


{
We will need  the following technical lemma.}
\begin{lemma}\label{Qlemma}
 Let {$Q\in C^{\infty}(\R^N)$}, $Q\ge 0 $ with equality only for $x=0$,
$Q$  convex in $B_1$,  and constant in $\R^N\setminus B_2$. 

Let $\f\in C^2(\R^N)$, and let $\bar x$ be such that $D \f(\bar x) \neq 0$. 
 For every $\eta\in\R$ set
$$
\f^\eta(x):= \f(x)+\eta Q(x-\bar x).
$$
Then, for almost every $\eta$ small enough  the $\f(\bar x)$-level set of $\f^\eta$ is not critical.
\end{lemma}

\begin{proof}
Let $B(\bar x,\delta)$ be a neighborhood of $\bar x$ where $D \f  \neq 0$.
Clearly in $B(\bar x,\delta)$ $D\f+\eta DQ(x-\bar x)\neq 0$ if $\eta$
is small enough.

Then, in $\R^N\setminus \overline{B(\bar x,\delta/2)}$,
we consider the {$C^\infty$} function 
$$
x \mapsto - \frac{ \f(x) - \f(\bar x)}{Q(x-\bar x)},
$$
and by
Sard's theorem, we know that for a.e. $\eta>0$, the level set
$\eta$ of this function is not critical. This means that for
all~$x\not\in \overline{B(\bar x,\delta/2)}$ with $\f(x)+\eta Q(x-\bar x)= \f(\bar x)$, one has
\begin{multline*}
0\,\neq\,\frac{1}{Q^2(x-\bar x)}\left( -D \f(x)Q(x-\bar x) + D Q(x-\bar x) (\f(x) -  \f(\bar x))   \right)
\\=\, -\frac{1}{ Q(x-\bar x)} \left( D \f(x)+\eta D Q(x-\bar x) \right)\,,
\end{multline*}
so that the $\f(\bar x)$-level set  of $\f^\eta$ is not
critical.
\end{proof}

\begin{remark}\label{convo}
 As it is standard in the theory of viscosity solutions, the  maximum in Definition \ref{defvisco} of  subsolutions can be assumed to be strict (and similarly for supersolutions).
Assume for instance that $u$ is a subsolution, 
{
$u-\f$ has a  maximum
}
 at some $(\bar x, \bar t)$, with $\f$ admissible at $(\bar x,\bar t)$. If $D \f(\bar x, \bar t)\neq 0$ we replace $\f$ by
$$
\f^s(x,t):= \f(x,t)+s  \min \{1, | x-\bar x|^2  +|t-\bar t|^2\}.
$$
Then the maximum of $u-\f^s$ at $(\bar x, \bar t)$ is strict, and we recover the inequality \eqref{eqsubsol} for $\f$ by letting $s \to 0$ and using the semicontinuity property of $\kappa_*$ provided by Lemma~\ref{scif}.
If $D \f(\bar x, \bar t) = 0$,  we choose {$f \in \F$} as in Definition~\ref{defadmissible},
and replace $\f$ by 
{
$$
\tilde \f(x,t):= \f(x,t) +f(|y-x|) +   |t-s|^2.
$$ 
}
We still have $D \tilde\f(\bar x, \bar t) = 0$,  $\tilde \f$ is  admissible at $( \bar x, \bar t)$, 
$\tilde\f_t( \bar x, \bar t)=\f_t(\bar x,\bar t)$ and  now the 
{ maximum 
}
of $u-\tilde \f$ is strict.

Moreover,  one can assume without loss of generality that $\f$ is smooth. 
{
If $D \f(\bar x, \bar t)\neq 0$, this follows again by Lemma~\ref{scif} and by standard mollification arguments.
If $D \f(\bar x, \bar t) = 0$,
since $\f$ is admissible at $z$, there are  $f \in \F$ and $\omega \in C^\infty(\R)$ with $\omega'(0)= 0$ such that
$$
|\f(x,t) - \f(\bar x, \bar t) - \f_t(x,t)(t-\bar t)|\le f(|x- \bar x|) + \omega(t-\bar t).
$$ 
Then it is enough to replace $\f$ by 
$$
  \psi(x,t):=\f_t(z)(t-\bar t)+ f(|x-\bar x|)+\omega (t-\bar t). 
$$
Finally, in view of Lemma \ref{Qlemma} one can assume that the superlevel set of $\f$ in  Definition \ref{defvisco},  ii)  is not critical. 
}
\end{remark}
\begin{remark}\label{rm:equivdef}
Note that in view of the above observations we have shown, in particular, that Definitions~\ref{defviscoC2} and~\ref{defvisco} are equivalent.
\end{remark}


We now introduce the notion of parabolic sub/superjets. 
\begin{definition}\label{def:parajet}
Let $u:\R^N \times (0, T) \to \R$ be upper semicontinuous
  at $(x,t)$. We say that $(a,p,X)\in \R\times\R^N\times\mathbb{M}_{sym}^{N\times N}$ is in the parabolic superjet 
$ \PJet^{2,+}u(x,t)$ of $u$ at $(x,t)$,  if  
$$
u( y,s)\leq u(x,t) +  a(s-t)\,+\, p\cdot(y-x)\,+\, \half(X(y-x))\cdot(y-x)
\,+\,o(|t-s|+|x-y|^2)
$$
 for $(y,s)$ in a neighborhood of $(x,t)$.  If $u$ is lower semicontinuous at $(x,t)$ we can define the parabolic subjet
  $ \PJet^{2,-}u(x,t)$ of $u$ at $(x,t)$ as  $ \PJet^{2,-}u(x,t):=- \PJet^{2,+}(-u)(x,t)$.
\end{definition}

{
The next lemma 
provides another  equivalent definition of viscosity solutions
in terms of the superlevel sets of $u$ and the corresponding parabolic jets.
}

\begin{lemma}\label{laura}
Let $u$ be a viscosity subsolution of~\eqref{levelsetf} in the sense
of { Definition~\ref{defvisco}. }
Then, for all $(x,t)$ in $\R^N\times (0,T)$, if $(a,p,X)\in \PJet^{2,+}u(x,t)$,
and $p\neq 0$, then
\begin{equation}\label{eqsubsolslepcev}
a\,+\, |p|\kappa_*\left(x,p,X,\{y: u(y,t)\ge u(x,t)\}\right) \ \le\ 0.
\end{equation}
A similar statement holds for supersolutions.
\end{lemma}

\begin{proof}[Proof of Lemma~\ref{laura}]
 By definition of parabolic subjets,  given $\e>0$, $\delta>0$,
there exists a neighborhood $U$ of $(x,t)$ in $\R^N\times (0,t]$ where
\[
u(y,s)\,\le\, u(x,t) +  (a-\e)(s-t)\,+\, p\cdot(y-x)\,+\,
\half(X+\delta I)(y-x)\cdot(y-x)
\]
with a strict inequality if $y\neq x$ or $s<t$. Let $p_{\e,\delta}:\R^N\times [0,T]$ be  a continuous function such that
$$
p_{\e,\delta}(y,s)=u(x,t) + (a-\e)(s-t)+ p\cdot(y-x)+\half(X+\delta I)(y-x))\cdot(y-x)
$$
in  $U$, $p_{\e,\delta}\ge u$ in $\R^N\times (0,t]$, with equality only on $(x,t)$, 
$p_{\e,\delta}\ge u + c$ in $(\R^N\times (0,t])\setminus U$ for some  $c>0$, and $p_{\e, \delta}$ is   constant {(possibly depending on time) in $(\R^N\setminus \mathcal K)$, where also $u$ is constant.} 
Consider a decreasing sequence $\psi^k$ of smooth functions, such that $\psi^k$ is  constant in { $(\R^N\setminus \mathcal K)$}, 
$\inf_k \psi^k=u$, and  $\psi^k\ge u+1/k$. Such a sequence exists
because $u$ is upper-semicontinuous. 
{ Let $\f^k:= \min\{\psi^k,p_{\e,\delta}\}$, so that
$\f^k> u$ in $\R^N\times (0,t]$,
except at $(x,t)$ where equality holds, and 
$\f^k=p_{\e, \delta}$ near $(x,t)$.}

For any $n\in \N$ large enough,
the function $(y,s)\mapsto u(y,s) - \f^k(y,s) - 1/[n(t-s)]$
attains a  maximum at a point $z_n = (y_n,s_n) \in (0,t)\times \R^N$, where $z_n\to z=(x,t)$
as $n\to \infty$. Moreover, $D\f^k(z_n)=p+(X+\delta I)(y_n-x)\neq 0$
for $n$ large.

{Hence, by Definition~\ref{defvisco} of a viscosity subsolution,
}
\begin{multline*}
\f^k_t(z_n)\, +\,  \frac{1}{n(t-s_n)^2}
\\
+ |D\f^k(z_n)| \, \kappa_* (y_n, D\f^k(z_n), D^2\f^k(z_n), \{\f^k(\cdot,s_n)\ge \f^k(z_n)\}) \,\le\, 0. 
\end{multline*}
Since $\partial_t \f^k_t(z_n) = a-\e$ it follows that
\[
a\,+\, |D\f^k(z_n)| \, \kappa_* (y_n, D\f^k(y_n,s_n),D^2\f^k(y_n,s_n), \{\f^k(\cdot,s_n)\ge \f^k(y_n,s_n)\}) 
\,\le\,\e.
\]
Letting $n\to\infty$ and invoking Lemma~\ref{scif} 
we obtain
\[
a\,+\,  |D\f^k(x,t)| \, \kappa_*(x, D\f^k(x,t),D^2\f^k(x,t), \{\f^k(\cdot,t)\ge \f^k(x,t)\}) 
\,\le\,\e\,,
\]
that is
\[
a\,+\, |p|\kappa_*(x, p,X+\delta I, \{\f^k(\cdot,t)\ge \f^k(x,t)\}) 
\,\le\,\e.
\]
Now, as $\{\f^k(\cdot,t)\ge \f^k(x,t)\}$ is a decreasing sequence converging to $\{u(\cdot,t)\ge u(x,t)\})$, we get that  $(p,X+\delta I, \{\f^k(\cdot,t)\ge \f^k(x,t)\})\to (p,X+\delta I, \{u(\cdot,t)\ge u(x,t)\})$ with uniform superjet, as $k\to\infty$. Therefore,  by Lemma~\ref{lemmalsc} we infer
\[
a\,+\, |p|\kappa_*(x, p,X+\delta I, \{u(\cdot,t)\ge u(x,t)\})
\,\le\,\e.
\]
The conclusion follows by applying again  Lemma~\ref{lemmalsc}, after observing that  
$$
( p,X+\delta I, \{u(\cdot,t)\ge u(x,t)\})\to ( p,X, \{u(\cdot,t)\ge u(x,t)\})$$
 with uniform superjet, as $\delta\to 0$.
\end{proof}

In the next lemma we show that  equation \eqref{levelsetf} is satisfied in a suitable viscosity sense also  for $t=T$. To this purpose, we notice that the notion of admissible test functions $\f$ given in Definition \ref{defadmissible} can be extended also at points $(\hat x,T)\in \R^N\times\{T\}$ without any change. 
This is a classical fact, we adapt here the proof in~\cite{IS}.

\begin{lemma}\label{Tfinale} Let $u\in USC(\RT)$ be a subsolution of \eqref{levelsetf}. 
If $\f$ is admissible at  $(\hat x,T)$ 
and $u-\f$ has a (one-sided w.r.t. time) maximum in $\RT$ at $(\hat x,T)$, then  { i) and ii) of Definition \ref{defvisco} are satisfied}  at $(\hat x,T)$. An analogous statement holds for supersolutions.
\end{lemma}
\begin{proof} 
First assume that $D\f(\hat x,T) \neq 0$. As usual, we can assume that the maximum is strict. For  $n$ large enough, the function 
$u(x,t) - \f(x,t) - 1/[n(T-t)]$ has a  maximum at a point $z_n:=(\hat x_n,t_n) \in \QT$ converging to $z:=(\hat x,T)$ as $n\to\infty$. 
Since $u$ is a subsolution in $\QT$, for $n$ large enough we have
$$
\f_t(z_n)+ \frac{1}{n(T-t_n)^2} + |D\f(z_n)| \, \kappa_*\left(x_n,D\f(z_n),D^2\f(z_n),\{\f(\cdot,t_n)\ge \f(z_n)\}\right)
\le 0. 
$$
Letting $n\to\infty$, the conclusion follows from Lemma \ref{scif}.

If now  $D\f(z) = 0$, we follow the lines of \cite[Proposition 1.3]{IS}. 
Since $\f$ is admissible at $z$, there are {$f \in \F$} and {$\omega \in C^\infty(\R)$ with $\omega'(0)= 0$} such that
$$
|\f(x,t) - \f(z) - \f_t(z)(t-T)|\le f(|x- \hat x|) + \omega(t-T).
$$ 
{ Set 
\begin{align}
& \psi(x,t)=\f_t(z)(t-T)+ 2f(|x-\hat x|)+2\omega(t-T)\,,\nonumber\\
 & \psi_n(x,t)=\psi(x,t) -\frac{1}{n(T-t)}\,. \label{psienne} 
\end{align}
}
We have that $ u-\psi$ has a strict maximum at $z$.
Hence for $n$ large enough $ u-\psi_n$ has a { maximum }
at $z_n=(\hat x_n,t_n) \in \QT$, with $z_n\to z$. 
As $\psi_n$ is admissible at $z_n$ and $u$ is  a subsolution, we have
{
\begin{multline}\label{boh}
\f_t(z)+\omega'(t_n-T)
\\
+ { 2 } f'(|\hat x_n-x|)\kappa_*(\hat x_n,  D \psi_{n}(\hat x_n),  D^2 \psi_{n}(\hat x_n), \{\psi_n(\cdot, t_n)\geq \psi_{n}(z_n)\})\leq 0
\end{multline}
if $\hat x_n\neq x$, while $\f_t(z)+\omega'(t_n-T)\leq 0$ if $\hat x_n=x$. 
Note that $\{\psi_n(\cdot, t_n)\geq \psi_{n}(z_n)\}=\R^N\setminus B_{f^{-1}(|\hat x_n-x|)}(x)$. 
Letting $n\to \infty$, we get $\f_t(z)\leq 0$ thanks to \eqref{ISvera}.
Hence, as claimed, $u$ is a subsolution. 
}
\end{proof}
\begin{remark}\label{rm:laura}
A similar reasoning also shows that the alternative characterization of sub- and super-solutions provided by Lemma~\ref{laura} holds also at 
points of the form $(x, T)$.
\end{remark}
\subsection{Existence of a viscosity solution}
Let $u_0:\R^N\mapsto \R$ be a continuous function, constant out of a compact set $\mathcal K$.  The  existence of a viscosity solution
to the Cauchy problem~\eqref{levelsetf} follows by standard arguments once the existence of at least one supersolution and a stability property for supersolutions are established. To this purpose, we first prove a confinement condition.

\begin{lemma}\label{barriera}
Let $R,T>0$ be fixed.  
There exists a constant $R' >R$ such that if
 $u\in USC( \RT)$ is a  subsolution  of \eqref{levelsetf}  with
 $u(x,0)\le C_0$ for $|x|\ge R$, then
 $$
u(x,t) \le  C_0 \qquad \text{for }|x|>R'\text{ and }t\in [0, T]\,.
$$
\end{lemma}
\begin{proof}
Let {$\psi\in C^\infty([0,+\infty))$} be such that $\psi'(0)=0$, $\psi(s)\equiv C_0$ for $s\geq 2R$,  $\psi$ strictly decreasing in $[0, 2R]$, and $\psi(|x|)\geq u(x, 0)$ for all $x\in \R^N$. We now construct a test function $\f$, by letting all the superlevel sets of $\psi(|\cdot|)$ expand with constant normal velocity equal to $K$, where $K$ is the constant appearing in \eqref{lwrbdkappa}. Precisely, we set
\begin{equation}\label{esesub}
\f(x,t):=\begin{cases}
\psi(|x|-Kt) & \text{if $|x|\leq 2R+Kt$,}\\
C_0 & \text{otherwise.}
\end{cases}
\end{equation}
The lemma is proven (with $R':= 2R+KT$) once we show that $\f\geq u$. Assume toward a contradiction that there exists $\delta>0$ and $(x,t)\in \QT$ such that $(u-\f )(x,t)> \delta t$. Then,  setting $\f_\delta(x,t):=\f(x,t)+\delta t$, we have
$\max_{\RT} u-\f_\delta>0$. Let $z_\delta:= (x_\delta, t_\delta)$ be a maximum point and note that necessarily $t_\de>0$.
If $D \f_\delta(z_\delta) = 0$, recalling the definition of a subsolution we get the contradiction $0\le \delta = (\f_\delta)_t (z_\delta)\le 0$. 
If $D \f_\delta(z_\delta) \neq  0$, recalling again Definition~\ref{defvisco} of subsolution we get the contradiction
\begin{multline*}
\delta = \f_t(z_\delta)+\delta -  K|D\f(z_\delta)|  
\\ <  \f_t(z_\delta)+\delta 
+ |D\f(z_\delta)|  \kappa_*\left(x_\delta,D\f(z_\delta),D^2\f(z_\delta),
\{\f(\cdot,t_\delta)\ge \f(z_\delta)\}\right) \le 0.
\end{multline*}
\end{proof}

It is very easy to show that, as in the classical case, the maximum of two subsolutions  is still a subsolution.  In the following we show that the notion of subsolution is stable also with respect to taking upper relaxed limits.
\begin{proposition}\label{propstab}
Let $(u_n)_{n\ge 1}$ be a sequence of viscosity subsolutions such that $u_n = c_n$ in $(\R^N\setminus \mathcal K) \times [0,T)$, for some constant $c_n\in \R$ and some compact $\mathcal K\subseteq \R^N$.
Let, for any $z=(x,t)$,
\begin{equation}\label{uceu}
u^*(z)\ =\ \lim_{r\downarrow 0}\,\, \sup\left\{ 
u_n(\zeta)\,:\, |z-\zeta|\le r\,, n\ge \frac{1}{r}\right\}.
\end{equation}
If $u^*(z)<+\infty$ for all $z$, then $u^*$ is a subsolution.
\end{proposition}

Of course, a symmetric result holds for supersolutions.
\begin{proof}
Let $A\subset (0,T)$ be an open interval and let $\f:\R^N\times \ov A\to \R$ be an admissible  test function at $z=(x,t)$ with 
$\f(\cdot, s) = C(s)$ in $(\R^N\setminus \tilde {\mathcal K})\times \ov A$ for some compact set $\tilde {\mathcal K}$ and for all $s\in \ov A$, and  such that $u-\f$ has a strict maximum at $z$. Assume first that  $D\f(z)\neq 0$.  
Let $z_n$ be a  maximum point of $u_n-\f$ in { $(\mathcal K\cup \tilde {\mathcal K}) \times \bar A$}. By standard arguments  it follows that $z_n\to z$. Since for every $n$
$$
\f_t(z_n)+ |D\f(z_n)| \, \kappa_*\left(x,D\f(z_n),D^2\f(z_n),\{y: \f(y,t)\ge \f(z_n)\}\right)\le 0,
$$
by Lemma~\ref{scif} we conclude that
$$
\f_t(z)+ |D\f(z)| \, \kappa_*\left(x,D\f(z),D^2\f(z),\{y: \f(y,t)\ge \f(z)\}\right)\le 0.
$$

If now  $D\f(z) = 0$, we follow the lines of \cite[Proposition 1.3]{IS}. 
Since $\f$ is admissible at $z$, there are $\delta>0$, { $f \in \F $  and $\omega \in C^\infty(\R)$ with $\omega'(0)=  0$}, $\omega(t)>0$ for $t>0$  such that
$$
|\f(y,s) - \f(z) - \f_t(z)(s-t)|\le f(|x-y|) + \omega(s-t)
$$ 
for all $(y,s)\in \R^N\times A$. 
{ Set 
\begin{equation}\label{psienne2}
 \psi(y,s)=  \f_t(z)(s-t)+ 2f(|y-x|)+2\omega(s-t). 
\end{equation}
}
Note that  $u-\psi$ has a unique maximum at $z$ in $\R^N\times  A$. 
Let $z_n=(y_n,  s_n)$ be a maximum point of 
$u_n-\psi$ in $\R^N\times \bar A$. Then $z_n\to z$. 
If $D\psi(z_n)=0$, then $y_n=x$ and $\psi_n$ is admissible at $z_n$
thanks to~\eqref{psienne2}.
We deduce $\f_t(z)+\omega' (s_n-t)\leq 0$. Otherwise $y_n\neq x$,
$\psi$ is still admissible at $z_n$ and we have
\begin{multline}
\f_t(z)+\omega'(s_n-t)
\\
+ {2} f'(|y_n-x|)\kappa_*(y_n,  D \psi_{n}(y_n),  D^2 \psi(y_n), \{\psi(\cdot, s_n)\geq \psi(z_n)\})\leq 0.
\end{multline}
Note that $\{\psi(\cdot, s_n)\geq \psi(z_n)\}=\R^N\setminus B_{f^{-1}(|y_n-x|)}(x)$. 
Letting $n\to \infty$, we get $\f_t(z)\leq 0$ thanks to \eqref{ISvera}.
Hence, as claimed, $u$ is a subsolution. 
\end{proof}

We now can state a general existence  result:
\begin{theorem}\label{thexun}
Let $u_0:\R^N\to \R$ be a uniformly continuous function with $u_0 = C_0$ for $|x|\ge R$. Let $R'$ be the constant  given by Lemma \ref{barriera}.  
Then, there exists a   viscosity 
solution $u:\R^N\times[0,T]\to \R$ of~(\ref{levelsetf})  with $u=C_0$ for $|x|\ge R'$.
\end{theorem}

\begin{proof}
We briefly sketch  the proof of this result which is  very classical, see~\cite{CIL,IS}, and
based on Perron's method. 
Let $\f$ be the function defined in \eqref{esesub}, and notice that it is a supersolution. 
Then, we can set
\begin{equation*}
 u(x,t) = \inf \{ u(x,t):\, u \textrm{ is a supersolution of the Cauchy problem \eqref{levelsetf}} \}.
\end{equation*}
The fact that $u$ is bounded from below follows easily by using a
smooth barrier as in Lemma~\ref{barriera}. {For instance,  the
barrier $C_0 + \lambda(C_0-\f)$ will work for $\lambda$ so large that $C_0+\lambda(C_0-\psi(|\cdot |))\leq u(\cdot,0)$}, also showing that that   $u=C_0$ for $|x|\ge R'$
(this is also is a consequence of the fact that $u(\cdot,0) = u_0$, which
we explain at the end of this proof). 
Let $u^*$, $u_*$ be the upper and lower  semicontinuous envelope of $u$.
The fact that $u_*$
is a supersolution   follows from Proposition~\ref{propstab},
observing that at each point $(x,t)$ we can find a suitable sequence
of supersolutions $(u_n)_{n\ge 1}$, constant on $(\R^N\setminus B_{R'})\times[0,T)$ (see Lemma \ref{barriera}), whose relaxed lower limit is $u_*(x,t)$.

The fact that $u^*$ is  a subsolution is standard and obtained
by contradiction, assuming that at some point $\bar z=(\bar x,\bar t)$
of (strict) contact with a test function $\f\ge u^*$, $\f$ does not
satisfy~\eqref{eqsubsol}. 
If $D\f(\bar z)\neq 0$, one can use the test function $\f$ to construct a new supersolution $ u< u_*$ in a neighborhood of $\bar z$, thus contradicting the minimality of $\bar u$. 
To treat the case $D\f(\bar z)=0$ one repeats the same construction, but  (as  in the proof of \cite[Prop. 1.3]{IS}) with $\f$ replaced by
\[
\psi(x,t)=\f(\bar z)
+ \f_t(\bar z)(t-\bar t) + 2f(|x-\bar x|) + 2\omega(t-\bar t).
\]

Finally, the  initial condition $u(\cdot,0) = u_0$  
can be shown as in the last part of the proof of \cite[Theorem 1.8]{IS}.
\end{proof}

\section{Uniqueness of viscosity solutions}\label{sec:uniqueness}
In this section we will prove that, under some additional assumptions, \eqref{levelsetf} admits a unique viscosity solution. 
In the first subsection, we consider {\em first order} geometric flows, corresponding to the  case where the relaxed curvatures $\kappa_*$ and $\kappa^*$ depend only on the first order super-jet and sub-jet, respectively.
Examples of relevant first-order flows are given in Section~\ref{sec:examples}.

 In the second subsection, we deal with truly second order flows, under  an additional  uniform continuity assumption on the nonlocal curvature $\kappa$. 

Before entering the details of the uniqueness theory, it is convenient to give the following definition and state an auxiliary lemma.

\begin{definition}\label{setf}
We say that a set valued function $F: [0,T]\to \Ins$ is a subsolution of the geometric flow \eqref{oee} if $\chi_{F(t)}$ is a viscosity subsolution of \eqref{levelsetf} in the sense of Definition \ref{defvisco}. 
The definition of  supersolutions and solutions of the geometric flow are analogous. 
\end{definition}

\begin{lemma}\label{L1}
Let $u$ be a subsolution of \eqref{levelsetf}. Then, for every $s\in\R$ the set function $t \to F(t):=\{u(\cdot,t) \ge s\}$ is a subsolution of the geometric flow \eqref{oee}, according with Definition \ref{setf}. 
The analogous statement holds for supersolutions.
\end{lemma}
The proof is classical and follows
 by approximating the Heavyside function as a supremum
of smooth increasing functions $H_n$, {so that}
$\chi_F(x,t)=\sup_n H_n(u(x,t)-s)$.

 \subsection{Uniqueness for first-order flows}\label{subsec:1storder}

Here we consider the case of curvatures, which generate  a first order flow.  
Namely, we denote by $\Regm$ the class of sets of $\R^N$, which are the closure of an open set of class $C^{1,1}$ with compact boundary, and  we  assume that the following property holds:
\begin{itemize}
\item[(FO):] Let $\Sigma\in \Regm$, let  $x\in \pa \Sigma$, and let $(p,X)$ and $(p,Y)$ be elements
of $\Jet^{2,+}_\Sigma(x)$ and $\Jet^{2,-}_\Sigma(x)$, respectively. 
Then,
\begin{equation}\label{eq:FO}
\kappa_*(x, p, X, \Sigma)=\kappa^*(x, p, Y, \Sigma)\,.
\end{equation}
\end{itemize}
Note that the above assumption postulates that the semicontinuous extensions $\kappa_*$ and $\kappa^*$ are  independent of the second derivative variables $X$ and $Y$, at least on $C^{1,1}$-sets. Under these circumstances, we may regard the common value of the quantities in \eqref{eq:FO} as an extension of the definition of  curvature to sets of class $C^{1,1}$; i.e., for all $\Sigma\in \Regm$ 
 we may set 
 \begin{multline*}
\kappa(x, \Sigma):=\kappa_*(x, p, X, \Sigma)=\kappa^*(x, p, Y, \Sigma)\\
 \text{for any $(p,X)\in \Jet^{2,+}_\Sigma(x)$ and $(p,Y)\in \Jet^{2,-}_\Sigma(x)$.}
\end{multline*}

In this situation, the problem becomes similar to the methodology developed
in~\cite{C-NLII} by P.~Cardaliaguet. In particular, as show
the following Lemma~\ref{lm:kstarFO}, our extensions
$\kappa_*$, $\kappa^*$ correspond here precisely to the extensions
$h^\sharp$, $h^\flat$ found in eqn~(6), (7) of~\cite{C-NLII} (with $h=-\kappa$).
For completeness, and also because of slight differences (in particular,
we have no sign restriction on our curvatures), we present here complete
proofs of uniqueness, which rely as in~\cite{C-NLII} on Ilmanen's
interposition lemma.

\begin{lemma}\label{lm:kstarFO}
Assume that {\rm(FO)} holds and let $F\in\Ins$. Then, for any  $(p,X)\in \Jet^{2,+}_F(x)$ we have
\begin{multline}\label{k_*FO}
\kappa_*(x, p, X, F)=\sup\bigl\{\kappa(x, \Sigma): \Sigma\in \Regm, \\
F\subset\Sigma,\, x\in \pa \Sigma, 
\text{ and } p\perp \pa\Sigma\text{ at }x \bigr\}\,.
\end{multline}
Analogously, for any  $(p,Y)\in \Jet^{2,-}_F(x)$ we have
\begin{multline}\label{k^*FO}
\kappa^*(x, p, Y, F)=\inf\bigl\{\kappa(x, \Sigma): \Sigma\in \Regm, \\
\inter\Sigma\subset F,\, x\in \pa \Sigma, 
\text{ and } p\perp \pa\Sigma\text{ at }x \bigr\}\,.
\end{multline}
\end{lemma}
\begin{proof}
We only prove \eqref{k_*FO}, the proof of \eqref{k^*FO} being analogous. Denote by $\underline\kappa (x, p, F)$ the quantity defined by
the right-hand side of \eqref{k_*FO}. Clearly, by definition of $\kappa_*$ we immediately have that $\kappa_*(x, p, X, F)\leq \underline\kappa (x, p, F)$. 

To prove the opposite inequality, fix $\e>0$ and let $\Sigma$ be a $C^{1,1}$-set, admissible for the definition of $\underline\kappa (x, p, F)$, such that 
\begin{equation}\label{FO1}
\kappa(x, \Sigma)\geq \underline\kappa (x, p, F)-\e\,.
\end{equation}
Moreover, let {$A\in \Reg$},  
admissible for the definition of $\kappa_* (x, p, X, \Sigma)= \kappa (x,  \Sigma)$, such that 
\begin{equation}\label{FO2}
\kappa(x, A)\geq \kappa (x,  \Sigma)-\e\,.
\end{equation}
Since $A$ is also admissible for $F$, we have
\begin{equation}\label{FO3}
\kappa_*(x, p, X, F) \ge \kappa(x, A)\geq \kappa (x, \Sigma)-\e \ge  \underline\kappa (x, p, F)-2\e,
\end{equation}
and the conclusion follows from the arbitrariness of $\e$. 
\end{proof}

We continue with the following lemma, which provides a crucial comparison property between $\kappa_*$ and $\kappa^*$. 

\begin{lemma}\label{lm:mono1*}
Assume that condition {\rm(FO)} above  holds and let  $F$, $G$ be a closed and an open set, respectively,  with compact boundaries   and such that $F\subset G$.  Let  $x\in \pa F$,  $y\in \pa G$ satisfy
\begin{equation}\label{mindist}
|x-y|=\dist(\pa F, \pa G)\,.
\end{equation}
Then,  for all $(p,X) \in \Jet^{2,+}_F(x)$ and $(p,Y) \in \Jet^{2,-}_G(y)$, with $p:=x-y$, we have  
$$
\kappa_*(x,p,X,F) \ge \kappa^*(y,p,Y,G)\,. 
$$
\end{lemma}
\begin{proof}
Exploiting \eqref{mindist},  we may apply the Ilmanen Interposition Lemma (\cite{IL}), to find 
a  set $\Sigma\in \Regm$  such that $F\subset \Sigma$,  $\overline \Sigma\subset G$, and $[x, y]\cap\pa\Sigma=\{\hat z\}$,
with $\hat z$ satisfying
$$
|x-\hat z|=\dist(\pa F, \pa\Sigma)= |y-\hat z|=\dist(\pa G, \pa\Sigma)\,.
$$
Here $[x,y]$ stands for the segment with endpoints $x$ and $y$.  In particular, $F\subset \Sigma+x-\hat z$ with $x\in \pa (\Sigma+x-\hat z)$ and
$p\perp \pa (\Sigma+x-\hat z)$ at $x$.   Analogously $\inter{\Sigma}+y-\hat z\subset G$, with $y\in \pa ({\Sigma}+y-\hat z)$ and 
$p\perp \pa ({\Sigma}+y-\hat z)$ at $y$ .  Recalling \eqref{k_*FO} and \eqref{k^*FO}, we may conclude
$$
\kappa_*(x,p,X,F)\geq \kappa(x,  \Sigma+x-\hat z)=\kappa(x, {\Sigma}+y-\hat z)\geq \kappa^*(y,p,Y,G)\,.
$$
\end{proof}

Uniqueness of viscosity solutions is a  straightforward consequence of the following  Comparison Principle, which is the main result of this subsection.

\begin{theorem}[First Order Comparison Principle] \label{th:CP1st}
Assume that condition {\rm(FO)} holds. Let $u\in USC( \RT)$ and $v\in LSC(\RT)$, both constant (spatially)
out of a compact set,
be a subsolution and a supersolution of \eqref{levelsetf}, respectively. If $u(\cdot,0) \le v(\cdot,0)$,  
then $u\le v$ in $\RT$.
\end{theorem}
\begin{proof}
Assume by contradiction that there exists $\ov z:= (\bar x, \bar t)\in \R^N\times (0,T]$ such that 
$u(\bar z)-v(\bar z)>0$.
Let $f\in \F$, $\alpha,\e>0$ and set 
$$
\Phi(x,t,y,s):= u(x,t) - v(y,s) - \alpha f(|x-y|) - \alpha(t-s)^2 -\e t - \e s.
$$ 

Notice that $\Phi$ is u.s.c. Let $(\hat x, \hat t, \hat y, \hat s) \in \RT \times \RT$
be a maximum point of $\Phi$. We may choose  $\e$ so small that the maximum is strictly positive. Moreover, as $\alpha\to\infty$ we clearly have $|\hat x - \hat y|$, $|\hat s - \hat t| \to 0$. Thus, since $\Phi(x,0,x,0)\le 0$,  we can choose  $\alpha$ so large that   $\hat s$ and $\hat t$ are strictly positive. We consider now two cases.

{\it First case: $\hat x = \hat y$.} Let 
\begin{equation}\label{eq:f}
 \f(x,t):=  v(\hat y,\hat s) + \alpha f(|x-\hat y|) + \alpha(t-\hat s)^2 +\e t + \e \hat s.
 \end{equation}
\begin{equation}\label{eq:psi}
 \psi(y,s):=  u(\hat x,\hat t) - \alpha f(|\hat x-y|) - \alpha(\hat t-s)^2 - \e \hat t - \e s.
 \end{equation}

Since $u$ is a subsolution and $v$ is a supersolution we have
$$
0\ge\f_t(\hat x,\hat t) =  
2\alpha (\hat t - \hat s) + \e, \qquad    0 \le \psi_t(\hat y,\hat s) = 2\alpha (\hat t - \hat s) - \e,
$$
which yields a contradiction.

{\it Second case: $\hat x \neq \hat y$.} 
Note that
$$
\Bigl(2\alpha (\hat t - \hat s) + \e, \alpha f'(|\hat p| ) \frac{\hat p}{|\hat p|} , X\Bigr) \in \PJet^{2,+} u(\hat x,\hat t), 
$$
$$
\Bigl(2\alpha (\hat t - \hat s) - \e, \alpha f'(|\hat p| ) \frac{\hat p}{|\hat p|} , -X\Bigr) \in  \PJet^{2,-} v(\hat y,\hat s),
$$
where  $\hat p:=\hat x-\hat y$ and $X:=D^2\f(\hat x, \hat t)$, with $\f$ defined in \eqref{eq:f}. Thus, by Lemma~\ref{laura} and Remark~\ref{rm:laura}, we have
\begin{equation}\label{sub&super}
\begin{array}{c}
\displaystyle 2\alpha (\hat t - \hat s) + \e+k_*\Bigl(\hat x,  \alpha f'(|\hat p| ) \frac{\hat p}{|\hat p|}, X, \{ u(\cdot,\hat t) \ge u(\hat x, \hat t)\}\Bigr)\le 0, \vspace{10pt} \\
\displaystyle 2\alpha (\hat t - \hat s) - \e+k^*\Bigl(\hat y,  \alpha f'(|\hat p| ) \frac{\hat p}{|\hat p|}, -X, \{ v(\cdot,\hat s) > v(\hat y, \hat s)\}\Bigr)\geq 0.
\end{array}
\end{equation}

Observe now that if $x\in  \{ u(\cdot,\hat t) \ge u(\hat x, \hat t)\}$
and $|y-x|< |\hat{y}-\hat{x}|$, then
\[
v(\hat y,\hat s) - v(y,\hat s)
\ \le\ u(\hat x,\hat t)-u(x,t) +\alpha f(|x-y|)-\alpha f(|\hat x-\hat y|)
\ <\ 0
\]
so that $y\in \{ v(\cdot,\hat s) > v(\hat y, \hat s)\}$. In other words, 
$$
\{ u(\cdot,\hat t) \ge u(\hat x, \hat t)\}+B(0,|\hat{y}-\hat{x}|)\subset  \{ v(\cdot,\hat s) > v(\hat y, \hat s)\},
$$
which by Lemma~\ref{lm:mono1*} implies
\begin{multline*}
k_*\Bigl(\hat x,  \alpha f'(|\hat p| ) \frac{\hat p}{|\hat p|}, X, \{ u(\cdot,\hat t) \ge u(\hat x, \hat t)\}\Bigr)\\
\geq k^*\Bigl(\hat y,  \alpha f'(|\hat p| ) \frac{\hat p}{|\hat p|}, -X, \{ v(\cdot,\hat s) > v(\hat y, \hat s)\}\Bigr)\,.
\end{multline*}
This inequality, combined with \eqref{sub&super}, easily leads to the  contradiction  $2\e\leq 0$. This  concludes the proof of the theorem.
\end{proof}

 \subsection{Uniqueness for second-order flows}\label{subsec:2ndorder}
Here we consider general second-order flows. In order to establish uniqueness we will need to assume the following reinforced continuity property, which replaces C) of Subsection~\ref{sec:curvature}:

\begin{itemize}
\item[C')] Uniform continuity: Given $R>0$, there exists a modulus of continuity $\omega_R$ such that the following holds.
For all $E\in\Reg$, $x\in\partial E$, such that $E$ has both an internal
and external ball condition of radius $R$ at $x$,
 and for all $\Phi:\R^N\to \R^N$  diffeomorphism of class {$C^{\regu}$},
with $\Phi(y)=y$ for $|y-x|\ge 1$, we have  
$$
|\kappa(x,E) - \kappa(\Phi(x),\Phi(E))|\le \omega_R(\|\Phi - Id\|_{C^{\regu}}).
$$
\end{itemize}
(We can observe that if $\Phi$ is a translation out of
$B(x,1)$, then the estimate still holds by translational invariance
of the curvature.)

We now prove that the uniform continuity property stated in C') extends to $\kappa_*$ and $\kappa^*$.
\begin{lemma}\label{lm:Bkappastar}
Assume that C') holds. Then,
given $R>0$, there exists a modulus of continuity $\omega_R$ such that the following holds.
For all $F\in\Ins$, $x\in\partial F$, with internal and external
ball condition at $x$ of radius $R$,
any $(p,X)\in \Jet^{2,+}_F(x)$ with $p\neq 0$,
$|X|/|p|\le 1/R$,
then for any $\Phi:\R^N\to \R^N$  diffeomorphism of class {$C^{\regu}$},
we have
\begin{multline*}
|\kappa_*(x,p,X,F)- \kappa_*(\Phi(x),D (f \circ  \Phi^{-1})(\Phi(x)),
D^2 (f \circ \Phi^{-1})(\Phi(x)) ,\Phi(F))|\\
\le \omega_R(\|\Phi - Id\|_{C^{\regu}})
\end{multline*}
where $f(y) =  (y-x)\cdot p   +  \frac{1}{2} X(y-x)\cdot (y-x)$.
The same holds true for  $\kappa^*$.
\end{lemma}
\begin{proof}
Let $E\in \Reg$ with $E\supseteq F$ and $(p,X)\in \Jet^{2,-}_{{E}}(x)$.
A first remark is that $E$ has an inner ball condition at $x$ of
radius $R$, since it contains $F$. Given $R'<R$, and letting
$B_{R'}=B(x-R'p/|p|,R')$ be the external ball
of radius $R'$ which touches $\partial F$ at $x$ (only),
we observe that
we may find a set $E'\in \Reg$ such that $E'\subseteq E\setminus B_{R'}$
and $E'\supseteq F$, and still (since the ball $B_{R'}$ is
also exterior and tangent to the set $\{f\ge 0\}$, thanks to the
condition $|X|/|p|\le 1/R$), $(p,X)\in \Jet^{2,-}_{{E'}}(x)$.
By assumption A) 
 one has
$\kappa(x,E')\ge \kappa(x,E)$. By assumption C'),
\[
\kappa(x,E')  - \kappa(\Phi(x),\Phi(E'))\le \omega_{R'}(\|\Phi - Id\|_{C^{\regu}})\,.
\]
Since $\Phi(E')\in \Reg$, $\Phi(E')\supseteq \Phi(F)$ and
$(D (f \circ  \Phi^{-1})(\Phi(x)),D^2 (f \circ \Phi^{-1})(\Phi(x)) )
\in \Jet^{2,-}_{{\Phi(E')}}(\Phi(x))$, then
\begin{multline*}
\kappa_*(\Phi(x),D (f \circ  \Phi^{-1})(\Phi(x)),
D^2 (f \circ \Phi^{-1})(\Phi(x)) ,\Phi(F)) 
\\ \ge 
\kappa(\Phi(x),\Phi(E'))
\ge \kappa(x,E')-\omega_{R'}(\|\Phi - Id\|_{C^{\regu}})
\\ \ge \kappa(x,E)-\omega_{R'}(\|\Phi - Id\|_{C^{\regu}})\,.
\end{multline*}
Taking the the supremum over all suitable sets $E$,
we deduce
\begin{multline*}
\kappa_*(\Phi(x),D (f \circ  \Phi^{-1})(\Phi(x)),
D^2 (f \circ \Phi^{-1})(\Phi(x)) ,\Phi(F)) 
\\ \ge
\kappa_*(x,p,X,F)
-\omega_{R'}(\|\Phi - Id\|_{C^{\regu}})\,.
\end{multline*}
{The other inequality in order to conclude the proof (possibly redefining slightly $\omega_R$) follows analogously, swiching the role of $E$ with $\Phi(E)$.}
\end{proof}

Now, if C') holds, we also deduce a natural comparison property
for the curvatures $\kappa_*$, $\kappa^*$ of sets included in one
another with a unique contact point.
\begin{lemma}\label{lm:mono*}
Assume C') holds.
Let $x\in\R^N$,  $F,G\in \Ins$  with $F\subset G\cup \{x\}$
and $\partial F \cap \partial G=\{x\}$.
Let $(p,X) \in \Jet^{2,+}_F(x)$,   
$(p,Y) \in \Jet^{2,-}_G(x)$, with $X\le Y$. Then,
$$
\kappa_*(x,p,X,F) \ge \kappa^*(x,p,Y,G). 
$$
\end{lemma}
\begin{proof}
First we observe that if $X<Y$ then one may find $X',Y'$ with
$X<X'<Y'<Y$ so that near the contact point $x$, $F$ lies
inside the set $\{\la p,y-x\ra + \la X'(y-x),y-x \ra/2 \ge 0 \}$
while $G$ contains $\{\la p,y-x\ra + \la Y'(y-x),y-x \ra/2 > 0 \}$.
Then, since $x$ is the unique contact point of $\partial F$ and
$\partial G$, one can build a set $E \in \Reg$ with jet $(p,X')$
at $x$ and $F\subseteq E$, $\inter E\subseteq G$. It follows,
by definition, that
\[
\kappa_*(x,p,X,F) \ge \kappa(x,E) \ge \kappa^*(x,p,Y,G). 
\]
Now, if $X\not< Y$, we let $F'=F\cup B$ where
$B=B(x+R p/|p|,R)$ with $R<|p|/|X|$ small enough
so that $x$ remains the unique point in $\partial F'\cap\partial G$.
This new set has both an
internal and external ball condition of radius $R$ at $x$, and
one still has $(x,p,X) \in \Jet^{2,+}_{F'}(x)$, moreover
since $F\subset F'$,
$\kappa_*(x,p,X,F) \ge \kappa_*(x,p,X,F')$. We can  find
a diffeomorphism of the form  $\Phi_\e(y)=y+ \e\eta(y)p$ where
$\eta:\R^d\to [0,1]$
is a smooth function with support in a neighborhood of $x$, with
$\eta(x)=0$ and $D^2\eta(x)=I$, such that the set $\Phi_\e(F')$
still has $x$ as unique contact point with $\partial G$.
Lemma~\ref{lm:Bkappastar} ensures that
$\kappa_*(x,p,X,F') \ge \kappa_*(x,p,X_\e,\Phi_\e(F')) - \omega_R(\e\|\eta\|_{C^{\regu}}|p|)$, where $X_\e = X-\e|p|^2 I$.
Since $X_\e<Y$, one has $\kappa_*(x,p,X_\e,F') \ge \kappa^*(x,p,Y,G)$,
hence
\[
\kappa_*(x,p,X,F)  \ge \kappa^*(x,p,Y,G) - \omega_R(\e\|\eta\|_{C^{\regu}}|p|)
\]
and sending $\e$ to zero we recover the thesis of the Lemma.
\end{proof}

As in the previous subsection, uniqueness of viscosity solutions is a  straightforward consequence of the following  Comparison Principle, which is the main result of this subsection.

\begin{theorem}[Second Order Comparison Principle]\label{th:CP2nd}
Assume C') holds. Let $u\in USC( \RT)$ and $v\in LSC(\RT)$, both constant (spatially)
out of a compact set,
be a subsolution and a supersolution of \eqref{levelsetf}, respectively. If $u(\cdot,0) \le v(\cdot,0)$,  
then $u\le v$ in $\RT$.
\end{theorem}
\begin{proof}
Without loss of generality we can assume $u(\cdot,0) <  v(\cdot,0)$. 
Assume by contradiction that there exists  $a\in\R$ and $t\in (0,T]$ such that
$F(t):=\{u(\cdot,t) \ge a\}$ is not contained in $G(t) :=\{v(\cdot,t) > a\}$.

Notice that, since $u(\cdot,0) <  v(\cdot,0)$, we have $F(0)\subset {G(0)}$.
A first remark is that we can assume, without loss of generality, that
the sets $F(t)$ satisfy an internal ball condition with some fixed radius
$r>0$, at all time, while $G(t)$ satisfy an external ball condition with
same radius.
Indeed, we can always replace $F$ and $G$ with, respectively,
the sets
\[
F(t)+\overline{B(0,r)}=\bigcup_{|z|\le r} (z+F(t))
\quad \mbox{ and }\quad \{ x\,:\,  B(x,r)\subset G(t)\}
\]
for some $r>0$,
which are still, respectively, a sub and a super-solution of \eqref{oee};
moreover if $r$ is small enough, the inclusion $F(0)\subset G(0)$ is
preserved.

Let $f\in \F\cap C^{\regu}$, and let
\begin{align}\label{convsol1}
 & u^{\lambda}_F(x,t) := \max_{\xi\in \R^N,  \tau \in [t-T,t]}  \chi_{F(t-\tau)}(x-\xi) - \lambda  (f(|\xi|) + \tau^2)  , \\
 \label{convsol2}
 & v^{\lambda}_G(x,t):= \min_{\xi\in \R^N, \tau\in [t-T,t]} \chi_{G(t-\tau)}(x-\xi) +  \lambda (f(|\xi|) +   \tau^2)     
\end{align}
Since $F(0)\subset G(0)$, for $\lambda$ large enough  $u^{\lambda}_F(\cdot,0)  \le  v^{\lambda}_G(\cdot,0)$. 
Then, the function 
$$
\Phi(x,t,y,s):= u^{\lambda}_F (x,t) - v^{\lambda}_G(y,s) -\e (t+s) - \lambda (f(|x-y|) + |t-s|^2)
$$ 
is semiconvex, and for $\e>0$ small and $\lambda$ large enough admits a positive maximum 
at some $(x_0,t_0,y_0,s_0) \in \RT\times \RT$ with $t_0,\, s_0>0$. 

\textit{First case:}  $ x_0 =  y_0$.
Let 
$$
 \f(x,t):=  v^{\lambda}_G( y_0, s_0) + \lambda  f(|x- y_0|)+ \lambda(t- s_0)^2 +\e t + \e s_0.
 $$
$$
 \psi(y,s):=  u^{\lambda}_F ( x_0, t_0) -  \lambda f(| x_0-y|)-  \lambda ( t_0-s)^2 - \e  t_0 - \e s.
 $$
We observe that by construction, $u^{\lambda}_F$ is a subsolution and $v^{\lambda}_G$ is a supersolution on $\R^N\times [2/\sqrt{\lambda},T]$ (we need
$t\ge 2/\sqrt{\lambda}$ to ensure that the max in~\eqref{convsol1} is not reached 
at $\tau=t$, observe though that if $\lambda$ is large enough one will
have $t_0,s_0>2/\sqrt{\lambda}$). Hence, we have
$$
0\ge\f_t( x_0 ,t_0) =  
2 \lambda ( t_0 -  s_0) + \e, \qquad    0 \le \psi_t( y_0, s_0) = 2 \lambda ( t_0 -  s_0) - \e,
$$
which yield a contradiction.

\textit{Second case:}  $ x_0 \neq  y_0$.
We can always assume (choosing $\lambda$ large enough) that $f(|x_0 - y_0|) <1$. Moreover, we have 
\begin{equation}\label{minoreuno}
u^\lambda_F(x_0,t_0) <1.
\end{equation}
Indeed, observe that $u^\lambda_F(x, t) = 1$ if and only if $x\in F(t)$,  that $D
u^\lambda_F(x, t) = 0$ on $F(t)$, and $D
u^\lambda_F(x_0, t_0) =D_x f(|x_0 - y_0|) \neq 0$. Thus,  $x_0$ can not belong to $F(t_0)$. 

Let $q: [0,+\infty]\to [0,1]$ be a smooth, nondecreasing, function with $q(r)= r^4$ for $r<1/2$
and $q(r)=1$ for $r>3/2$. For $\rho>0$, let then
$$
\Phi_{  \rho}(x,t,y,s):= \Phi_{ }(x,t,y,s) - \rho [q(|x-x_0|)+ q(|y-y_0|)+q(|t-t_0|)+q(|s-s_0|)],
$$
so that $(x_0,t_0,y_0,s_0)$ is a strict maximum of $\Phi_{  \rho}$. 
Let $\eta:\R^N\to \R$ be a smooth cut-off function, with compact support and equal to one in a neighborhood $U$ of the origin.
For every
 $\Delta:=(\zeta_u, h_u ,\zeta_v, h_v) \in \R^N\times \R\times\R^N\times\R$, the function
$$
\Phi_{\rho}(x,t,y,s) - \Big(\eta(x - x_0)  (\xi_u, h_u)\cdot (x,t) + \eta(y - y_0)   (\xi_v,h_v)\cdot  (y,s)  \Big)
$$ 
is maximized at some $(x_\Delta, t_\Delta, y_\Delta, s_\Delta)$ such that 
\begin{equation*}
(x_\Delta, t_\Delta, y_\Delta, s_\Delta) \to (x_0, t_0, y_0, s_0) \qquad \text{ as }  |\Delta| \to 0. 
\end{equation*} 
Thus, by Jensen's Lemma~\cite[Lemma~A.3]{CIL}, we may assume that for every $\delta>0$ sufficiently small there exists 
$\Delta _{\rho, \delta}:=(\zeta^{\rho, \delta}_u, h^{\rho, \delta}_u ,\zeta^{\rho, \delta}_v, h ^{\rho, \delta}_v)$, {with $|\Delta _{\rho, \delta}|\le \delta$}, such that the function
\begin{multline*}
\Phi_{  \rho, \delta}(x,t,y,s):=\Phi_{\rho}(x,t,y,s) \\ 
- \Big(\eta(x - x_0)  (\xi ^{\rho, \delta}_u, h ^{\rho, \delta}_u)\cdot (x,t) + \eta(y - y_0)  (\xi ^{\rho, \delta}_v,h ^{\rho, \delta}_v)\cdot  (y,s)  \Big)
\end{multline*}
attains a maximum at some $z _{\rho, \delta}:=(x _{\rho, \delta}, t _{\rho, \delta}, y _{\rho, \delta}, s _{\rho, \delta})$
 where $\Phi_{\delta, \rho}$ is twice differentiable and such that  $x _{\rho, \delta}-x_0$, $y _{\rho, \delta}-y_0 \in U$ and $t _{\rho, \delta},\, s _{\rho, \delta} >0$.  Moreover, 
 \begin{equation}
z _{\rho, \delta} \to (x_0, t_0, y_0, s_0) \qquad \text{ as }  \delta \to 0. 
\end{equation} 
Notice that since $\Phi_{\rho}$ is twice differentiable  at $z _{\rho, \delta}$ it follows that also $u^{\lambda}_F$ and $ v^{\lambda}_G$ are twice differentiable at  $(x _{\rho, \delta}, t _{\rho, \delta})$ and $(y _{\rho, \delta}, s _{\rho, \delta})$, respectively.

Let  $ \tau^{\rho, \delta}_u, \,   \tau^{\rho, \delta}_v \in \R$ be the maximizing $\tau$'s in \eqref{convsol1}, \eqref{convsol2} corresponding to the points $(x _{\rho, \delta},t _{\rho, \delta})$,  $(y _{\rho, \delta},s _{\rho, \delta})$, respectively.  Set
\begin{align}\label{tilde}
&\tilde u_F(x,t) := \max_{\xi\in \R^N}  \left\{ \chi_{F(t- \tau^{\rho, \delta}_u)} (x-\xi) - \lambda  f(|\xi|)\right\}  - \lambda   (\tau^{\rho, \delta}_u)^2 \\ 
&\tilde v_G(y,s):= \min_{\xi\in \R^N} \left\{ \chi_{G(s- \tau^{\rho, \delta}_v)}(y - \xi) +  \lambda f(|\xi|) \right\} + \lambda   (\tau^{\rho, \delta}_v)^2.
\end{align}
Observe that by construction,
{
 \begin{align}\label{convsol4}
& u^{\lambda}_F\ge \tilde u_F,
\qquad v^{\lambda}_G\le \tilde v_G,\, \\
& u^{\lambda}_F(x _{\rho, \delta},t _{\rho, \delta}) =  \tilde u_F(x _{\rho, \delta},t _{\rho, \delta}), \, 
\qquad v^{\lambda}_G(y _{\rho, \delta},s _{\rho, \delta}) = \tilde v_G(y _{\rho, \delta},s _{\rho, \delta}). \nonumber     
\end{align}
}
Set now
\begin{multline*}
\hat u_F(x,t):= \tilde u_F(x,t) -\rho \,q(|x-x _{\rho, \delta}|)   \\
 - \rho [q(|x-x_0|) + q(|t-t_0|)] - \eta(x - x_0)  (\xi ^{\rho, \delta}_u, h ^{\rho, \delta}_u)\cdot (x,t) ,
\end{multline*}
\begin{multline*}
\hat v_G(y,s):=  \tilde v_G(y,s) +\rho\, q(|y-y _{\rho, \delta}|) \\
+ \rho \, [ q(|y-y_0|) +q(|s-s_0|) ] + \eta(y - y_0)  (\xi ^{\rho, \delta}_v,h ^{\rho, \delta}_v)\cdot  (y,s) .
\end{multline*}
Then by construction, the function
\[ 
 \hat u_F(x,t) -  \hat v_G(y,s)
 -\e(t+s)- \lambda( f(|x-y|) + |t-s|^2). 
\] 
has its maximum at  $z _{\rho, \delta}$, which is now
strict with respect to the spatial variables. Hence if we set
\begin{equation*}
 \hat F_{\rho, \delta}(t):=\{ \hat u_F(\cdot,t) \ge 
\hat u_F(x _{\rho, \delta},t _{\rho, \delta})\}, 
\qquad
 \hat G_{\rho,\delta}(s):=\{   \hat v_G(\cdot,s) > \hat v_G(y _{\rho, \delta},s _{\rho, \delta}) \}. 
\end{equation*}
we have $\hat F_{\rho,\delta}(t _{\rho, \delta})\subseteq \hat G_{\rho,\delta}(s _{\rho, \delta})$ and moreover
 $(x _{\rho, \delta},  y _{\rho, \delta})$ is the only pair of minimal distance in $\partial \hat F_{\rho,\delta}(t _{\rho, \delta}) \times \partial \hat G_{\rho,\delta}(s _{\rho, \delta})$. 
In addition, we observe that at the maximum point,
$|D \hat u_F(x_{\rho,\delta},t_{\rho,\delta})|\approx
|D \tilde{u}(x_{\rho,\delta},t_{\rho,\delta})|=
|D u^\lambda_F(x_{\rho,\delta},t_{\rho,\delta})|\approx \lambda f'(|x_{\rho,\delta}-y_{\rho,\delta}|) \approx \lambda f'(|x_0-y_0|)\neq 0$ up to perturbations
which go to zero as $\rho,\delta\to 0$, and that the function
$\hat u_F$ is semiconvex, hence $\hat{F}_{\rho,\delta}(t_{\rho,\delta})$ has an
interior ball condition at $x_{\rho,\delta}$ with a radius independent
on $\rho$ and $\delta$, if small enough. In the same way, 
$\hat G_{\rho,\delta}(s _{\rho, \delta})$ has an exterior ball condition at
$y_{\rho,\delta}$. {In turns, $\hat F_{\rho,\delta}(t_{\rho, \delta})$ and $\hat G_{\rho,\delta}(s _{\rho, \delta})$ satisfy both and internal and external ball condition.}

Set 
$$
\breve \Phi_{  \rho, \delta}(x,t,y,s):= \Phi_{\rho, \delta}(x,t,y,s)+\lambda(  f(|x-y|) + |t-s|^2)
$$ 
and
\begin{align}\label{super1}
& ( \breve a_{\rho,\delta}, \breve p_{\rho,\delta}, \breve X_{\rho,\delta}) 
:= ( \partial_t \breve\Phi_{  \rho, \delta} (z _{\rho, \delta}) , D_x \breve\Phi_{  \rho, \delta} (z _{\rho, \delta}), D^2_x \breve\Phi_{  \rho, \delta} (z _{\rho, \delta}) ),
\\ 
\label{super2}
& (\breve b_{\rho,\delta}, \breve q_{\rho,\delta},  \breve Y_{\rho,\delta}) :=  ( \partial_s \breve\Phi_{  \rho, \delta} (z _{\rho, \delta}) , D_y \breve\Phi_{  \rho, \delta} (z _{\rho, \delta}), D^2_y \breve\Phi_{  \rho, \delta} (z _{\rho, \delta}) ).
\end{align}
Then, recalling \eqref{convsol4}, 
{ we observe that} the superjet $( \breve a_{\rho,\delta}, \breve p_{\rho,\delta},  \breve X_{\rho,\delta})$ of
\[
u^\lambda_F(x,t)-
  \rho [q(|x-x_0|) + q(|t-t_0|)] - \eta(x - x_0)   (\xi ^{\rho, \delta}_u, h ^{\rho, \delta}_u)\cdot (x,t) 
\]
at $(x_{\rho,\delta},t_{\rho,\delta})$ is also a superjet for $\hat{u}_F(x,t)$
at the same point, hence, also, for the function $\hat u_F (x _{\rho, \delta},t _{\rho, \delta}) \chi_{ \hat F_{\rho,\delta}}$
{(since $\hat u_F(x,t)\geq \hat u_F (x _{\rho, \delta},t _{\rho, \delta}) \chi_{ \hat F_{\rho,\delta}(t)}(x)$).} 
Hence,  we have
\begin{equation}\label{Paratona}
 ( \breve a_{\rho,\delta}, \breve p_{\rho,\delta},  \breve X_{\rho,\delta}) \in \PJet^{2,+}_{\hat u_F (x _{\rho, \delta},t _{\rho, \delta}) \chi_{ \hat F_{\rho,\delta}} }(x _{\rho, \delta},t _{\rho, \delta})
\end{equation}
and analogously,
\[
 ( \breve b_{\rho,\delta},  \breve q_{\rho,\delta},  \breve Y_{\rho,\delta}) \in \PJet^{2,-}_{\hat v_G (y _{\rho, \delta},s _{\rho, \delta}) \chi_{  \hat G_{\rho,\delta}} }(y _{\rho, \delta},s _{\rho, \delta}).
\]
Since ${D_z} \Phi_{\rho, \delta}(z _{\rho, \delta}) = 0$, we deduce 
\begin{equation}\label{deduce}
 \breve a_{\rho,\delta} -  \breve b_{\rho,\delta} = 2\e, \qquad  \breve p_{\rho,\delta} =  \breve q_{\rho,\delta}.
\end{equation}
Moreover, since $D^2_{x,y}  \Phi_{\rho, \delta}(z _{\rho, \delta}) \le 0$, one can check that 
\begin{equation}\label{deduce2}
 \breve X_{\rho,\delta} \le  \breve Y_{\rho,\delta}.
\end{equation} 
By construction, $\breve \Phi_{\rho,\delta}$ is also semiconvex, so 
that $\breve X_{\rho,\delta} \ge -cI$, $\breve Y_{\rho,\delta}\le
cI$ for a constant $c$ depending on $\lambda$.

Let 
$$
c_{\rho,\delta}(x,t):= 
   \tilde u_F(x,t) + (\hat u_F(x _{\rho, \delta},t _{\rho, \delta}) - \hat u_F(x,t)).
$$
Notice that, as $\delta\to 0$,  $\rho\to 0$, we have  $c_{\rho,\delta} \to u^\lambda_F(x_0,t_0)$ uniformly. In view of \eqref{minoreuno} we can thus assume that
$c_{\rho,\delta} < 1$. Observe also that $c_{\rho,\delta}$ is smooth and
constant away from a neighborhood of $(x_0,t_0)$, { and it converges
also in $C^{\regu}$}.

We have
$
\hat F_{\rho, \delta}(t) = \{ x:  \tilde u_F(x,t)\ge c_{\rho,\delta}(x,t)\}.
$
Thus, by the definition of $\tilde u_F$, we have that $x\in \hat F_{\rho, \delta}(t)$ if and only if there exists $\xi\in \R^N$ such that $x\in \xi+F(t- \tau^{\rho, \delta}_u)$, with
$$
\chi_{\xi+F(t - \tau^{\rho, \delta}_u)}(x)-\lambda f(|\xi|)=1-\lambda f(|\xi|)\geq c_{\rho,\delta}(x,t)\, , 
$$
i.e.,
\begin{multline}\label{fhatrd}
\hat F_{\rho, \delta}(t)=\Biggl\{x:  \, x \in \xi + F(t- \tau^{\rho, \delta}_u)
\\
 \text{ for some } \xi\in\R^N \text{ with } |\xi|\le f^{-1} \left(\frac{1-c_{\rho,\delta}(x,t)}{\lambda}\right) \Biggr\}.
\end{multline}

For $\rho,\delta$ small enough,
$x_{\rho, \delta}\not\in F(t_{\rho, \delta}- \tau^{\rho, \delta}_u)$ and we can introduce
$w_{\rho,\delta}\neq 0$ such that $x_{\rho, \delta} + w_{\rho,\delta} $ is a projection of $x _{\rho, \delta}$ on  $ F(t _{\rho, \delta} -  \tau^{\rho, \delta}_u)$.
Precisely, one has that $c_{\rho,\delta}(x_{\rho,\delta},t_{\rho,\delta})=
u^\lambda_F(x_{\rho,\delta},t_{\rho,\delta})$, $\xi=-w_{\rho,\delta}$ reaches
the max in~\eqref{convsol1} (for $x=x_{\rho,\delta}$),
and $|w_{\rho,\delta}|=f^{-1}((1-c_{\rho,\delta}(x_{\rho,\delta},t_{\rho,\delta}))/\lambda)$. We then set
$$
\Psi_{\rho,\delta}(x) : = x- f^{-1} \left(\frac{1-c_{\rho,\delta}(x, t_{\rho,\delta})}{\lambda}\right)  \frac{w_{\rho,\delta}}{|w_{\rho,\delta}|} + w_{\rho,\delta}
$$
which is a $C^{\regu}$ diffeomorphism (being $c_{\rho,\delta}$ bounded away from $1$), which is a constant (small) translation
out of a neighborhood of $x_0$,  and converges {$C^{\regu}$}
to the identity as
$\delta \to 0$ and $\rho\to 0$. Observe
that $\Psi_{\rho,\delta}(x_{\rho,\delta})=x_{\rho,\delta}$. Then, we let
\begin{equation}\label{checkF}
\check F_{\rho,\delta}(t):=\Psi_{\rho,\delta} (F(t- \tau^{\rho, \delta}_u) - w_{\rho,\delta}) .
\end{equation}
By construction, ${\check F}_{\rho,\delta}(t_{\rho, \delta}) \subseteq  \hat F_{\rho,\delta}(t_{\rho, \delta})$ and $x _{\rho, \delta} \in \partial {\check F}_{\rho,\delta}(t _{\rho, \delta}) \cap \partial  \hat F_{\rho,\delta} (t _{\rho, \delta})$. 
Moreover, we recall that $\hat F_{\rho,\delta}(t_{\rho,\delta})$ has
an internal ball condition at $x_{\rho,\delta}$ while
$\hat G_{\rho,\delta}(s_{\rho,\delta})$ satisfies an external ball condition
at $y_{\rho,\delta}$,
with radius bounded away from $0$ uniformly with respect to $\rho$ and $\de$ sufficiently small.
 Thus,  recalling that $\hat F_{\rho,\delta}(t _{\rho, \delta}) + (y _{\rho, \delta}-x _{\rho, \delta})\subseteq \hat G_{\rho,\delta}(s _{\rho, \delta}) $ (being $y _{\rho, \delta}$ the only contact point), we have that
$\hat F_{\rho,\delta}(t _{\rho, \delta})$, and in turn $\check F_{\rho,\delta}(t _{\rho, \delta})$ satisfies a uniform external ball condition in $x _{\rho, \delta}$. 
In addition, since we have assumed that $F(t)$ had a uniform internal
ball condition for some radius $r>0$, the same holds for 
$\check F_{\rho,\delta}(t _{\rho, \delta})$ with a smaller radius.

Notice that {$\| \Psi_{\rho,\delta} - I \|_{C^{\regu}} \to 0$} as $\rho, \, \delta\to 0$. Set
\begin{multline*}
 (p_{\rho,\delta}, X_{\rho,\delta}) := 
\Big( D_x(\breve\Phi_{\rho, \delta}(\cdot, t_{\rho, \delta}, y_{\rho, \delta}, s_{\rho, \delta})\circ\Psi_{\rho, \delta})(x_\delta), \\
D_x^2(\breve\Phi_{\rho, \delta}(\cdot, t_{\rho, \delta}, y_{\rho, \delta}, s_{\rho, \delta})\circ\Psi_{\rho, \delta})(x_\delta) \Big),
\end{multline*}
By construction {(see \eqref{Paratona}),} we have
$$
( \breve a_{\rho,\delta},p_{\rho,\delta}, X_{\rho,\delta}) \in \PJet^{2,+}_{\hat u_F(x _{\rho, \delta},t _{\rho, \delta}) \chi_{F(t- \tau^{\rho, \delta}_u  )   }  }(x _{\rho, \delta} + w_{\rho,\delta})
$$
Since $\hat u_F(x _{\rho, \delta},t _{\rho, \delta}) \chi_{F(t- \tau^{\rho, \delta}_u)}$ is a subsolution, we have
\begin{equation}\label{ass1}
 \breve a_{\rho,\delta} + |p_{\rho,\delta}| \kappa_* (x _{\rho, \delta} + w_{\rho,\delta},p_{\rho,\delta}, X_{\rho,\delta},F(t _{\rho, \delta} - \tau^{\rho, \delta}_u ))  \le 0.
\end{equation}
Note that  
$$
p_{\rho,\delta}\to Du_F^{\lambda}(x_0, t_0)\neq 0\,,
$$
as $\rho$, $\de\to 0$, and thus $|p_{\rho,\delta}|$ is bounded away from zero for  $\rho$ and $\delta$ sufficiently small. Since also $\breve X_{\rho,\delta}$
and hence $X_{\rho,\delta}$ is bounded, we can invoke Lemma~\ref{lm:Bkappastar}
and deduce that 
\begin{equation}\label{ass3}
 \breve a_{\rho,\delta} + | \breve p_{\rho,\delta}| \kappa_* (x _{\rho, \delta},  \breve p_{\rho,\delta},  \breve X_{\rho,\delta}, \check F_{\rho, \delta}(t _{\rho, \delta}  ))  \le \omega(\rho,\delta),
\end{equation}
where $ \omega(\rho,\delta)\to 0$ as $\rho,\, \delta\to 0$.

Analogously, we also have
\begin{equation}\label{ass4}
 \breve a_{\rho,\delta} -2\e + | \breve p_{\rho,\delta}| \kappa^* (y _{\rho, \delta} ,  \breve p_{\rho,\delta},  \breve Y_{\rho,\delta}, \check G_{\rho, \delta}(s _{\rho, \delta}))  \ge \omega(\rho,\delta)
\end{equation}
for a suitable set $\check G_{\rho, \delta}(s _{\rho, \delta}))$ such that $\hat F(t _{\rho, \delta}  ) + (y _{\rho, \delta} - x _{\rho, \delta}) \subseteq \check G_{\rho, \delta}(s _{\rho, \delta}))$ and  $\partial (\check F _{\rho, \delta}(t _{\rho, \delta}  )+  (y _{\rho, \delta} - x _{\rho, \delta})) \cap \partial \check G _{\rho, \delta}(s _{\rho, \delta})) = \{y _{\rho, \delta}\}$.
By~\eqref{deduce2} and Lemma~\ref{lm:mono*} we get 
$$
\kappa_* (x _{\rho, \delta},  \breve p_{\rho,\delta},  \breve X_{\rho,\delta}, \check F_{\rho, \delta}(t _{\rho, \delta}  )) \ge
\kappa^* (y _{\rho, \delta} ,  \breve p_{\rho,\delta},  \breve Y_{\rho,\delta}, \check G_{\rho, \delta}(s _{\rho, \delta})),
$$ 
and thus, in particular, 
\begin{equation*}
 \breve a_{\rho,\delta} -2\e + | \breve p_{\rho,\delta}| \kappa_* (x _{\rho, \delta},  \breve p_{\rho,\delta},  \breve X_{\rho,\delta}, \check F_{\rho, \delta}(t _{\rho, \delta}  ))  \ge 2\omega(\rho,\delta),
\end{equation*}
which, together \eqref{ass3} gives $\e \le \omega(\rho,\delta)$. This  is a contradiction for $\rho,\,\delta$ sufficiently small. 
\end{proof}

{
\begin{remark}\label{intrinsicsuplev}
\rm
By the uniqueness property stated in Theorem \ref{th:CP1st} and Theorem \ref{th:CP2nd}, we deduce that the evolution of { open and closed superlevel sets is intrinsic in the following sense. Let 
$u^0, \, \tilde u^0:\R^N\to \R$ be two initial conditions such that 
$
\{u^0(\cdot) >0\} = \{\tilde u^0(\cdot) >0\}. 
$
Then, denoting by $u$ and $\tilde u$ the corresponding geometric evolutions we have   
$$
\{u(\cdot,t) >0\} = \{\tilde u(\cdot,t) >0\} \qquad \text{ for all } t\in[0,t] 
$$
(and the same identity holds for the closed superlevels).}
Indeed, for any $\lambda\geq 0$ set
\begin{align*}
& A_\lambda (t):= \{u(\cdot,t ) > \lambda\},    \qquad \tilde A_\lambda (t):= \{\tilde u(\cdot,t ) > \lambda\}, \\
& C_\lambda (t):= \{u(\cdot,t ) \ge \lambda\},    \qquad \tilde C_\lambda (t):= \{\tilde u(\cdot,t ) \ge \lambda\},
\end{align*}
In view of Lemma~\ref{L1} and of Theorems~\ref{th:CP1st} and \ref{th:CP2nd},  we have that for every  $\lambda>0$
$$
\tilde C_\lambda (t) \subseteq A_0 (t), \qquad  C_\lambda (t) \subseteq \tilde A_0 (t). 
$$
Thus, we conclude
$$
\tilde A_0 (t) = \bigcup_{\lambda >0} \tilde C_\lambda (t) \subseteq A_0 (t), \qquad 
A_0 (t) = \bigcup_{\lambda >0}  C_\lambda (t)   \subseteq \tilde A_0 (t). 
$$
\end{remark}
}

\newpage
\part{Variational Nonlocal Curvature Flows}\label{part:2}

In this part we further assume the nonlocal curvature to be {\em variational}; that is, we assume that $\kappa$ is the {\em first variation} of a suitable {\em generalized perimeter}. The second part of the paper is organized as follows. In Section~\ref{sec:4} we introduce a suitable class of translation invariant generalized perimeters $J$ and we give a rather weak notion of curvature as a first variation of the perimeter functional with respect to measurable perturbations of the set shrinking to a point $x$ of the boundary. We also show how some of the structural assumptions of $J$ translate into properties of the curvature $\kappa$. 

 In Section~\ref{secfirstvar} we study how the weak notion of curvature compares to more standard ones. Section~\ref{sec:examples} is devoted to showing some relevant examples of variational nonlocal curvatures that fit in our abstract framework. Finally, Section~\ref{sec:ATW} contains the main result of this part, namely the fact that the minimizing movement scheme applied to $J$ converges to the associated nonlocal curvature flow.

\section{Generalized perimeters and curvatures}\label{sec:4}

\subsection{Generalized perimeters}\label{subsec:gp}
We now will consider a generalized notion of (possibly nonlocal) perimeter.  
We will say that a functional  $J:\Ins\to [0,+\infty]$ is a generalized perimeter if it satisfies the following properties:

\begin{itemize}
\item[i)]  $J(E)<+\infty$  for every $E\in\Reg$; 
\item[ii)] $J(\emptyset)=J(\R^N)=0$; 
\item[iii)] $J(E)=J(E')$ if $|E\triangle E'|=0$; 
\item[iv)] $J$ is $L^1_{loc}$-l.s.c.: if $|(E_n\triangle E)\cap B_R|\to 0$ for every $R>0$, then
 $$J(E)\le \liminf_n J(E_n);$$
\item[v)] $J$ is { submodular}: For any $E, \, F\in\Ins$,
\begin{equation}\label{eq:submo}
J(E\cup F)\,+\, J(E\cap F)\ \le\ J(E)\,+\,J(F)\ ;
\end{equation}
\item[vi)] $J$ is translational invariant: 
\begin{equation}\label{eq:transinv}
J(x+E)\ =\ J(E)\quad \text{ for all } E \in \Ins ,\,x\in \R^N\,.
\end{equation}
\end{itemize}

We can extend the functional $J$ to $L_{loc}^1(\R^N)$
enforcing the following {\it generalized co-area formula}:
\begin{equation}\label{visintin}
{J}(u) :=\ \int_{-\infty}^{+\infty} J(\{u>s\})\,ds \qquad \text{ for every } u\in L_{loc}^1(\R^N).
\end{equation}
It can be
shown that, under the assumptions above,  $J$ is  a convex l.s.c.~functional in $L_{loc}^1(\R^N)$ (see~\cite{ChamGiacoLuss}). 

Observe that the following holds true:
\begin{lemma}\label{lemapprox}
Let $u\in L^1_{loc}(\R^N)$ and $\rho$ be a nonnegative
and compactly supported mollifier,
and let $\rho_\e(x):=\rho(x/\e)/\e^N$ for $\e>0$ small.
Then for all $\e$,
\begin{equation}\label{eq:minore}
J(\rho_\e*u)\le J(u)\,.
\end{equation}
Moreover,  $\lim_{\e\to 0}J(\rho_\e*u)= J(u)$.
\end{lemma}
\begin{proof}The first statement follows from the convexity of $J$, by approximating $\rho_\e*u$
by appropriate finite convex combinations and then passing to the limit
thanks to the lower semicontinuity assumption iv). The last statement is then an immediate consequence of  \eqref{eq:minore} and once again of
assumption iv). 
\end{proof}
\subsection{A weak notion of curvature}
We introduce here a definition of the curvature of sets
relative to the generalized perimeter $J$ which will be useful
for studying the geometric ``gradient flow'' of $J$. It is based on
a sort of local subdifferentiability property. We will
show in Section~\ref{secfirstvar} that it is implied by
more standard definitions based on global variations of the boundaries
of smooth sets.

\begin{definition}\label{defofk}
Let $E\in \Reg$ and $x\in\partial E$. We set
\begin{equation}\label{ineqbelow}
\kappa^+(x,E):= \inf \left\{   \liminf_n \frac{J(E\cup W_n)\, - J(E)}{|W_n\setminus E|}: { \overline{W}_n \kto \{x\}}, \, |W_n\setminus E|>0\right\} 
\end{equation}
and 
\begin{equation}\label{ineqabove}
\kappa^-(x,E): = \sup  \left\{    \limsup_n \frac{J(E)\,- \, J(E\setminus W_n)}{|W_n\cap E|}:  { \overline{W}_n \kto \{x\}}, \, |W_n\cap E|>0   \right\} .
\end{equation}
 We say that $\kappa(x,E)$  is the curvature of $E$ at $x$  (associated with the perimeter $J$) if  $\kappa^+( x,E)= \kappa^-( x,E)=:\kappa(x,E)\in \R$.
\end{definition} 

Notice that if $J(E) = J(\R^N\setminus E)$  it follows that $\kappa^+(x,E) = - \kappa^-(x,\R^N\setminus E)$, and therefore  $\kappa(x,E) = -\kappa(x,\R^N\setminus E)$ (whenever it exists). 

\vskip 10pt
{\bf Standing Assumptions of Part~\ref{part:2}}. {\em Throughout Part~\ref{part:2} we  assume that  the curvature exists for all sets in $\Reg$, i.e., 
$$
\text{ $ {\kappa(x,E):=\kappa^+( x,E)= \kappa^-( x,E)\in \R}$
 for all ${E\in\Reg}$ and all ${x\in\partial E}$,}
$$
 and  that
 $$
\text{ $\kappa$ satisfies  axiom C) of Subsection~\ref{sec:curvature}.}
$$}

\vskip 10pt
The  translational invariance B) follows
naturally from the translational invariance of the perimeter $J$.
The monotonicity property  A) stated in Subsection~\ref{sec:curvature} is a consequence of the submodularity assumption~\eqref{eq:submo}, as shown in the next lemma.
\begin{lemma}\label{lemmaxp}
Let $E, F\in\Reg$ with $E\subseteq F$, and assume that $x\in \partial F\cap \partial E$:
then $\kappa(x,F)\le\kappa(x,E)$.
\end{lemma}
\begin{proof}
First, we observe that we can find sets $F_n$ which converge to
$F$ in $\Reg$, with $F_n\supseteq F$ and $\{x\}=\partial E\cap \partial F_n$.
In particular, $\kappa(x,F_n)\to \kappa(x,F)$.
Then, let $\nu$ be the (outer) normal vector to $E$ and (all of the) $F_n$ at
$x$, and for $\e>0$ small let $E^\e=E+\e\nu$. Let $W^\e=E^\e\setminus \inter F_n$,
and observe that $ W^\e\kto \{x\}$  as $\e\to 0$, and
$|W^\e|>0$ if $\e>0$ (small).
Thanks to~\eqref{eq:submo} (applied to $E^\e$ and $F_n$) and~\eqref{eq:transinv}, we have
\[
J(F_n\cup W^\e)-J(F_n)\ \le\ J(E^\e)-J((E^\e)\setminus W^\e)
\ =\ J(E)-J(E\setminus (W^\e-\e\nu))\,.
\]
Then, by the very definition of $\kappa$ we deduce $\kappa(x,F_n)\le \kappa(x,E)$. The conclusion follows noticing that,  by the continuity property  C),  $\kappa(x,F_n)\to\kappa(x,F)$.
\end{proof}

\section{First variation of the perimeter}\label{secfirstvar}

Let $J$ be a generalized perimeter. In this section we compare the weak notion of curvature given in Definition~\ref{defofk} with the more standard one based on the {\em first variation} of the perimeter functional. 
The latter is in turn related to  \textit{shape derivatives}, a notion which dates back to Hadamard,
extensively studied in particular in~\cite{MuratSimon76}.

\begin{definition}\label{deffirstvar}
We say that $\kappa(x,E)$, defined for $E\in \Reg$ and $x\in \partial E$,
is the first variation of the perimeter { $J$ if}
for every $E\in \Reg$,  and any one-parameter family of  diffeomorphisms $(\Phi_\e)_\e$ of class $C^{\regu}$ both in $x$ and in $\e$
 with $\Phi_0(x)=x$, we have
\begin{equation}\label{subgrad3}
\frac{d}{d \e} J \big( \Phi_\e(E)\big)_{|_{\e=0}} =   \int_{\partial E} \kappa(x,  E) \psi(x)\cdot\nu_E(x) d \H^{N-1}(x)\,,
\end{equation}
where $\psi(x):=\frac{\partial \Phi_\e}{\partial \e}(x)_{|_{\e=0}}$ and $\nu_E(x)$ is the $C^{\regu[-1]}$ outer normal to the set $E$ at $x$.
\end{definition}

We will show that if such a $\kappa(x,E)$ is continuous with
respect to $C^{\regu}$ perturbations of the sets $E$, then it is also a curvature
in the sense of Definition~\ref{defofk}.
We start with the following intermediate result.
\begin{proposition} \label{propsubgrad}
Let $\kappa(x,E)$ be a function defined for
all $E\in\Reg$ and $x\in \partial E$, and assume that it satisfies
the continuity property C). Then $\kappa(x,E)$ is a first variation
of $J$
in the sense of Definition~\ref{deffirstvar} if and only if
for any $\f\in C^{\regu}_c(\R^N)$, and any $t_1<t_2$ such that $D \f\neq 0$
in the set $\{t_1\le \f \le t_2\}$, one has
\begin{equation}\label{intkappa} 
J(\{ \f\ge t_1\})  =  J(\{\f \ge t_2\})+\int_{\{t_1<\f<t_2\}}
\kappa(x,\{\f \ge \f(x)\})\,dx.
\end{equation}

Moreover, in this case, one also has that for any set $W\in \Ins$
such that $\{\f\ge t_2\}\subset W \subset \{\f\ge t_1\}$,
\begin{equation}\label{intW} 
J(W)\ge   J(\{\f\ge t_2\}) +\int_{W\setminus \{\f\ge t_2\}} 
\kappa(x,\{\f\ge \f(x)\})\,dx.
\end{equation}
\end{proposition}
Observe that in both integrals, the argument $\kappa(x,\{\f\ge\f(x)\})$
is a continuous function of $x$, thanks to assumption C).
\begin{proof}
Consider the function  $f(s):= J(\{\f\ge s\})$, for $t_1<s<t_2$.
We claim that if $\kappa$ is a first variation of $J$, then
\begin{equation}\label{derivf}
f'(s) = -\int_{\partial \{f \ge s\}} \frac{\kappa(x,\{\f\ge s\})}{|D \f(x)|}
\,d \H^{N-1}(x)
\end{equation}
for all $s\in (t_1, t_2)$.
To this aim, given $s\in (t_1, t_2)$ with $t_1<s<t_2$, we need to find a family of $C^{\regu}$ diffeomorphisms
which transport $\{\f\ge s\}$ on $\{\f\ge s+\e\}$ and compute
its derivative at $\e=0$. 

If $\f$ were smooth (at least $C^{\regu[+1]}$), 
a simple way would be to
consider a smooth vector field $V(x)$ which is zero in
$\{\f\ge t_2\}\cup\{\f\le t_1\}$ and equal to $D \f/|D \f|^2$ in a neighborhood of $\{\f=s\}$. We would then
let $\Phi_\e(x)$ defined for all $x$ by
\[
\begin{cases}
\frac{d \Phi_\e(x)}{d \e} = V(\Phi_\e(x)) & \ \e>0,\\
\Phi_0(x) = x.
\end{cases}
\]
In this case for $\e$ small,
we would have that $\f(x)=s$ implies $\f(\Phi_\e(x))=s+\e$,
since clearly $d (\f(\Phi_\e(x)))/d\e = D\f(\Phi_\e(x))\cdot V(\Phi_\e(x))=1$
for such $x$ and $\e$. Then, \eqref{derivf} would follow from~\eqref{subgrad3}.
However, if $\f$ is merely $C^{\regu}$, this construction
builds only a $C^{\regu[-1]}$ diffeomorphism.

In general the situation is a bit more complex, however it is
clear that such a diffeomorphism exists. A relatively simple construction
consists in smoothing $\f$ with a smooth mollifier in order to find
a $C^\infty$ set $\tilde E$  such that for all $\e$ small enough (here
both positive and nonpositive),
the surfaces $\partial \{\f\ge s+\e \}$ are represented as
$C^{\regu}$ graphs over $\partial\tilde E$:
\[
\partial \{\f\ge s+\e \} = 
\left\{ x+h_\e(x)\nu_{\tilde{E}}(x)\,:\, x\in \partial\tilde E \right\}.
\]
By the implicit function theorem, $h_\e$ exists and is $C^{\regu}$ (in both
$\e$ and $x$) for $\e$ near $0$. Moreover, since
$\f(x+h_\e(x)\nu_{\tilde{E}}(x))=s+\e$, one checks that
for any $\e$ small and $x\in \partial \tilde{E}$,
\[
\frac{\partial h_\e(x)}{\partial \e} = \frac{1}{D \f(x+h_\e(x)\nu_{\tilde{E}}(x))\cdot \nu_{\tilde{E}}(x)}. 
\]
The diffeomorphism $\Phi_\e(x)$ is then simply defined, in a neighborhood
of the surface $\partial\tilde{E}$, by
\begin{equation}\label{num}
\Phi_\e(x)=
x+(h_\e(\pi_{\partial \tilde{E}}(x))-h_0(\pi_{\partial \tilde{E}}(x)))
\nu_{\tilde{E}}(\pi_{\partial \tilde{E}}(x))\,,
\end{equation}
where  $\pi_{\partial\tilde{E}}$ denotes the orthogonal projection
onto $\partial\tilde{E}$, which is well-defined and smooth in a sufficiently small neighborhood of the surface.

Then, one has that for $x\in \partial\{ \f\ge s\}$ (which is the graph of
$h_0$)
\[
\psi(x):=\lim_{\e\to 0}\frac{\Phi_\e(x)-x}{\e}=
\frac{\partial h_\e}{\partial \e}_{|_{\e=0}}(\pi_{\partial \tilde{E}}(x))
\nu_{\tilde{E}}(\pi_{\partial \tilde{E}}(x))
=\frac{\nu_{\tilde{E}}(\pi_{\partial \tilde{E}}(x))}{D \f(x)\cdot \nu_{\tilde{E}}(\pi_{\partial \tilde{E}}(x))}
\]
and, in turn,
\[
\psi(x)\cdot \nu_{\partial\{\f\ge s\}}(x)=
-\psi(x)\cdot\frac{D \f}{|D \f|}(x)
=-\frac{1}{|D \f(x)|}\,.
\]
Hence~\eqref{subgrad3} yields
\begin{multline}
\lim_{\e\to 0}\frac{J(\{\f\ge s+\e\}) - J(\{\f\ge s\})}{\e} 
\\=\lim_{\e \to 0} \frac{J(\Phi_\e(\{\f\ge s\})) - J(\{\f\ge s\})}{\e}
= -\int_{\partial \{\f \ge s\}} \frac{\kappa(x,\{\f\ge s\})}{|D \f(x)|}
\,d \H^{N-1}(x),
\end{multline}
which shows~\eqref{derivf}. Equation~\eqref{intkappa} follows
from~\eqref{derivf} and the co-area formula for $BV$ functions.

Conversely, assume now that \eqref{intkappa} holds for all {$\f\in C^{\regu}_c(\R^N)$} and  $t_1<t_2$ such that $D \f\neq 0$
in the set $\{t_1\le \f \le t_2\}$.
We consider a family of diffeomorphism $\Phi_\e$ as in Definition~\ref{deffirstvar}.
We start by showing that \eqref{subgrad3} holds.  Write  $E$ as $E=\{\f\geq \frac12\}$ for a suitable $\f\in C^{\regu}(\R^N)$, constant out of a compact set
and with $D \f\neq 0$ in $\{0\le \f\le 1\}$.
Since $\Phi_\e(E)\Delta E$ is contained in the  $M\e$-neighborhood 
$(\pa E)_{M\e}$ of
 $\pa E$ for some $M>0$, if $\e>0$ is sufficiently small we may find a diffeomorphism $\widetilde \Phi_\e$ such that  $\widetilde \Phi_\e=Id$ outside $(\pa E)_{2M\e}$ (and in particular out of $\{|\f-1/2|\le 1/4\}$),  $\widetilde \Phi_\e(E)=\Phi_\e(E)$, and 
 { $\|\widetilde \Phi_\e -Id\|_{C^{\regu}}\to 0$} as $\e\to 0$.
In particular, by construction $ J ( \Phi_\e(E)) =J(\widetilde \Phi_\e(E))$.
 Since $\Phi_\e(E)=\widetilde \Phi_\e(E)=\{\f\circ \widetilde \Phi^{-1}_\e\geq \frac12\}$ and $\{\f\circ \widetilde \Phi^{-1}_\e\geq 1\}=
 \{\f\geq 1\}$, by \eqref{intkappa} we have
 \begin{align*}
 J ( \Phi_\e&(E))-J(\{\f\ge 1\})\\
&=\int_{\widetilde \Phi_\e(E)\setminus \{\f\ge 1\}}k(x, E_{\f\circ \widetilde \Phi^{-1}_\e(x)})\, dx \\
&=\int_{(\pa E)_{2M\e}\cap\widetilde{\Phi}_\e(E)} 
\left(k(x, E_{\f\circ \widetilde \Phi^{-1}_\e(x)})- k(x, E_{\f(x)}) \right)\, dx+\int_{\widetilde \Phi_\e(E)\setminus \{\f\ge 1\}}k(x, E_{\f(x)})\, dx\\
&\le\int_{(\pa E)_{2M\e}}
\left|k(x, E_{\f\circ \widetilde \Phi^{-1}_\e(x)})- k(x, E_{\f(x)})\right|\, dx+\int_{\Phi_\e(E)\setminus \{\f\ge 1\}}k(x, E_{\f(x)})\, dx\,.
 \end{align*}
 Since  $\|k(x, E_{\f\circ \widetilde \Phi^{-1}_\e(x)})- k(x, E_{\f(x)})\|_{L^{\infty}((\pa E)_{2M\e})}\to 0$ as $\e \to 0$, 
 thanks to Assumption C), we have that
 $$
 \int_{(\pa E)_{2M\e}}\bigl|k(x, E_{\f\circ \widetilde \Phi^{-1}_\e(x)})- k(x, E_{\f(x)})\bigr|\, dx =o(\e)\,.
 $$
 Therefore 
 \begin{multline*}
 \frac{d}{d \e} J \big( \Phi_\e(E)\big)_{|_{\e=0}} =
 \frac{d}{d \e}\biggl( \int_{\Phi_\e(E)\setminus \{\f\ge 1\}}k(x, E_{\f(x)})\, dx \biggr)_{|_{\e=0}}
\\ =\int_{\pa E}k(x, E_{\f(x)})\psi(x)\cdot\nu_E(x)\, d\H^{N-1}\,.
\end{multline*}

It now remains to prove~\eqref{intW}. 
We first consider a $C^{\regu}$ function $\psi$ such that $\psi-\f$ is compactly
supported in $\{ t_1<\f<t_2\}$. In particular if $\e>0$ is small enough,
$\{ t_1<\f+\e(\psi-\f) <t_2\}=\{ t_1<\f<t_2\}$.
We also introduce $\tilde\f:= t_1\vee(\f\wedge t_2)$ and 
$\tilde\psi(x):=\psi(x)$ if $t_1<\f<t_2$,
$\tilde{\psi}(x)=\tilde{\f}(x) \in \{t_1,t_2\}$ else.
Observe that thanks to~\eqref{visintin},
\begin{equation}\label{equazionestar}
J(\tilde\f) =  \int_{t_1}^{t_2} J(\{\f\ge s\})\,ds
\end{equation}
(which is finite thanks to~\eqref{intkappa}), so that (for $\e$ small)
\[
\frac{J(\tilde\f+\e(\tilde{\psi}-\tilde\f))-J(\tilde\f)}{\e}=
\int_{t_1}^{t_2}
\frac{J(\{\f+\e(\psi-\f)\geq s\} )-J(\{\f\ge s\}))}{\e}\,ds.
\]
Again, for a fixed $s\in (t_1,t_2)$,
one can find a family of $C^{\regu}$-diffeomorphisms $(\Phi_\e)_{\e>0}$ which
transform $\{\f>s\}$ into $\{\f+\e(\psi-\f)>s\}$. One proceeds
as before, but now the implicit function theorem is applied
 to the function
\[
(x,\e,h)\mapsto (1-\e)\f(x+h\nu_{\tilde{E}}(x)) + \e\psi(x+h\nu_{\tilde{E}}(x)) - s.
\]
for $x\in\partial\tilde{E}$, $\e$ small enough and $h$ in a suitable
neighborhood of $0$. The diffeomorphisms are defined as in \eqref{num},
and we can compute again the derivative of $h_\e$ with respect to $\e$,
for $x\in\partial\tilde{E}$:
\[
\frac{\partial h_\e(x)}{\partial \e} = 
\frac{\f(x+h_\e(x)\nu_{\tilde{E}}(x))-\psi(x+h_\e(x)\nu_{\tilde{E}}(x))}
{((1-\e)D \f(x+h_\e(x)\nu_{\tilde{E}}(x))+\e(D \psi(x+h_\e(x)\nu_{\tilde{E}}(x))))\cdot \nu_{\tilde{E}}(x)}. 
\]
Hence, at $\e=0$, for $x\in\partial\{\f\ge s\}$ we have
\[
\frac{\partial h_\e}{\partial \e}_{|_{\e=0}}(\pi_{\tilde{E}}(x)) = 
\frac{\f(x)-\psi(x)}
{D \f(x)\cdot \nu_{\tilde{E}}(\pi_{\tilde{E}}(x))}. 
\]
In turn, 
\[
\psi(x)=\lim_{\e\to 0}\frac{\Phi_\e(x)-x}{\e}=
\frac{\partial h_\e}{\partial \e}_{|_{\e=0}}(\pi_{\partial \tilde{E}}(x))
\nu_{\tilde{E}}(\pi_{\partial \tilde{E}}(x))
=\frac{(\f(x)-\psi(x))\nu_{\tilde{E}}(\pi_{\partial \tilde{E}}(x))}
{D \f(x)\cdot \nu_{\tilde{E}}(\pi_{\tilde{E}}(x))}
\]
and
\[
\psi(x)\cdot \nu_{\partial\{\f\ge s\}}(x)
=\frac{\psi(x)-\f(x)}{|D \f(x)|}\,.
\]
 Using~\eqref{subgrad3}, we infer
\begin{multline*}
\lim_{\e\to 0}
\frac{J(\{\f+\e(\psi-\f)\ge s\} )-J(\{\f\ge s\}))}{\e}
\\= \int_{\partial \{\f\ge s\}} \frac{\psi(x)-\f(x)}{|D \f(x)|} \kappa(x,E_{\f(x)})\,d\H^{N-1}(x)\,.
\end{multline*}
Notice that more generally, if $\e$ is small enough, one obtains that
\begin{multline*}
\lim_{\e'\to \e}
\frac{J(\{\f+\e'(\psi-\f)\ge s\} )-J(\{\f+\e(\psi-\f)\ge s\}))}{\e'-\e}
\\= \int_{\partial \{\f+\e(\psi-\f)\ge s\}}
\frac{\psi(x)-\f(x)}{|(1-\e)D \f(x)+\e D\psi(x)|} \kappa(x,\{\f+\e(\psi-\f)\ge s\})\,d\H^{N-1}(x)\,.
\end{multline*}

Denoting $E_{s,\e}$ the sets $\{\f+\e(\psi-\f)\ge s\}$ and observing
that they are continuous in $\Reg$ as $\e$ varies, we see that,
thanks to assumption C),
this derivative is continuous with respect to $\e$ (small)
and in particular,
\begin{multline}\label{preciseincrement}
\frac{J(\{\f+\e(\psi-\f)\ge s\} )-J(\{\f\ge s\}))}{\e}
\\=\frac{1}{\e}\int_0^\e
 \int_{\partial E_{s,t}}
\frac{\psi(x)-\f(x)}{|(1-t)D \f(x)+tD\psi(x)|} \kappa(x,E_{s,t})
\,d\H^{N-1}(x) dt
\end{multline}
Another observation is that the range of the $\e$ for which this is
true can be taken to be the same for all $s\in [t_1,t_2]$, since
it depends on  $C^2$ bounds for the boundaries  $\pa\{\f\ge s\}$
(more precisely, on their largest curvature)
and for $\psi$.
Hence, if $\e$ is small enough, integrating~\eqref{preciseincrement}
for $s$ between $t_1$ and $t_2$ and using the co-area formula
for $BV$ functions, we obtain that
\begin{multline*}
\frac{J(\tilde\f+\e(\tilde{\psi}-\tilde\f))-J(\tilde\f)}{\e}
\\=
\int_{\{ t_1<\f<t_2\}} (\psi(x)-\f(x))
\left( \frac{1}{\e}\int_0^\e \kappa(x,E_{(1-t)\f(x)+t\psi(x),t})\,dt\right)\,dx
\end{multline*}
In the limit, using the continuity assumption C) of $\kappa$ again,
we conclude that
\[
\lim_{\e\to 0}\frac{J(\tilde\f+\e(\tilde{\psi}-\tilde\f))-J(\tilde\f)}{\e}=
\int_{\{ t_1<\f<t_2\}} (\psi(x)-\f(x))\kappa(x,E_{\f(x)})\,dx\,.
\]
The convexity of $J$ in turn implies that
$J(\tilde\f+\e(\tilde{\psi}-\tilde\f))-J(\tilde\f)
\le \e(J(\tilde{\psi})-J(\tilde\f))$ and it follows
\begin{equation}\label{Jsubgrad}
J(\tilde{\psi}) \ge J(\tilde{\f})
+ \int_{\{ t_1<\f<t_2\}} (\psi(x)-\f(x))\kappa(x,E_{\f(x)})\,dx\,.
\end{equation}
Finally, notice that  if $\psi-\f$ only vanishes
on $\{\f\le t_1\}\cup\{\f\ge t_2\}$, instead of having compact support in $\{t_1<\f<t_2\}$,
then~\eqref{Jsubgrad} still holds, since by our assumptions
one can find $t'_1<t_1$, $t'_2>t_2$ such that $|\nabla \f|>0$ in 
$\{t'_1<\f<t'_2\}$.

Consider now $W$ as in Proposition~\ref{propsubgrad}.
Without loss of generality (possibly replacing
$t_1$ with a smaller value and $t_2$ with a larger value) we may assume
that $\partial W\subset \{s_1<\f<s_2\}$, where $t_1<s_1<s_2<t_2$.
Introducing a mollifier as in Lemma~\ref{lemapprox}, we can
approximate $\psi=t_1+(t_2-t_1)\chi_W$ with a sequence of smooth functions 
$\psi_n$ such that $\psi_n-\f$ is supported in $\{t_1\le\f\le t_2\}$,
and $\lim_n J(\psi_n)=J(\psi)=(t_2-t_1)J(W)$.

It follows from~\eqref{Jsubgrad} that
\[
J(\psi_n) \ge J(\tilde{\f})
+ \int_{\{ t_1<\f<t_2\}} (\psi_n(x)-\f(x))\kappa(x,E_{\f(x)})\,dx\,,
\]
and in the limit we obtain that
\[
(t_2-t_1)J(W) \ge J(\tilde{\f})
+ \int_{\{ t_1<\f<t_2\}} (\psi(x)-\f(x))\kappa(x,E_{\f(x)})\,dx\,,
\]
where we recall that $\psi=t_1+(t_2-t_1)\chi_W$.
We can conclude using~\eqref{intkappa} and \eqref{equazionestar}, or observing that
this inequality also implies that 
\[
(t_2-t_1)J(W) \ge J(g_n(\f))
+ \int_{\{ t_1<\f<t_2\}} (\psi(x)-g_n(\f(x)))\kappa(x,E_{\f(x)})\,dx\,.
\]
where $g_n:\R\to [0,1]$ is a smooth, nondecreasing approximation
of $t_1+(t_2-t_1)\chi_{\{t\ge t_2\}}$. Passing to the limit, we obtain~\eqref{intW}.
\end{proof}
%
%
\begin{corollary}\label{sdefofk}
Let $\kappa$ be a first variation of $J$ in the sense 
of Definition~\ref{deffirstvar}, and assume it satisfies assumption C). 
Then, it is also the curvature in the sense of Definition \ref{defofk}.
\end{corollary}
\begin{proof}

We need to  prove that \eqref{ineqbelow} and \eqref{ineqabove}  hold with $\kappa^+ = \kappa^- = \kappa$. To this purpose,  fix $x\in \partial E$ and let $\{W_n\} \subset \Ins$ with $|W_n\cap  E|>0$ and $\overline{W_n} \kto \{x\}$.

We can as before assume  that $E$ is 
the level set $1/2$ of a $C^{\regu}$ function $\f$, constant out of a compact set,
such that $D \f\neq 0$ in $\{0\le \f\le 1\}$.
Since 
\[
J(E) = J(\{\f\ge 1\})+\int_{\{1/2< \f < 1\}} \kappa(x,E_{\f(x)})\,dx
\]
and
\[
J(E\setminus W_n) \ge J(\{\f\ge 1\})+\int_{(E\setminus W_n)\setminus \{\f\ge 1\}} \kappa(x,E_{\f(x)})\,dx,
\]
by~\eqref{intkappa} and~\eqref{intW} respectively, 
it follows that 
\[
J(E)-J(E\setminus W_n) \le \int_{W_n \cap E} \kappa(x,E_{\f(x)})\,dx
\]
for $n$ sufficiently large.
Dividing both sides by $|W_n\cap E|$ and letting $n \to \infty$, using also the continuity property of $\kappa$ we conclude $\kappa^-\le \kappa$. The opposite inequality easily follows by \eqref{subgrad3}, choosing a sequence $\eta_n$
of smooth cut-off functions whose support concentrates around $x$ and which are
1 in a neighborhood of $x$, defining $\Phi_{\e,n}:y\mapsto y-\e\eta_n(y)\nu_{E}(x)$, and setting $W_n:= \Phi_{\e_n,n}(E)\setminus E$, where $\e_n$ is chosen through a  standard diagonal argument. The proof that $\kappa^+=\kappa$ is analogous.  
\end{proof}

\section{Examples of  perimeters and their curvature}\label{sec:examples}
In this part we  present some examples of generalized perimeters and corresponding curvatures that fit into our theory. We will consider here,
unless otherwise stated, that $\ell=2$, $\beta=0$.

\subsection{The Euclidean perimeter}
Let $J$ be the Euclidean perimeter. More precisely, let $J$ be its lower semi-continuous extension to measurable sets 
introduced by Caccioppoli and De Giorgi.
Then, $J$ satisfies all the assumptions i)-vi). 
Moreover, let $\kappa$ be the standard Euclidean curvature, i.e., the sum of the principal curvatures of $\partial E$ at $x$. 
It is standard that it is the first variation of the perimeter
in the sense of Definition~\ref{deffirstvar}, hence by
Propositions \ref{propsubgrad}  we deduce that the Euclidean curvature $\kappa$ is also the curvature of $J$ in the sense of Definition \ref{defofk}.

Clearly, the Euclidean perimeter is also uniformly continuous with respect to $C^2$ inner variations of  sets, namely it satisfies the continuity assumption C'). More in general, any local curvature $\kappa(x,E)$ that depends continuously on the normal and on the second fundamental form of $E$ at $x$ fits with our theory. For such local and posslbly anisotropic curvatures, we recover the well known existence and uniqueness of a viscosity solution to the  geometric flows.

\subsection{The fractional mean curvature flow}
Let $\alpha \in (0,\frac12)$, and let $E\in \Ins$ be such that $\chi_E \in  H^{\alpha}(\R^N)$.
We recall that   the $\alpha-$fractional seminorm of the characteristic function of  $E$ is defined by 
$$
[\chi_E]_{H^\alpha}  :=
\left((1- \alpha) \int_{\R^N\times\R^N} \frac{|\chi_E(x)-\chi_E(y)|}{|x-y|^{N+2 \alpha}}\,dxdy\right)^{\frac12}.
$$
Let us introduce the generalized perimeter 
$$
J(E) :=
\begin{cases}
\displaystyle
[\chi_E]^2_{H^\alpha}   & \text{ if } \chi_E \in H^{\alpha}(\R^N),\\
+\infty & \text{ otherwise.}
\end{cases}
$$

Since the work in  \cite{CRS}, this nonlocal perimeter has attracted much attention; we refer the interested reader to \cite{V}. 
It is easy to check that  $J$ satisfies all the properties i)-vi), so that it fits with our notion of generalized perimeters. 

A  notion of curvature corresponding to $J$ has been introduced in \cite{CS}, \cite{I}:
let 
$$
\rho(x):=1/|x|^{N+2\alpha}, \qquad \rho_\delta(x)=(1-\chi_{B(0,\delta)}(x))\rho(x).
$$
Then, for every $E\in\Reg$ set
\begin{align*}
& \kappa_\delta(x,E)\ =\ -2(1-\alpha)\int_{\R^N} (\chi_{E}(y)-\chi_{\R^N\setminus E}(y))
\rho_\delta(x-y)\,dy, 
\\
& \kappa(x,E):= \lim_{\delta\to 0} \kappa_\delta(x,E).
\end{align*}
The curvature $\kappa$ is well defined for all smooth sets,
 it is the first variation of the perimeter $J$  and   it is continuous with respect to $C^2$ convergence. 
By 
Proposition~\ref{propsubgrad} 
we deduce that  $\kappa$  satisfies \eqref{intkappa} and \eqref{intW}. 
In particular, $\kappa$ is the curvature of $J$ according with Definition \ref{defofk}.

 Indeed $\kappa(x,E)$ is well defined  for any set which satisfies an internal and external ball condition at $x$ (see \cite{I}, \cite{CS}). In particular, $\kappa$ is well defined for any $E\in \Regm$.
This suggests that $\kappa$ is a first order curvature.
Let us show that this is the case.

\begin{proposition}
Let $\Sigma\in \Regm$, let  $x\in \pa \Sigma$, and let $(p,X)$ and $(p,Y)$ be elements
of $\Jet^{2,+}_\Sigma(x)$ and $\Jet^{2,-}_\Sigma(x)$, respectively. 
Then,
\begin{equation}\label{eq:FOvero}
\kappa_*(x, p, X, \Sigma)=\kappa^*(x, p, Y, \Sigma) = \kappa(x,\Sigma).
\end{equation}
In particular, 
the curvature $\kappa$ satisfies the first order curvature assumption (FO). 
\end{proposition} 
\begin{proof}
By Lemma \ref{lemmalsc} there exists a sequence $(p_n,X_n,E_n)\to (p,X,\Sigma)$ with uniform superjet at $x$, with $E_n\in\Reg$, such that 
\begin{equation}\label{cialtron0}
\kappa(x,E_n) \to \kappa_*(x, p, X, \Sigma).
\end{equation}
Moreover, we can always assume that $E_n \to E$ in $L^1$ (see Remark \ref{mds}). 
Since $\Sigma\in\Regm$, there exists $r>0$ such that $B_r(x+r \frac{p}{|p|})\subset \Sigma$.
Set
$$
\tilde E_n:= E_n \cup B_r(x+r \frac{p_n}{|p_n|}).
$$
Clearly, $(p_n,X_n,\tilde E_n)$ still converge to  $(p,X,\Sigma)$ with uniform superjet at $x$. By the lower semicontinuity and monotonicity properties of $\kappa$ we have 
\begin{equation}\label{still}
 \kappa_*(x, p, X, \Sigma) \le \liminf_n \kappa(x, \tilde E_n) \le \liminf_n \kappa(x,  E_n)  = \kappa_*(x, p, X, \Sigma).
\end{equation}
Moreover, since  $\tilde E_n\to \Sigma$ in $L^1$ and $\tilde E_n$ satisfy a uniform internal and external ball condition at $x$, it is easy to see
(see for instance ~\cite{I}) that 
$
\kappa(x,\tilde E_n)\to \kappa(x,E),
$
which together with \eqref{still} proves that 
$
\kappa_*(x, p, X, \Sigma)= \kappa(x,\Sigma).
$
The proof for $\kappa^*(x, p, Y, \Sigma)$ is identical.  
\end{proof}

Once proved that $\kappa$ is a first order curvature, in view of Theorem \ref{th:CP1st} we recover the existence and uniqueness of a viscosity solution to the geometric flow, already proved in  \cite{I}.
Instead, the convergence of the corresponding minimizing movement scheme studied in Section~\ref{sec:ATW} { is completely new  for this class of nonlocal perimeters and furnishes an approximation algorithm which is alternative to the threshold-dynamics-based one studied in \cite{CS}.} 

\begin{remark}
{\rm
It can be proved that in fact $\kappa$ satisfies also the uniform continuity assumption C'). Thus,  uniqueness  could also be deduced from the second order theory of Subsection~\ref{subsec:2ndorder}, but of course the "first order" point of view is more convenient and straightforward in this case.
}
\end{remark}

We conclude this part giving a self contained proof  of \eqref{intkappa} and \eqref{intW}, which, in view of Proposition \ref{sdefofk},  yield  that $\kappa$ is the first variation of $J$.  

Let
$$
J_\delta(E) := (1- \alpha) \int_{\R^N\times\R^N}  |\chi_E(x)-\chi_E(y)| \rho_\delta(x-y)\,dxdy.
$$
We will first show that $J_\delta$, $\kappa_\delta$  satisfy \eqref{intkappa} and \eqref{intW}. Let $W$ be a bounded measurable set. Then,
\begin{multline}\label{eq:subsub}
\int_W  \kappa_\delta(x,E_{\f(x)})\,dx\,
\\
=\,
- 2 (1-\alpha) \int_{\R^N \times \R^N} \chi_{W}(x) (\chi_{E_{\f(x)}} (y) - \chi_{\R^N\setminus E_{\f(x)}}(y))
\rho_\delta(x-y)\,dy dx
\\
=\,
- (1-\alpha)  \int_{\R^N \times \R^N} 
(\chi_{W}(x) 
-
\chi_{W}(y) )
(\chi_{E_{\f(x)}} (y) - \chi_{\R^N\setminus E_{\f(x)}}(y))
\rho_\delta(x-y)\,dy dx
\\
\le\ 
\int_{\R^N\times \R^N}
|\chi_{W}(x) - \chi_{W}(y) | \rho_\delta(x-y)\,dy dx,
\end{multline}
with equality if and only if
$W = E_{s}$ for some $s\in (0,1)$.

On the other hand, it is easy to show that, as $\delta\to 0$, the following limits hold:
\begin{itemize}
\item[i)] For every $E\in\Ins$, $J_\delta(E) \to J(E)$,
\item[ii)]
Let $\f:\R^N\to (0,1)$ be such that $E_s:=\{\f \ge s\}$ are bounded and belong to $\Reg$ for every $s\in(0,1)$
and $D^2\f$ is negative definite on $\{\f=1\}$.
Then, $\kappa_\delta(x,E_{\f(x)}) \to \kappa(x,E_{\f(x)})$ in $L^1_{loc}(\R^N)$.
\end{itemize}
This implies that also $J$, $\kappa$ satisfies \eqref{intW} and \eqref{intkappa}.



\subsection{General two body interaction perimeters}
More in general one may consider a class of integral nonlocal perimeters of the form (see \cite{I})
\begin{equation}\label{pikappa}
J_K(E):=\int_E\int_{\R^N\setminus E}K(x-y)\,dxdy\,,
\end{equation}
where the (possibly singular) nonnegative kernel $K$ satisfies:
\begin{itemize}
\item[(i)] $K\in L^1(\R^N\setminus B(0,\de))$ for all $\de>0$;
\item[(ii)] for all $r>0$ and $e\in \mathbb{S}^{N-1}$ we have that $K\in L^1(\{z\in \R^N:\, r|z \cdot  e |\leq |z- (z \cdot e) e|^2\})$.
\end{itemize}
The associated nonlocal curvature 
$$
\kappa(x,E):=-2\int_{\R^N} (\chi_{E}(y)-\chi_{\R^N\setminus E}(y))
K(x-y)\,dy
$$
is well defined in the principal value sense provided that $E$ satisfies both an inner and an outer ball condition at the point $x\in \pa E$.
One can check that also these curvatures are covered by  both the first order and second order theories of generalized curvatures. 

\subsection{The flow generated by the regularized pre-Minkowski content }\label{parmino}
Let $\rho>0$ be fixed, and consider the measure of the $\rho$-neighborhood of the boundary of $ E$, i.e., 
\begin{equation}\label{defmink0}
\mathcal M_\rho(E):= |(\partial E)_\rho| = |(\cup_{x\in\partial E }B_\rho(x))|.
\end{equation}
We  refer to $M_\rho$ as the {\it pre-Minkowski content} of $\partial E$, since as $\rho\to 0$,  ${|(\partial E)_\rho| }/{2\rho}$ approximates the Minkowski content, which coincides with the standard perimeter on smooth sets. 

An issue with definition~\eqref{defmink0} is that it depends on
the choice of the representative  within the Lebesgue equivalence class of the set $E$. 
For this reason, one may introduce the following variant:
\begin{equation}\label{defmink}
J_\rho(E)\ =\ \frac{1}{2\rho}\int_{\R^N} \osc_{B(x,\rho)}(\chi_E)\,dx
\end{equation}
where $\osc_A(u)$ denotes the \textit{essential oscillation} of the
measurable function $u$ over a measurable set $A$, defined by
$\osc_A(u)\ =\ \textup{ess} \sup_A u\,-\, \textup{ess} \inf_A u$.
One checks that $J_\rho(E)$ coincides with the measure of the $\rho$-neighborhood
of the  essential boundary of $E$.
Moreover, 
$J_\rho(E)\ =\ \inf \{ \M_\rho(E'): \,  |E\triangle E'| = 0 \}$, 
where $E\triangle E'$ denotes the symmetric difference $(E\setminus E') \cup (E'\setminus E)$.

In \cite{CMP0} we have proved that the functional \eqref{defmink} is a generalized perimeter, we have introduced the corresponding curvature, and studied the geometric flow.  
Let us  introduce  a notion of curvature corresponding to $J_\rho$;  let $E\in\Reg$, and  denote by 
$\nu_E(x)$ the outer normal unit vector to $\partial E$ at $x$.

For $x\in\partial E$, set
\begin{equation}
\kappa_\rho(x,E)\ =\ \kappa^{out}_\rho(x,E)\,+\,\kappa^{in}_\rho(x,E),
\end{equation}
where
\begin{equation}\label{kappaprho}
\kappa^{out}_\rho(x,E)\ =\begin{cases}
\ds \frac{1}{2\rho}\det (I+\rho D \nu_E(x))\  &
\textrm{ if }  \dist(x+\rho \nu_E(x),E)= \rho\,, \\[2mm]\\
0 & \textrm{ otherwise,}
\end{cases}
\end{equation}
\begin{equation}
\kappa^{in}_\rho(x,E)\ =\begin{cases}
\ds - \frac{1}{2\rho}\det (I-\rho D \nu_E(x)) &
\textrm{ if }  \dist(x -\rho \nu_E(x),E^c)= \rho\,, \\[2mm]\\
0 & \textrm{ otherwise.}
\end{cases}
\end{equation}
These quantities  correspond to the variation
of the volume of the strips $\{0<d_E<\rho\}$ (for $\kappa^{out}_\rho$)
and $\{-\rho<d_E<0\}$ (for $\kappa^{in}_\rho$) when the boundary is infinitesimally
modified at $x$, and their sum is a natural candidate for the
curvature associated to the energy $\E_\rho$.
Indeed, in \cite{CMP0} we have proved that $\kappa_\rho(x,E)$ is the first variation of  $J_\rho$ (in the classical sense \eqref{subgrad3}) whenever  $E\in\Reg$ is such that the points at distance $\rho$ from $\partial E$ admit a unique projection on $\partial E$ (indeed such  condition can be weakened a little). 
In order to have a well defined curvature for all $E\in\Reg$, one can consider  the following regularization of $J_\rho$:
\begin{equation*}\label{varmr}
J^f(E)\ =\ \int_{\R^N} f(d_E(x))\,dx = \int_0^{\rho} (-2s f'(s)) J_s(E) \,ds,
\end{equation*}
where   $d_E$ is the signed distance from $\partial E$  and  $f:\R\to \R_+$ is even, smooth and decreasing in  $\R_+$, with support in $[-\rho,\rho]$. 
Such a regularization was considered also in \cite{BKLMP} for numerical purposes.

The corresponding curvature  $\kappa_f$ is 
\begin{equation}\label{kappaf}
\kappa_f(x,E) = \kappa^{out}_f(x,E) + \kappa^{in}_f(x,E), 
\end{equation}
where
$$
\kappa^{out}_f(x,E)   = \int_0^\rho (- 2s f'(s)) \kappa^{out}_s(x,E) \,ds, \quad \kappa^{in}_f(x,E)   = \int_0^\rho (- 2s f'(s)) \kappa^{in}_s(x,E) \,ds.
$$

Let $r^{in}$ be the maximal  $r\in [0,\rho]$ such that $E$ satisfies the internal  ball condition with radius $r$ at $x$, and let $r^{out}$ be defined analogously. Clearly, $r^{in}$, $r^{out}$ and the second fundamental form at $x$ are uniformly continuous with respect to smooth inner variations. We immediately
deduce that  $\kappa_f$ satisfies the uniform continuity assumption C'). 

In \cite{CMP0} we have  proved that $\kappa_f(x,E)$
is the curvature corresponding to $J^f$, according to
both Definitions~\ref{deffirstvar} and~\ref{defofk}, and we have studied the corresponding curvature flow through the minimizing movements method. 
As a consequence of the analysis of this paper, namely by the Comparison Principle provided by Theorem~\ref{th:CP2nd},  we  get { the new result that  such a geometric evolution is indeed unique}.  

\subsection{The shape flow generated by $p$-capacity}\label{subsec:capacity}

In this subsection we show that the shape flow of {\em bounded sets} generated by the $p$-capacity fits into our general framework. Notice that the case  $p=2$ yields an evolution that is similar to the Hele-Show type flow considered in \cite{CL07}.  

 To this aim, given $1<p<N$, we consider the following  {\em relaxed $p$-capacity} of a set $E\in\Ins$ defined by
\begin{equation}\label{cap}
\cp(E):=\inf\biggl\{\int_{\R^N}|D w|^p\, dx:\, w\in K^p \text{ and }w\geq1 \text{ a.e. in }E  \biggr\}\,,
\end{equation}
where $K^p$ stands for the subspace of functions $w$ of $L^{p^*}(\R^N)$ such that  $Dw\in L^p(\R^N)$. Note that the above definition departs from the standard one in which the condition $w\geq1 \text{ a.e. in }E$ is replaced by $E\subset\inter {\{w\geq 1\}}$. It may be thought as a sort of $L^1$-lower semicontinuous envelope of the standard $p$-capacity, having the property of being insensitive to negligible sets and thus independent of the Lebesgue representative of $E$. 
Clearly the two definitions coincide on open sets and it is not difficult to check that they also agree on all closed sets $F$ such that 
$F=\overline{\inter F}$, with $|\pa F|=0$, in particular on all sets in $\Reg$.

Formula \eqref{cap} does not provide yet a generalized perimeter. Indeed, $\cp(\R^N)=+\infty$ and, more in general, if $E\in \Reg$, then 
$\cp(E)<+\infty$ if and only if $\cp(\R^N\setminus E)=+\infty$. Thus, the requirements  i) and ii) stated at the begining of Subesction~\ref{subsec:gp} are not fulfilled. On the other hand, properties iii) and vi) are evident, the lower-semicontinuity  iv) follows in a standard way, while 
  the submodularity property v) can be proven as in the case of the standard capacity (see \cite[Theorem 2-(vii) of Section 4.7]{EG}). 
  Since our focus will be on the evolution of bounded sets, we will build a generalized perimeter $J_p$, by enforcing the following properties:
  \begin{itemize}
 \item[a)] $J_p(E)=\cp(E)$ for all bounded sets $E\in \Ins$;
 \item[b)] $J_p(E)=J_p(\R^N\setminus E)$ for all $E\in \Ins$.
  \end{itemize} 
 This is achieved by setting
  \begin{equation}\label{jp}
  J_p(E):=\min\{\cp(E), \cp(\R^N\setminus E)\}
  \end{equation}
  for all $E\in \Ins$. It follows immediately from the definition and from the properties of $\cp(\cdot)$ recalled above that $J_p$ satisfies
  i)--iv) of Subsection~\ref{subsec:gp} and the translation invariance vi). It only remains to check the submodularity property v). To this aim, let us consider the case of two sets $E$, $F\in \Ins$ such that $\cp(E)<+\infty$ and $\cp(\R^N\setminus F)<+\infty$. As   $\cp(\R^N\setminus E)=\cp(F)=+\infty$, we have  $J_p(E)=\cp(E)$ and 
  $J_p(F)=\cp(\R^N\setminus F)$. Moreover, since $\cp(E\cap F)\leq \cp (E)$, we also have $J_p(E\cap F)= \cp(E\cap F)$, while the fact that 
$\cp(E\cup F)\geq \cp(F)=+\infty$ implies $J_p(E\cup F)=\cp(\R^N\setminus (E\cup F))$.  Thus, in this case  the submodularity inequality is equivalent to 
$$
\cp(\R^N\setminus (E\cup F))+ \cp(E\cap F)\leq \cp(E)+\cp(\R^N\setminus F)\,,
$$
which is obviously true since $\cp(\R^N\setminus (E\cup F))\leq \cp(\R^N\setminus F)$ and $\cp(E\cap F)\leq \cp(E)$ by the non-decreasing monotonicity of the set function $\cp(\cdot)$.
 Since all the remaining cases are either trivial or reduce the the submodularity of $\cp(\cdot)$, also property v) is established for $J_p$, which is therefore a generalized perimeter.

By a standard application of the Direct Method of the Calculus of Variations one may also check the existence of a {\em unique capacitary potential} $w_E$ associated with any set $E$, i.e., of a unique solution to the problem \eqref{cap}, whenever $\cp(E)<+\infty$. The Euler-Lagrange conditions for \eqref{cap} easily yield that  $w_E$ is super $p$-harmonic in $\R^N$,  in fact it is determined as the unique solution $w_E\in K^p$ to 
\begin{equation}\label{wEconditions}
\begin{cases}
\displaystyle -\int_{\R^N}|Dw_E|^{p-2}Dw_ED\f\, dx\geq 0 & \text{for all $\f\in K^p$, with $\f\geq 0$ a.e. on $E$.}\\
 w_E = 1 &\text{a.e. in $E$.}
 \end{cases}
 \end{equation}
Denoting by $E^{(0)}$ the set of points with vanishing density with respect to $E$, it follows in particular that {\em $w_E$ is p-harmonic in the interior 
of $E^{(0)}$}.
 
 In order to identify the nonlocal curvature corresponding to $J_p(\cdot)$, we exploit the theory developed in Section~\ref{secfirstvar}. Let  $E\in \Reg$ and bounded, and let 
 $(\Phi_\e)_\e$ be a one-parameter family of diffeomorphisms from $\R^N$ onto itself of class $C^2$ both in $\e$ and $x$ and such that
 $\Phi_0(x)=x$ for all $x\in \R^N$.  Denote $\psi(x):=\frac{\partial \Phi_\e(x)}{\partial \e}_{|_{\e=0}}$. Then by the Hadamard formulae (see for instance \cite{SZ})  one has
 \begin{multline}\label{capkappa}
\frac{d}{d_\e}J_p(\Phi_\e(E))_{|_{\e=0}}= \frac{d}{d_\e}\cp(\Phi_\e(E))_{|_{\e=0}}=\frac{d}{d_\e}\int_{\R^N\setminus\Phi_\e(E)}|D w_{\Phi_\e(E)}|^p\, dx_{\big|_{\e=0}}\\
 =\int_{\pa E}|D w_E|^{p}(x)\psi(x)\cdot\nu_E(x)\, d\H^{N-1}(x)\,, 
 \end{multline}
 where, as usual, $\nu_E$ denotes the outer unit normal to $E$.  Motivated by the above formula and recalling that 
 $J_p(E)=\cp(\R^N\setminus E)$ for $E\in \Reg$ and unbounded, 
   for every $E\in \Reg$ and  $x\in \partial E$ we set
 \begin{equation}\label{capkappa2}
 \kappa_p(x,E):=
 \begin{cases}
 \displaystyle \hphantom{-} |D w_E(x)|^p=\biggl|\frac{\partial w_E}{\partial\nu}(x)\biggr|^p & \text{if $E$ is bounded,}\\
 \\
\displaystyle- |D w_{\R^N\setminus E}(x)|^p=-\biggl|\frac{\partial w_{\R^N\setminus E}}{\partial\nu}(x)\biggr|^p & \text{if $E$ is unbounded.}
 \end{cases}
 \end{equation}
 Recalling that $w_E$ is p-harmonic on $E^c$ and satisfies the Dirichlet condition $w_E=1$ on $\pa E$, the well-established regularity theory for the $p$-Laplacian (see for instance \cite{Lieb}) yields that  {$w_E$ is of class $C^{1, \alpha}$} up to the boundary for all $\alpha\in (0,1)$, with 
 the $C^{0,\alpha}$-norm of $Dw_E$ depending only on its $L^p$-norm and the $C^{0,\alpha}$-norm of $\pa E$. 
{
In fact, whenever $E_n \to E$ in $C^{1,\alpha}$ and $x\in \partial E\cap\partial E_n$ we have
\begin{equation}\label{regures}
\frac{\partial w_{E_n}}{\partial\nu}(x)\to \frac{\partial w_E}{\partial\nu}(x)
\end{equation} 
 as $n\to\infty$. } 
 In particular, it follows that 
 the nonlocal curvature $\kappa_p$ defined in \eqref{capkappa2} satisfies the continuity property C) of Subsection~\ref{sec:curvature}. 

 In turn, by Corollary~\ref{sdefofk} the set function  \eqref{capkappa2} is the curvature associated with $J_p$ in the sense of Definition~\ref{defofk}.
 Lemma~\ref{lemmaxp} now implies that  the monotonicity property A) stated in Subsection~\ref{sec:curvature} holds for $\kappa_p$.

Since the translation invariance of $\kappa_p$ is evident, we have shown that \eqref{capkappa2} satisfies axioms A), B), and C) of 
Subsection~\ref{sec:curvature}. We recall that these axioms are enough to guarantee the convergence (up to subsequences) of the minimizing movements scheme  studied in Section~\ref{sec:ATW}  to a viscosity solution of the corresponding level set equation. 

It remains to investigate the uniqueness. Instead of establishing the reinforced continuity property C'), we check that 
the nonlocal perimeter $J_p$ generates a ``first-order'' flow and we apply the theory of Subsection~\ref{subsec:1storder}.
 To this aim, denote by $(\kappa_p)_*$ and $(\kappa_p)^*$ the lower and the upper semicontinuous extensions of $\kappa_p$ provided by formulas  \eqref{defkappal} and \eqref{defkappau}, respectively. Note that,  as a straightforward 
 consequence of the definition and of \eqref{defkappau}, we have
 \begin{equation}\label{kappa*compl}
 (\kappa_p)_*(x, p, X, E)=-(\kappa_p)^*(x, -p, -X, \R^N\setminus E)
 \end{equation}
for $E\in \Ins$, $x\in \pa E$, and $(p,X)\in \Jet^{2,+}_E(x)$. 

We are now in a position to prove that  condition (FO) of Subsection~\ref{subsec:1storder} is satisfied. Uniqueness will then follow by applying the Comparison Principle provided by Theorem~\ref{th:CP1st}.
\begin{lemma}\label{lm:k*c11}
Let $\Sigma\subset \R^N$ belong to $\Regm$. Let $x\in \pa \Sigma$ and let $(p,X)$ and $(p,Y)$ be elements
of $\Jet^{2,+}_\Sigma(x)$ and $\Jet^{2,-}_\Sigma(x)$, respectively. 
Then,
$$
(\kappa_p)_*(x, p, X, \Sigma)=(\kappa_p)^*(x, p, Y, \Sigma)\,.
$$
\end{lemma}
\begin{proof}
In light of \eqref{kappa*compl}, it is enough to consider the case of a bounded set $\Sigma$ of class $C^{1,1}$. Let $w_\Sigma$ be the associated
capacitary potential.  The conclusion of the lemma will be achieved by showing  that 
\begin{equation}\label{eq0}
(\kappa_p)_*(x, p, X, \Sigma)=\biggl|\frac{\partial w_\Sigma}{\partial\nu}(x)\biggr|^p=(\kappa_p)^*(x, p, Y, \Sigma)\,.
\end{equation}
To this aim, let  $E\supseteq\Sigma$ be a bounded set of $\Reg$ admissible for the Definition~\ref{defkappal} of $(\kappa_p)_*(x, p, X, \Sigma)$, and let $w_E$ be the corresponding capacitary potential. Recall that by \eqref{wEconditions}, we have that $w_E$ is super p-harmonic in 
$\R^N\setminus\Sigma$, while $w_\Sigma$ is p-harmonic in the same set. Since $w_E=w_\Sigma=1$ on  $\pa \Sigma$, by the Maximum Principle we infer that $1\geq w_E\geq w_\Sigma$ in $\R^N\setminus\Sigma$. In turn, 
$$
\kappa_p(x, E)=\biggl|\frac{\partial w_E}{\partial\nu}(x)\biggr|^p\leq \biggl|\frac{\partial w_\Sigma}{\partial\nu}(x)\biggr|^p
$$
and therefore we may conclude that 
\begin{equation}\label{ineq1}
(\kappa_p)_*(x, p, X, \Sigma)\leq \biggl|\frac{\partial w_\Sigma}{\partial\nu}(x)\biggr|^p\,.
 \end{equation}
 To show the opposite inequality, fix $\delta>0$ and  construct a sequence of sets $(\Sigma_n)_n\subset\Regm$ with the following properties:
 \begin{itemize}
 \item[i)] $\Sigma\subset\inter\Sigma_n\cup\{x\}$, with $\pa\Sigma_n\cap\pa\Sigma=\{x\}$;
 \item[ii)] $(p,X+\delta I)\in \Jet^2_{\Sigma_n}(x)$ for all $n\in \N$;
 \item[iii)] $\Sigma_n\to \Sigma$ in the $C^{1, \alpha}$-sense, for all $\alpha\in (0,1)$. 
 \end{itemize}
 Such a sequence can be constructed in many different ways: one possibility is to consider the $\frac1n$-level sets of the signed distance function from $\Sigma$ and modify them in the proximity to $x$ in order to fulfill conditions i) and ii).   
 {By \eqref{regures}}, for any given small $\e>0$ we may fix $\bar n$ such that
 \begin{equation}\label{ineq2}
 \biggl|\frac{\partial w_{\Sigma_{\bar n}}}{\partial\nu}(x)\biggr|^p\geq \biggl|\frac{\partial w_\Sigma}{\partial\nu}(x)\biggr|^p-\e\,.
 \end{equation}
 Recall now that from the proof of Lemma~\ref{lemmalsc}, we may construct a decreasing sequence of sets $(E_n)_n\subset\Reg$ such that
\begin{itemize}
\item[a)] $E_n \searrow \Sigma$ in the Hausdorff sense;
\item[b)] $ (p, X+\delta_n I)\in \Jet^2_{E_n}(x)$ for some $\delta_n\searrow 0$;
\item[c)] $\kappa_p(x, E_n) = \Bigl|\frac{\partial w_{E_n}}{\partial\nu}(x)\Bigr|^p\to (\kappa_p)_*(x, p, X, \Sigma)$.
\end{itemize}
Taking into account i) and ii) above, it follows from a) and b) that $E_n\subset\Sigma_{\bar n}$ for $n$ large enough. For all such $n$'s, by the Maximum Principle as in the first part of the proof, we have
$$
 \biggl|\frac{\partial w_{E_n}}{\partial\nu}(x)\biggr|^p\geq \biggl|\frac{\partial w_{\Sigma_{\bar n}}}{\partial\nu}(x)\biggr|^p\geq \biggl|\frac{\partial w_\Sigma}{\partial\nu}(x)\biggr|^p-\e\,,
$$
where in the last inequality we have used \eqref{ineq2}. {By c),} passing to the limit in the left-hand side of the above formula and by  the arbitrariness of $\e$, we deduce
$$
  (\kappa_p)_*(x, p, X, \Sigma)\geq \biggl|\frac{\partial w_\Sigma}{\partial\nu}(x)\biggr|^p,
$$
which, together with \eqref{ineq1},  establishes the first equality in \eqref{eq0}. The second equality can be proven in a completely analogous fashion.  \end{proof}

\section{The minimizing movements approximation}\label{sec:ATW}
In this section we implement the minimizing movements scheme to solve and approximate the nonlocal $\kappa$-curvature flow, 
in the spirit of  \cite{ATW,LS}. We extend the approach  of \cite{Chambolle} (see also \cite{CMP0}) to our general framework.

\subsection{The time-discrete scheme for bounded sets}

We start by introducing the incremental minimum problem. To this purpose,  given a bounded set $E\neq \emptyset$, we let
\begin{equation}\label{defsigndist}
d_E(x)\ =\ \dist(x,E)-\dist(x,\R^N\setminus E)
\end{equation}
be the signed distance function to $\partial E$. 
Fix a time step $h>0$ and consider the problem
\begin{equation}\label{varprob}
\min \left\{J(F)\ +\ \frac{1}{h}\int_F d_E(x)\,dx:\, F\in \Ins\right\}.
\end{equation}
Note that 
$$
\int_Fd_E(x)\, dx-\int_E d_E(x)\, dx=\int_{E\Delta F}\mathrm{dist}(x, \partial E)\, dx
$$
so that \eqref{varprob} is equivalent to 
$$
\min \left\{J(F)\ +\ \frac{1}{h}\int_{E\Delta F}\mathrm{dist}(x, \partial E)\, dx
:\, F\in \Ins\right\}.
$$

\begin{proposition}\label{exminpro}
The problem \eqref{varprob} admits a minimal and a maximal solution.
\end{proposition}
\begin{proof}
Since the functional $J$ extended to $L^1_{loc}$ according to \eqref{visintin} is convex, it is easy to check that the minimization problem
\begin{equation}\label{minext}
\min_{u\in L^\infty(\R^N;[0,1])} {J}(u)\,+\,\frac{1}{h}\int_{\R^N} u(x)d_E(x)\,dx
\end{equation}
 admits  a solution. Then,  observe that
\begin{equation}\label{minvisintin}
{J}(u)\,+\,\frac{1}{h}\int_{\R^N} u(x)d_E(x)\,dx\ =\ 
\int_{0}^1 \left( J(\{u>s\})\, +\,\frac{1}{h}\int_{\{u>s\}} d_E(x)\,dx
\right)\,ds\,,
\end{equation}
from which we  easily deduces that for a.e.~$s\in [0,1]$, $\{u>s\}$ is
a solution to~\eqref{varprob}.  Let  now $E_1$ and $E_2$ be two solutions to \eqref{varprob}. Then 
again by \eqref{minvisintin} their characteristic functions  and in turn, by convexity,  $\frac12({\chi_{E_1}+\chi_{E_2}})$ are solutions to \eqref{minext}.  
Since   almost all their  superlevel sets are solutions to \eqref{varprob}, we deduce, in particular,  that $E_1\cap E_2$ and $E_1\cup E_2$ are  solutions to  \eqref{varprob}. 
Finally let $E_n$ be a sequence of solutions to \eqref{varprob} such that 
$$
|E_n|\to m:=\inf\{|E|:\, E \text{ is a solution to \eqref{varprob}}\}.
$$
Then, $F_k:=\cap_{n=1}^{k}E_n$ is a decreasing sequence of solutions such that $|F_k|\to m$. Thus, by semicontinuity,
their $L^1$-limit $E:=\cap_{n=1}^{\infty}E_n$ is the minimal solution. The existence of a maximal solution can be proven analogously.
 \end{proof}

 { For any bounded set $E\neq \emptyset$ we let  $T^+_h E$ and  $T^-_h E$ denote  the maximal and the minimal solution of \eqref{varprob}, respectively. We also set $T_h^\pm\emptyset:=\emptyset$. We will mainly use minimal solutions, and write $T_h E:=  T^-_h E$. 
 This choice corresponds to consider open superlevels in our level set approach (see Proposition \ref{prop:solmin}).  
   It is convenient to fix a precise representative for $T^\pm_h E$. 
To this purpose, we will identify any measurable set with the representative given by the set of Lebesgue points of the characteristic function.} 


\begin{lemma}\label{lemcomp}
If $E\subseteq E'$, then ${T_h^\pm} E\subseteq T_h^\pm E'$.
\end{lemma}

\begin{proof}
The proof is classical and we just sketch it.
We  first assume that $E\subset\subset E'$, so that $d_E> d_{E'}$ everywhere. 
We compare  the energy~\eqref{varprob} of $T_h^\pm E$ with the one of $T_h^\pm  E \cap T_h^\pm E'$,
and the energy~\eqref{varprob}  (with $E$ replaced by $E'$)
of $T_h^\pm E'$ with the one of $T_h^\pm E \cup T_h^\pm E'$. 
We sum both inequalities and use~\eqref{eq:submo} to deduce that $T_h^\pm E \subseteq T_h^\pm E'$.

Now we conclude the proof by a perturbation argument. 
For $\e>0$ let $F_\e$ be the minimal solution of \eqref{varprob} with $d_E$ replaced by $d_E+\e$.
Arguing as before, we deduce that  $F_\e$ are increasing in $\e$ and  $F_\e \subseteq T_h^- E'$. 
Therefore  $F_\e \to F_0:= \cup_\e F_\e$ in $L^1_{loc}$. By lower semicontinuity it follows that $F_0$ is a solution, and thus 
$T_h^- E\subseteq F_0\subseteq T_h^-E'$. The inclusion $T_h^+ E\subseteq T_h^+E'$ can be proven similarly.
\end{proof}

\begin{remark}\label{rm:comp}
Let $f$ be a measurable function such that $f^-:= -f\land 0\in L^{1}(\R^N)$. 
Then one can argue as in Proposition~\ref{exminpro} to prove that the minimum problem
$$
\min\biggl\{J(F)+\int_F f\, dx:\, F\in \Ins\biggr\}
$$
admits a {\em minimal} and a {\em maximal} solution, denoted by $E^-_f$ and $E^+_f$ respectively. Moreover, arguing exactly as in the proof of Lemma~\ref{lemcomp}, one can show that if $f_1$, $f_2$ are measurable functions with $f_1^-$, $f^-_2 \in L^1(\R^N)$  and $f_1\leq f_2$ a.e., then
$$
E^\pm_{f_2}\subseteq  E^\pm_{f_1}\,.
$$
\end{remark}

\begin{lemma}\label{lemcomp2}
If $E+ B_R \subseteq E'$, then $(T_h^\pm E) + B_R \subseteq T_h^\pm E'$.
\end{lemma}

\begin{proof}
By Lemma \ref{lemcomp} for every $z\in B_R$ we have 
$T_h^\pm(E+z) \subseteq T_h^\pm E'$.
By translation invariance we conclude
$$
(T_h^\pm E) + B_R =  \bigcup_{z\in B_R} (T_h^\pm E) + z =  \bigcup_{z\in B_R} T_h^\pm(E+z) \subseteq T_h^\pm E'.
$$
\end{proof}

\begin{lemma}\label{lembd}
For any $R>0$ we have
$T_h^\pm (B_R) \subseteq B_{CR}$, where $C$ depends only on the dimension $N$.
%
\end{lemma}
\begin{proof}
By Lemma \ref{lemcomp2} we have
\begin{equation}\label{cle1}
T_h^\pm (B_R) +B_R\subseteq T_h^\pm (B_{2R}). 
\end{equation}

Let  $c>1$, and assume there exists $x\in  T_h^\pm (B_R)  \setminus B_{cR}$. 
Since in particular $x\in  T_h^\pm (B_R)$, by \eqref{cle1} we have $B(x,R)\subseteq T_h^\pm (B_{2 R})$. Hence
\begin{multline*}
0\ >\ J(T_h^\pm (B_{2 R}))\ +\ \frac{1}{h}\int_{T_h^\pm (B_{2 R})} |y|-2R \,dy
\\ \ge
\ \frac{1}{h} \left(\int_{B(x,R)} |y|-2R\,dy  + \int_{T_h^\pm (B_{2 R}) \setminus B(x,R)} |y| - 2R \,dy\right)
\\ \ge
\ \frac{1}{h} \left(\int_{B(x,R)} |y|-2R\,dy  + \int_{B_{2R}} |y| - 2R \,dy\right)
\\ \ge
\ \frac{1}{h} \left(\int_{B_R} (c-2)R - |y|\,dy  + \int_{B_{2R}} |y| - 2 R \,dy\right)
\end{multline*}
which is positive if $c$ is large enough (depending only on the dimension),
a contradiction.
\end{proof}

Next lemma provides  a more refined estimate.

\begin{lemma}\label{lembdh} 
Let $C>1$ be such that the statement  of Lemma \ref{lembd} holds,  and let $\uc$, $\oc$  be  as  in \eqref{defbarc}. 
Then, the following holds.
\begin{itemize}
\item[i)]
Let $R>0$. Then,  for every $h>0$ such that $R - h \uc (CR)>0$  we have $T_h^\pm B_R\subseteq B_{R - h \uc (CR)}$.
\item[ii)]
Let $R_0>0$ and $\sigma>1$ be fixed. Then, for  $h>0$ small enough (depending on $R_0$ and $\sigma$), we have $T_h^\pm B_R\supseteq B_{R-h\oc(R/\sigma)}$  for all $R\ge R_0$.
\end{itemize}

\end{lemma}

\begin{proof}
First, we know from the previous result that $T_h^\pm B_R\subseteq B_{C R}$.  \\
{\it Proof of i).}
We can always assume $T_h^\pm B_R\neq \emptyset$. 
Let
$\bar\rho=\sup\{ \rho \in [0,C R]\,:\, |T_h^\pm B_R\setminus B_\rho|>0\}$.
Let $\bar x\in\partial B_{\bar\rho}$ such that $|T_h^\pm B_R\cap B(\bar x,\e)|>0$ for
any $\e>0$, and let $\rho>\bar\rho$. Let 
$\tau\in\R^N$ be such that 
 $B(-\tau,\rho)\supset B_{\bar\rho}$
and
$\partial B(-\tau,\rho)$ is tangent to $\partial B_{\bar\rho}$ at $\bar x$; i.e., 
$\tau=(\rho/\bar\rho-1)\bar x$.

We let for $\e>0$ small $B^\e=B(-(1+\e)\tau,\rho)$ and $W^\e=T_h^\pm B_R\setminus
B^\e$. Notice that by construction $W^\e$  has positive measure and converges to $\bar x$ in the Hausdorff sense as $\e\to 0$.
By submodularity we have
\begin{equation}\label{susu}
J(B^\e\cap T_h^\pm B_R)+J(B^\e\cup T_h^\pm B_R)\ \le\ J(B^\e)+J(T_h^\pm B_R).
\end{equation}
By \eqref{susu} and using the minimality of $T_h^\pm B_R$ we have
\begin{multline*}
J(B(-\tau,\rho) \cup (W^\e + \e\tau) )-J(B(-\tau,\rho)) =  J(B^\e\cup W^\e)-J(B^\e)\\
 \le\ J(T_h^\pm B_R)-J(B^\e\cap T_h^\pm B_R)
\ \le\ -\frac{1}{h}\int_{W^\e} |x|-R\,dx\,.
\end{multline*}
 Dividing the previous inequality by $|W^\e|$ and  passing to the limit as $\e\to 0$, in view of the very definition  \eqref{ineqbelow} of $\kappa$
 we get
\[
\kappa(B(-\tau,\rho)) \le\ \frac{1}{h} (R-|\bar x|) = \frac{1}{h} (R- \bar \rho).  
\]
Recalling the definition of  $\uc$ and the fact that it is a continuous decreasing function, we deduce the thesis by sending $\rho\to \bar \rho$.
\\
{\it Proof of ii).}
Assume $x_0$ is such that 
$$
\bar\rho=\max\{
\rho>0\,:\, |B(x_0,\rho)\setminus T_h^\pm B_R|=0\}\in ]0, 2CR].
$$
As in the proof of i), we
can find $\bar x\in\partial B(x_0,\bar\rho)$ such that
$|B(\bar x,\e)\setminus T_h^\pm B_R|>0$ for any $\e>0$,
we fix $\rho<\bar\rho$ and set $\tau=(1-\rho/\bar\rho)(\bar x-x_0)$,
so that $\{\bar x\}=\partial B(x_0,\bar\rho)\cap\partial B(x_0+\tau,\rho)$.
We let $B^\e=B(x_0+(1+\e)\tau,\rho)$
and define $W^\e = B^\e\setminus T_h^\pm B_R$. By submodularity we have
\[
J(B^\e\cap T_h^\pm B_R) + J(B^\e \cup T_h^\pm B_R)\ \le\ J(B^\e)+J(T_h^\pm B_R).
\]
Using the minimality of $T_h^\pm(B_R)$ we deduce
\[
J(B^\e\setminus W^\e) - J(B^\e)\ \le\ J(T_h^\pm B_R)- J(B^\e \cup T_h^\pm B_R) \le \frac{1}{h} \int_{W_\e} |x| - R \, dx.
\]
Dividing the previous inequality by $|W^\e|$ and  passing to the limit as $\e\to 0$, in view of the very definition of $\kappa$ \eqref{ineqabove} we get
\[
- \kappa(\bar x,B^\e)\ \le\ \frac{1}{h} (|\bar x|-R).
\]
It follows that $|\bar x|\ge R-h\oc(\bar\rho)$.

Now, let $C$ be the constant of Lemma~\ref{lembd}, and choose $h$ so small  that
\[
 J(B_{R_0/8C})+\frac{1}{h}\int_{B_{R_0/8C}} |x|-\frac{R}{4C}\,dx
\ \le  J(B_{R_0/8C}) -\frac{R_0}{8Ch}|B_{R_0/(8C)}|\ <\ 0,
\]
so that $T_h^\pm B_{R/4C}\neq \emptyset$.
Note that $B_{{R}/{4C}} + B_{{3R}/{4}} \subseteq B_R$. Thus, by Lemma \ref{lemcomp2} $T_h^\pm(B_{{R}/{4C}}) + B_{{3R}/{4}} \subseteq T_h^\pm B_R$.
In particular, 
if $x_0\in T_h^\pm B_{R/(4C)}$
it follows that $B(x_0,\frac{3R}{4})\subseteq T_h^\pm B_R$. 
By the first part of the proof of ii), we find that $B(x_0,|\bar x-x_0|)\subseteq T_h^\pm B_R$
for some $\bar x$ with $|\bar x|\ge R-h\oc(3R/4)$. Hence, recalling also that, thanks to Lemma~\ref{lembd} $x_0\in T_h^\pm B_{R/(4C)}\subseteq B_{R/4}$ ,  
 we obtain that $B_{R/4}\subseteq T_h^\pm B_R$, provided that $h$ is small enough. We
can now use again the previous analysis with $x_0=0$, $\bar\rho \ge  R/4$
and we deduce that if $h$ is small enough, $B_{R-h\oc(R/4)}\subseteq T_h^\pm B_R$. Applying once again the first part of the proof with $x_0=0$ and $\bar \rho \ge R-h\oc(R/4)$ we conclude that, if $h$ is small enough, 
$B_{R-h\oc(R/\sigma)} \subseteq B_{R-h\oc(R-h\oc(R/4))}\subseteq T_h^\pm B_R$.
\end{proof}

\subsection{The time discrete scheme for unbounded sets}
Here we show how to extend  the time discrete scheme to the case of unbounded sets with bounded complement.
To this purpose, we introduce the perimeter $\tilde J$ defined as 
$$
\tilde J(E) := J (\R^N\setminus E) \qquad \text{ for all } E\in \Ins.
$$
Note that $\tilde J$ satisfies all the structural assumptions of generalized perimeters. Let $\tilde \kappa$ be the corresponding curvature. Then, it is easy to see that $\tilde \kappa(x,E) = -\kappa(x, \R^N\setminus E)$, and thus

\begin{equation}\label{defbarc3}
\begin{split}
\max_{x\in\partial B_\rho} \max\{\tilde \kappa(x,B_\rho),- \tilde \kappa(x,\R^N\setminus B_\rho)\} = 
\oc(\rho)
\\
\min_{x\in\partial B_\rho} \min\{\tilde \kappa(x,B_\rho),-\tilde \kappa(x,\R^N\setminus B_\rho)\} = 
\uc(\rho),
\end{split}
\end{equation}
where $\oc(\rho), \, \uc(\rho)$ are the functions defined in \eqref{defbarc0} and \eqref{defbarc}.

For every bounded set  $F$ we denote by 
$\tilde T_h^\pm(F)$ the {maximal and the minimal solution} to 
problem \eqref{varprob}, according to Proposition \ref{exminpro} with $J$ replaced by $\tilde J$. 
Finally, for 
every  $E\subseteq \R^N$ such that $F:= \R^N\setminus E$ is bounded we set  
{
\begin{equation}\label{defunb}
T_h^\pm E := \R^N \setminus \tilde T_h^\mp(\R^N\setminus E).
\end{equation}
As in the case of bounded sets, we let $T_h E:= T_h^- E$.}


Taking into account also \eqref{defbarc3}, one can easily check that Lemmas~\ref{lembd} and ~\ref{lembdh} 
translate into the following statements:

\begin{lemma}\label{lembd2}
For any $R>0$ we have
$\R^N\setminus B_{CR}  \subseteq T_h^\pm (\R^N\setminus B_R)$, where $C$ depends only on the dimension $N$.
\end{lemma}

\begin{lemma}\label{lembdh2} 
Let $C>1$ be such that the statement  of Lemma \ref{lembd} holds,  and let $\uc$, $\oc$  be  as  in \eqref{defbarc}. 
Then, the following holds: 
\begin{itemize}
\item[i)]
 Let $R>0$. Then, $\R^N\setminus B_{R - h \uc (CR)} \subseteq T_h^\pm (\R^N\setminus B_R) $ for every $h>0$ such that $R - h \uc (CR)>0$;
\item[ii)]
Let $R_0>0$ and $\sigma>1$ be fixed. Then, for  $h>0$ small enough (depending on $R_0$ and $\sigma$), we have 
$T_h^\pm (\R^N\setminus B_R)\subseteq \R^N\setminus B_{R-h\oc(R/\sigma)}$  for all $R\ge R_0$.
\end{itemize}
\end{lemma}

\begin{remark}\label{remglobalbound}\textup{
A consequence of  Lemmas \ref{lembd}, \ref{lembd2} 
is that $T_h B_R\subseteq B_{R+hK}$ and
$\R^N\setminus B_{R+hK} \subseteq T_h (\R^N\setminus B_R)  $
 for any $h>0$ and any $R>0$,
where $K$ is defined in \eqref{lwrbdkappa}. 
In particular, iterating these estimates, we deduce that  $T^{[t/h]}_h B_R\subseteq
B_{R+tK}$ and  $\R^N\setminus B_{R+tK}\subseteq T^{[t/h]}_h (\R^N\setminus B_R)$. In the limit as $h\to 0$, we will get an estimate for the extinction time of balls in the superlevels of our level set function (see Proposition \ref{estingue}). 
}
\end{remark}
Note now that  by Lemma~\ref{lemcomp} (applied to $\tilde J$ in place of $J$)  and \eqref{defunb} if $E_1$, $E_2$ are unbounded sets with compact boundary, then 
$$
E_1\subseteq E_2\quad \Longrightarrow \quad T_h^\pm E_1\subseteq T_h^\pm E_2\,.
$$
It remains to consider the case of $E_1$ bounded and $E_2$ unbounded. 
\begin{lemma}\label{lemcomp3}
Let $E_1\in \Ins$ be bounded and let $E_2\in \Ins$ be unbounded, with compact boundary, and such that $E_1\subseteq E_2$. Then, 
$T_h^\pm E_1\subseteq T_h^\pm E_2$.
\end{lemma}
\begin{proof}
Choose $R>0$ so large  that  $E_1$, $\R^N\setminus E_2\subseteq B_R$ and note that by Lemmas~\ref{lemcomp} and~\ref{lembd}
  (applied to $\tilde J$ in place of $J$) we get
  \begin{equation}\label{star}
  \R^N\setminus T^+_hE_2=\tilde T^- ( \R^N\setminus E_2)\subseteq \tilde T^- B_R\subseteq B_{CR}
  \end{equation}
  for some $C>0$ depending only on $N$. Recall that $\tilde T^- _h( \R^N\setminus E_2)$ is the minimal solution of
  $$
  \min\biggl\{J(\R^N\setminus F)+\frac1h\int_F d_{\R^N\setminus E_2}\, dx:\, F\in \Ins\biggr\}\,.
  $$
  Considering the change of variable $\widetilde F:=\R^N\setminus F$ and using that $d_{\R^N\setminus E_2}=-d_{E_2}$, we easily infer that 
  $T^+_h E_2=\R^N\setminus \tilde T^-_h(\R^N\setminus E_2)$ is the maximal solution of

\begin{multline*}
\min\biggl\{J(\widetilde F)-\frac1h\int_{\R^N\setminus\widetilde F} d_{ E_2}\, dx:\, \widetilde F\in \Ins\biggr\}  \\
=\min\biggl\{J(\widetilde F)+\frac1h\int_{B_{CR}}d_{E_2}\, dx-\frac1h\int_{\R^N\setminus\widetilde F} d_{ E_2}\, dx:\, \widetilde F\in \Ins\biggr\}-
\frac1h\int_{B_{CR}}d_{E_2}\, dx\,. 
\end{multline*}
Note now that  
$$
\int_{\widetilde F}d_{E_2}\chi_{B_{CR}}\, dx= \int_{B_{CR}}d_{E_2}\, dx-\int_{\R^N\setminus\widetilde F} d_{ E_2}\, dx
$$
for every $\widetilde F$ with  $\R^N\setminus \widetilde F\subseteq B_{CR}$. It follows, also by \eqref{star}, that  $T^+_h E_2$ is the 
maximal solution of 
\begin{equation}\label{starstar}
\min\biggl\{J(  \widetilde F)+\frac1h\int_{ \widetilde F}d_{E_2}\chi_{B_{CR}}\, dx:\,  \widetilde F\in \Ins\,, \R^N\setminus \widetilde F\subseteq B_{CR}\biggr\}\,.
\end{equation}
By the same reasoning, one can show that $T^-_hE_2$ is the minimal solution of  \eqref{starstar}.
Observing  that $d_{E_2}\chi_{B_{CR}}\leq d_{E_1}$ and that $T^\pm_hE_1\cup T^\pm_hE_2$, $T^\pm_hE_1\cap T^\pm_hE_2$ are admissible competitors for 
\eqref{starstar}, one can argue exactly as in the proof of Lemma~\ref{lemcomp} to conclude that  $T^\pm_hE_1\subseteq T^\pm_hE_2$.
\end{proof}
\subsection{The level-set approach}
Given any  bounded uniformly continuous 
function $u:\R^N\to \R$, constant outside a compact set,  we introduce a transformation  of $u $ which is defined by applying $T_h$ to all the superlevel sets of $u$. 
This is standard and has been done in a similar geometric
setting in many papers  (see~\cite{ChNoIFB,EGK}). 

To this purpose, notice that all the superlevels of $u$ are either bounded or with bounded complement, and that  for any couple of levels $s>s'\in \R$ we have  $\{u>s\}\subseteq \{u>s'\}$. 
Thus, in view of Lemma \ref{lemcomp}  we have $T_h\{u>s\}\subseteq T_h\{u>s'\}$. 

Let $\omega:\R_+\to\R_+$
an increasing, continuous modulus of continuity for $u$. 
Since
$$
\{u>s\} + B_{\omega^{-1}(s-s')}\subseteq \{u>s'\},
$$
by Lemma \ref{lemcomp2} we deduce that
$$
T_h\{u>s\} + B_{\omega^{-1}(s-s')} \subseteq T_h\{u>s'\}.
$$
 It follows
that the sets $T_h\{u>s\}$ are themselves the level sets $\{v>s\}$
of a uniformly continuous function $v=: T_h u$, with the same modulus
of continuity. 
{More precisely, we set $T_h u(x):= \sup \{\lambda\in\R: x\in T_h \{ u>\lambda\}\}$}. 
Notice that, by Lemmas \ref{lembd} and  \ref{lembd2}, also $T_h u$
is constant out of a compact set.  Moreover,
if $u\ge u'$, then $T_hu\ge T_h u'$.
In the two following propositions, equality between sets must be  understood up to negligible sets. 
\begin{proposition}\label{prop:solmin0}
For every $\lambda \in\R$ we have 
$$
T_h(\{u> \lambda\}))= T^-_h(\{u> \lambda\}))= \{ T_hu >\lambda\}\,.
$$
Analogously, 
$$
T^+_h(\{u\geq \lambda\}))= \{ T_hu \geq \lambda\}.
$$
\end{proposition}
\begin{proof}
 For every $\delta\ge 0$  set  
$$E_\delta := T_h(\{u > \lambda+\delta\}), \qquad 
A_\delta:= \{ T_hu >\lambda+\delta\}.
$$ 
We have to prove that $ E_0 = A_0$.
First, notice that by the very definition of $T_hu$,  for every $\delta\ge 0$
\begin{equation}\label{scatole}
A_\delta \subseteq  E_\delta  
\end{equation}
so that in particular 
$A_0 \subseteq E_0$.
To prove the reverse inclusion, observe that 
\begin{equation}\label{unidis}
d_{\{u > \lambda+\delta\}}  \to d_{\{u > \lambda\}}
\end{equation}
uniformly as $\delta\to 0$. Moreover, $A_\delta \nearrow A_0$ in $L^1_{loc}$ as $\delta \to 0$ . By \eqref{unidis} and  by  the lower semicontinuity of $J$, it easily follows that  $A_0$ is a solution of \eqref{varprob} with $E$ replaced by $\{u > \lambda\}$ (or of \eqref{starstar} in the unbounded case, with  $E_2$ replaced by $\{u > \lambda\}$) . Moreover, by \eqref{scatole} it is the minimal one, i.e., it coincides with $E_0$.  
The similar proof of the second statement is left to the interested reader.
\end{proof}

Given a continuous function $u^0$ constant
outside of a bounded set, define 
{

\begin{equation}\label{defuh}
u_h(x,t)= T^{[t/h]}_h(u_0)(x)
\end{equation} 
 for every $h>0$, $t\ge 0$, where $[\cdot]$ denotes the integer part.

}
Proposition~\ref{prop:solmin0} applied to  $u_h(\cdot,(k-1) h)$ 
yields the following:
\begin{proposition}\label{prop:solmin}
For every $h,k>0$ and for every $\lambda \in\R$ we have 
$$
T^-_h(\{u_h(\cdot,(k-1) h) > \lambda\})= \{ u_h(\cdot,k h) >\lambda\}
$$
and
$$
T^+_h(\{u_h(\cdot,(k-1) h) \geq \lambda\})= \{ u_h(\cdot,k h) \geq \lambda\}\,.
$$
\end{proposition}

We have seen that for all $t$,
$u_h(\cdot,t)$ is uniformly continuous (with the same modulus
$\omega$ as $u^0$). Let us now study the regularity in time
of this function.

\begin{lemma}\label{lemUCT}
For any $\e>0$, there exists $\tau>0$ and $h_0>0$
 (depending on $\e$) such that for all  $|t-t'|\le \tau$ and $h\le h_0$ we have
$|u_h(\cdot,t)-u_h(\cdot,t')|< \e$.
\end{lemma}

\begin{proof}
Let $\e>0$ and  let $R_0:= \omega^{-1}(\e/2)/2$. 
Since $\omega$ is a modulus of continuity for $u_h$ it readly follows that for every $x$
\begin{equation}\label{nome0}
B(x,\omega^{-1}(\e/2)) \subseteq \{u_h(\cdot,t) > u_h(x,t) - \e\}. 
\end{equation}
We only treat the case where $\{u_h(\cdot,t) > u_h(x,t) - \e\}$ is bounded, the other being analogous. 
Let $\tau:= R_0/\oc(R_0/4)$. By part ii) of Lemma \ref{lembdh}, and using that $\oc$ is a monotone decreasing function,   
 there exists  $h_0$ depending on $R_0$
such that 
\begin{equation}\label{nome}
B(x,R_0)\subseteq B(x, \omega^{-1} (\e/2) - n h \oc(R_0/4))\subseteq T_h^n B(x,\omega^{-1} (\e/2))
\end{equation}
as long as $\omega^{-1} (\e/2) - n h \oc(R_0/4)\ge R_0$, i.e., as long as $nh\le \tau$. 

Now, let $t'>t$ such that $t' - t \le \tau$, and let $n:= [(t' - t)/h]$. Since $n h \le \tau$, by \eqref{nome0}, \eqref{nome}, Lemma \ref{lemcomp}, and Proposition~\ref{prop:solmin} we have
\begin{eqnarray*}
\{u_h(\cdot,t') > u_h(x,t') - \e\} =  \{u_h(\cdot,t + nh) > u_h(t,x) - \e\} = 
\\
T_h^n \{u_h(\cdot,t) > u_h(x,t) - \e\} \supseteq T_h^n B(x,\omega^{-1} (\e/2) )
\supseteq B(x,R_0).
\end{eqnarray*}
 In particular, 
 $
 u_h(x,t') > u_h(x,t) - \e. 
 $
 In order to show 
 $$
 u_h(x,t') < u_h(x,t) + \e 
$$
we proceed in a similar way. Precisely, we observe that  
$$ 
B(x,\omega^{-1} (\e/2)) \subseteq \{u_h(\cdot,t) < u_h(x,t) + \e\}, 
$$
 that is,  
$$
 \{u_h(\cdot,t) \geq u_h(x,t) + \e\}  \subseteq \R^N\setminus B(x,\omega^{-1} (\e/2))\,.
$$
We then proceed as in the first part of the proof, but now using Lemma~\ref{lembd2} instead of Lemma~\ref{lembd} and Lemma~\ref{lemcomp3} instead of Lemma~\ref{lemcomp}.

\end{proof}

\subsection{Convergence analysis} 
In this subsection we show that any limit  of the discrete evolutions is a viscosity solution.
Recalling Lemma~\ref{lemUCT} and the uniform continuity in space of $u_h$, by a straightforward variant of Ascoli-Arzel\`a's Theorem we deduce the precompactness of $u_h$. Moreover, in view of Remark \ref{remglobalbound} we deduce also that the limit $u$ is constant out of a compact set. Summarizing, the following proposition holds.
\begin{proposition}\label{Compro}
Let $T>0$. Up to a subsequence, $u_h$ converges  uniformly on $\R^N\times[0,T]$
as $h\to 0$ to a function $u(x,t)$, which is bounded and  uniformly continuous, and constant 
out of a compact set.  
\end{proposition}

For every $r>0$, set
\begin{equation}\label{hatc}
\hat c(r) := \max \{1, \overline c(r)\}.
\end{equation}

Given $ r_0>0$, let $r(t)$ be the solution of the following ODE
\begin{equation}\label{ODE}
\left\{
\begin{array}{l}
\dot{r}(t) = - \hat c(r(t));\\
r(0) = r_0
\end{array}
\right.
\end{equation}
Notice that \eqref{ODE} admits a unique solution $r(t)$ until some extinction  time $T^*(r_0)$ with $r(T^*)=0$. 
\begin{proposition}\label{estingue}
Let $u(x,t)$ be the function given by Proposition \ref{Compro}, let $ \lambda \in \R$, and let $B(x_0,r_0)\subset \{u(\cdot,t_0) > \lambda\}$.
Then, $B(x_0,r(t- t_0))\subset \{u(\cdot,t) > \lambda\}$ for every $t \le T^*(r_0) + t_0$, where $r(t)$ is the solution of the ODE \eqref{ODE} and $T^*(r_0)$ is its extinction time.
The same   statement holds by replacing  the superlevel of $u$ with  its sublevel.
\end{proposition}
\begin{proof}
We only treat the case of $\{u(\cdot,t) > \lambda\}$ bounded, since the other one is analogous. By assumption, if $R_0<r_0$, for $h$ small enough
$B(x_0,R_0)\subset \{u_h(\cdot,t_0) > \lambda\}$. Let $\sigma>1$ and
$R_0$ be defined recursively by $R_{n+1}=R_n-h\overline{c}(R_n/\sigma)$.
By Lemmas~\ref{lemcomp}, ~\ref{lembdh}, ~and~\ref{prop:solmin} one has that
$B(x_0,R_{[(t-t_0)/h]+1})\subset \{u_h(\cdot,t) > \lambda\}$ for $t\ge t_0$,
as long as $R_{[(t-t_0)/h]+1}>0$.
Let also $r_\sigma$ be the unique solution of $\dot{r}_\sigma(t)=-\hat c(r_\sigma(t)/\sigma)$ with initial value $r_\sigma(0)=R_0$.
One observes that if $r_\sigma(nh)\le R_n$, then
\begin{multline*}
r_\sigma((n+1)h) \le  R_n - 
\int_{nh}^{(n+1)h} \hat c\left(\frac{r_\sigma(s)}{\sigma}\right)\,ds
\\ \le
R_n - \int_{nh}^{(n+1)h} \hat{c}\left(\frac{R_n}{\sigma}\right)\,ds
\le
R_n-\int_{nh}^{(n+1)h} \overline{c}\left(\frac{R_n}{\sigma}\right)\,ds = R_{n+1}
\end{multline*}
since $\hat{c}$ is nondecreasing. As a consequence,
$B(x_0,r_\sigma(h[(t-t_0)/h]+h)\subset \{u_h(\cdot,t) > \lambda\}$ for $t\ge t_0$
as long as the radius is positive. We conclude sending $h\to 0$,
then $R_0\to r_0$ and $\sigma\to 1$. 
The proof of the last part of the proposition is very similar. One observes that by Lemmas~\ref{lemcomp3}, ~\ref{lembdh2}, ~and~\ref{prop:solmin}, we  have (with the same definition of $R_n$) $\{u_h(\cdot,t) > \lambda\}\subset \R^N\setminus B(x_0,R_{[(t-t_0)/h]+1})$, that is 
$B(x_0,R_{[(t-t_0)/h]+1})\subset \{u_h(\cdot,t) \leq \lambda\}$ for $t\ge t_0$,
as long as $R_{[(t-t_0)/h]+1}>0$. The conclusion then follows as before. 
\end{proof}
We are now in a position to state and prove the main result of this section.
\begin{theorem}\label{ThMM}
The function $u$ provided by Proposition \ref{Compro}
is a viscosity
solution of the Cauchy problem \eqref{levelsetf} in the sense of Definition \ref{defvisco}.
\end{theorem}
\begin{remark}\textup{
We observe that this holds under assumptions C) and D) on the curvature. If
in addition C') holds, then the limit flow is unique and one also deduces
that the whole family $(u_h)_{h>0}$ converges uniformly as $h\to 0$.
}
\end{remark}

\begin{proof}
We denote by $u_{h_k}$ a subsequence of $u_h$ converging to $u$. Let us prove
that $u$ is a subsolution (the proof that it is a supersolution
is identical). Let $(\bar x, \bar t) \in \R^N\times (0,T)$.
Let $\f$  be a $C^{\regu}$ admissible test function at
$(\bar x, \bar t)$, and assume that  $(\bar x, \bar t)$ is a maximum point of $u-\f$.
We need to show that
\begin{equation}\label{viscosub}
\frac{\partial \f}{\partial t}(\bar x,\bar t)
+|D\f(\bar x,\bar t)|\kappa_*(\bar x,D\f(\bar x,\bar t),
D^2\f(\bar x,\bar t),\{\f(\cdot,\bar t)\ge\f(\bar x,\bar t)\})\ \le\ 0.
\end{equation}
\smallskip

\noindent\textbf{Step 1.}
Let us first assume that $D\f(\bar x,\bar t)\neq 0$.
By Remark \ref{convo} we can  assume that  this is a strict maximum point and that $\f$ is smooth.


If the maximum is strict,
then by standard methods we can find $(x_k,t_k)\to (\bar x,\bar t)$
such that $u_{h_k}-\f$ has a maximum at $(x_k,t_k)$.
Moreover, for $k$ large enough, $D\f(x_k,t_k)\neq 0$. 
We have that for all $(x,t)$,
\begin{equation}\label{ineqk}
u_{h_k}(x,t)\ \le\ \f(x,t) \,+c_k 
\end{equation}
where $c_k:= [u_{h_k}(x_k,t_k)-\f(x_k,t_k)]$,
with equality if $(x,t)=(x_k,t_k)$.

Let $\eta >0$ and set
\begin{equation}\label{deffeta}
\f^{\eta}_{h_k}(x)\ =\ \f(x,t_k) \,+\, c_k\,+
\, \frac{\eta}{2}Q(x-x_k)\,,
\end{equation}
where $Q$ is as in Lemma \ref{Qlemma} and $Q(z)=|z|^2$ for $|z|$ sufficiently small. 
Then, for all $x\in \R^N$,
\[
u_{h_k}(x,t_k)\ \le\ \f^{\eta}_{h_k}(x)
\]
with equality \textit{if and only if} $x=x_k$. We set $l_k:=u_{h_k}(x_k,t_k) = \f^{\eta}_{h_k}(x_k)$.

By Lemma \ref{Qlemma}, 
we can assume that  $\eta$ is  such that
the superlevel sets $\{\f^{\eta}_{h_k} \ge l_k\}$ are not critical for all $k$. 
Let $\e>0$ 
and set
{
\[
W_\e \ :=\ \{x\in\R^N\,:\, u_{h_k}(x,t_k)\ge l_k - \e\}
\setminus \{x\in\R^N\,:\, \f^{\eta}_{h_k}(x)\ge l_k\}\ .
\]
}
It is easy to see that for   $\e>0$ sufficiently small  $|W_\e|>0$, and converges to $\{x_k\}$ in the Hausdorff sense as
$\e\to 0$.
Now, if $\{ u_{h_k}(\cdot,t_k)\ge l_k - \e\}$ is bounded, by minimality we have
\begin{multline} \label{eq:bounded}
J(\{ u_{h_k}(\cdot,t_k)\ge l_k - \e\})\,+\,\frac{1}{h_k}
\int_{\{ u_{h_k}(\cdot,t_k)\ge l_k - \e\}} d_{\{ u_{h_k}(\cdot,t_k-h_k)\ge l_k - \e\}}(x)\,dx
 \\
 \le\,
J(\{ u_{h_k}(\cdot,t_k)\ge l_k - \e\}\cap\{\f^{\eta}_{h_k}\ge l_k\})
\\
 +\,\frac{1}{h_k}
\int_{\{ u_{h_k}(\cdot,t_k)\ge l_k - \e\}\cap\{\f^{\eta}_{h_k}\ge l_k \}}
 d_{\{ u_{h_k}(\cdot,t_k-h_k)\ge l_k - \e\}}(x)\,dx\,.
\end{multline}
Adding to both sides the term
$J(\{ u_{h_k}(\cdot,t_k)\ge l_k - \e\}\cup\{\f^{\eta}_{h_k}\ge l_k \})$
and using \eqref{eq:submo}, we obtain
\[
J(\{\f^{\eta}_{h_k}\ge l_k  \}\cup W_\e)-J(\{\f^{\eta}_{h_k}\ge l_k \})
\\+\,\frac{1}{h_k}\int_{W_\e}  d_{\{ u_{h_k}(\cdot,t_k-h_k)\ge l_k - \e\}}(x)\,dx \ \le\ 0\,.
\]
By \eqref{ineqk}, $\{ u_{h_k}(\cdot,t_k-h_k)\ge l_k - \e\}
\subseteq  \{\f(\cdot,t_k-h_k) \ge l_k  - c_k - \e\}$, so that we also have
\begin{multline}
\label{finalineq}
J(\{\f^{\eta}_{h_k}\ge l_k \}\cup W_\e)-J(\{\f^{\eta}_{h_k}\ge l_k \})
\\+\,\frac{1}{h_k}\int_{W_\e}  d_{\{\f(\cdot,t_k-h_k) \ge l_k-c_k - \e\}}(x)\,dx \ \le\ 0\,.
 \end{multline}

If instead $\{ u_{h_k}(\cdot,t_k)\ge l_k - \e\}$ is unbounded, then inequality \eqref{eq:bounded} must be replaced by
\begin{multline*} 
J(\{ u_{h_k}(\cdot,t_k)\ge l_k - \e\})\,+\,\frac{1}{h_k}
\int_{\{ u_{h_k}(\cdot,t_k)\ge l_k - \e\}\cap B_{R}} d_{\{ u_{h_k}(\cdot,t_k-h_k)\ge l_k - \e\}}(x)\,dx
 \\
 \le\,
J(\{ u_{h_k}(\cdot,t_k)\ge l_k - \e\}\cap\{\f^{\eta}_{h_k}\ge l_k\})
\\
 +\,\frac{1}{h_k}
\int_{\{ u_{h_k}(\cdot,t_k)\ge l_k - \e\}\cap\{\f^{\eta}_{h_k}\ge l_k \}\cap B_{R}}
 d_{\{ u_{h_k}(\cdot,t_k-h_k)\ge l_k - \e\}}(x)\,dx\,,
\end{multline*}
for $R$ sufficiently large, see \eqref{starstar}. Then, arguing as before, one obtains again \eqref{finalineq}.

Notice that for $z\in W_\e$ we have
\begin{equation}\label{E1}
l_k - \e < \f(z,t_k) + c_k +\frac\eta2Q(z-x_k)  < l_k.
\end{equation}
Since, in turn, $\f(z,t_k) + c_k\geq l_k - \e$ it follows that $\frac\eta2Q(z-x_k)<\e$ and thus, for $\e$ small enough,  
\begin{equation}\label{inparticular}
W_\e \subseteq B_{C\sqrt{\e}}(x_k).
\end{equation}
Moreover, for every $z\in W_\e$ 
\begin{equation}\label{E2}
\f(z,t_k-h_k) = \f(z,t_k)-h_k\partial_t \f(z,t_k) +
h_k^2 \int_0^1(1-s) \partial_{tt}^2 \f(z,t_k-sh_k) \,ds\,.
\end{equation}
Let $y$ be a point of minimal distance from $z$ such that $\f(y,t_k-h_k)=l_k-c_k -\e$.
Then, $|z-y|=|d_{\{\f (\cdot,t_k-h_k)\ge l_k-c_k -\e\}}(z)|$, 
and 
\begin{equation}\label{parallel}
(z-y)\cdot D\f(y,t_k-h_k)=\pm |z-y||D\f(y,t_k-h_k)|,
\end{equation}
 with
a `$+$' if $\f(z,t_k-h_k)>l_k-c_k -\e$ and a `$-$' else, so
that the sign is opposite to the sign of $d_{\{\f (\cdot,t_k-h_k)\ge l_k-c_k -\e\}}(z)$.
Hence,
\begin{multline}\label{E3}
\f(z,t_k-h_k)\,=\, \f(y,t_k-h_k)+(z-y)\cdot D\f(y,t_k-h_k)
\\+
\int_0^1(1-s) (D^2\f(y+s(z-y),t_k-h_k)(z-y))\cdot(z-y)\,ds
\\=\,l_k-c_k -\e - d_{\{\f (\cdot,t_k-h_k)\ge l_k-c_k -\e\}}(z)| D\f(y,t_k-h_k)|
\\+
\int_0^1(1-s) (D^2\f(y+s(z-y),t_k-h_k)(z-y))\cdot(z-y)\,ds\,.
\end{multline}

By \eqref{E1} we deduce in particular $ \ \f(x,t_k) \,+\, c_k\, < l_k$, i.e., 
{
\begin{equation}\label{E1bis}
  - \f(x,t_k)  \ge \, c_k\,  -  l_k.
\end{equation}
}
Combining \eqref{E1bis}, \eqref{E2}, and \eqref{E3}, we deduce 
\begin{multline*}
d_{\{\f(\cdot,t_k-h_k)\ge l_k-c_k -\e\}}(z)| D\f(y,t_k-h_k)|
\\ \ge\ -\e  + h_k\partial_t \f(z,t_k) \,-\,
h_k^2 \int_0^1(1-s) \partial_{tt}^2 \f(z,t_k-sh_k) \,ds\\ +\,
\int_0^1(1-s) (D^2\f(y+s(z-y),t_k-h_k)(z-y))\cdot(z-y)\,ds\,.
\end{multline*}
Note that, in view of \eqref{E1},  $|\f(z,t_k)-\f(y,t_k)|\leq \e+Ch_k=O(h_k)$,
provided that $\e<< h_k$ are small enough.
In turn, by \eqref{parallel} as $|D\f(y,t_k-h_k)|$ is bounded away from zero, we  have 
 $|z-y|=O(h_k)$ and, using also \eqref{inparticular}, we deduce
\begin{multline}
\label{ineqdist}
\frac{1}{h_k}d_{\{\f(\cdot,t_k-h_k)\ge l_k -c_k - \e\}}(z) \, \ge\,
\frac{\partial_t \f(z,t_k) -\frac{\e}{h_k}  +  O(h_k)}{| D\f(y,t_k-h_k)|}
\\=\,\frac{\partial_t \f(x_k,t_k) +O(\sqrt{\e}) -\frac{\e}{h_k}  +  O(h_k)}
{| D\f(x_k,t_k)| +O(\sqrt{\e}) + O(h_k)}\,.
\end{multline}

We now focus on the term
\[
J(\{\f^{\eta}_{h_k}\ge l_k\}\cup W_\e)-J(\{\f^{\eta}_{h_k}\ge l_k\})
\]
of inequality~\eqref{finalineq}. Thanks to~\eqref{ineqbelow},
if $\e$ is small enough we know that
\begin{multline}\label{ineqJfinal}
J(\{\f^{\eta}_{h_k}\ge l_k\}\cup W_\e)-J(\{\f^{\eta}_{h_k}\ge l_k\})
\\\ge\ 
|W_\e|(\kappa(x_k,\{\f^\eta_{h_k}\ge \f^\eta_{h_k}(x_k)\})-o_\e(1))\,,
\end{multline}
recalling that  $\f^\eta_{h_k}(x_k)$ is not a critical value of $\f^\eta_{h_k}$.

Using therefore~\eqref{finalineq}, \eqref{ineqdist} and~\eqref{ineqJfinal},
dividing by $|W_\e|$ and sending $\e\to 0$, we deduce that (for a.e.~$\eta>0$
small)
\[
\frac{\partial_t \f(x_k,t_k) +  O(h_k)}
{| D\f(x_k,t_k)| + O(h_k)}\,+\,\kappa(x_k,\{\f^\eta_{h_k}\ge \f^\eta_{h_k}(x_k)\})
\ \le\ 0\,.
\]

Letting simultaneously $\eta\to 0$ and $k\to\infty$
and using Lemma \ref{scif}  we deduce~\eqref{viscosub}.
\smallskip

\noindent\textbf{Step 2.}
Now we consider the case $D\f(\bar z)= 0$ and we  show that  $\f_t(\bar z) \le 0$. 
Let $\psi_n$ be defined as in \eqref{psienne}
and let 
$z_n=(x_n, t_n)$ be a sequence of 
{ 
maximizers 
}
of $u- \psi_n$, such that 
$x_n\to \bar x$ and $t_n\to \bar t^-$. If $x_n\neq \bar x$ for a (not relabeled) subsequence, then (for large $n$) 
$D  \psi_n(x_n, t_n)\neq 0$ and \eqref{viscosub} holds for $\psi_n$ at $z_n$. Passing to the limit and using the properties of $f$ (where $f$ is the function appearing in \eqref{psienne}), we deduce that  $\frac{\partial \f}{\partial t}(\bar z) \le 0$ (see \eqref{boh} for
the details). 

We therefore assume that $z_n=(\bar x, t_n)$  for all $n$ sufficiently large.
Set $b_n:=\bar t-t_n$ and set
\begin{equation}\label{errenne}
r_n:= f^{-1} (a_n  b_n) \,,
\end{equation}
where $a_n\to 0$ is chosen  so that
the extinction time $T^*(r_n)$ of the solution of \eqref{ODE} with $r_0$ replaced by  $r_n$, 
satisfies $T^*(r_n) \ge 2 b_n$ for $n$ large enough.
To show that such a choice for $a_n$ is possible, set
$$
g (t) = \sup_{0\le s\le t} \hat c (f^{-1} (s)) f' ( f^{-1}(s)),
$$
and notice that  $g (t) \leq \hat c(t)$ for $t$ small,   it  is non decreasing in $t$, and $g(t)\to 0$ as $t\to 0$  thanks to \eqref{ISvera}.
 We have
\begin{multline}\label{inclu}
\frac{T^*(r_n)}{b_n} \ge \frac{1}{b_n} \int_{r_n/2}^{r_n} \frac{1}{\hat c(r)} =  \frac{1}{b_n} \int_{f^{-1}(a_n b_n /2) }^{f^{-1}(a_n b_n)}  \frac{1}{\hat c(r)} 
\\
= \frac{a_n}{2} \medint_{a_n b_n /2}^{a_n b_n} \frac{1}{\hat c(f^{-1}(s)) f'(f^{-1}(s))} \, ds
\ge \frac{a_n}{2} \frac{1}{g(b_n)} = 2.
\end{multline}
where the last equality holds if we choose $a_n = 4 g(b_n) \to 0$.

By definition of $\psi_n$, we have that 
{
\begin{align*}
B(\bar x, r_n)& \subset \left\{ \psi_n(\cdot,t_n) \le \psi_n(\bar x, t_n) + 2f(r_n) \right\}\\
&\subset
 \left\{ u(\cdot,t_n) \le u(\bar x, t_n) + 2f(r_n) \right\}\,.
\end{align*}
}
Note that the last inclusion follows from  the maximality 
of $u-\psi_n$ at $z_n$ and the fact that $u(z_n)=\psi_n(z_n)$.
By \eqref{inclu} and Proposition \ref{estingue},
$$
\bar x\in \left\{ u(\cdot,\bar t) \le u(\bar x, t_n) + 2f(r_n)\right\}\,.
$$
Thus, using also the maximality of $u-\f$ at $\bar z$, and recalling \eqref{errenne}, we have
$$
\frac{\f(\bar x, t_n) - \f(\bar z)}{-b_n}\leq 
\frac{u(\bar x, t_n) - u(\bar x, \bar t)}{-b_n}\leq \frac{2f(r_n)}{-b_n} = -2 a_n.
$$
Passing to the limit, we conclude that $\partial_t\f(\bar z)\leq 0$.
\end{proof}

\subsection{Perimeter descent}\label{subsec:pd}
In this part we address the problem of the perimeter descent for variational curvature flows. The results refer to any viscosity solution $u:\RT\to \R$ to \eqref{levelsetf}, where $k$ is  the first variation of a generalized perimeter in the sense of Definition \ref{deffirstvar}. Throughout this subsection we also assume that the additional conditions stated in Subsection~\ref{subsec:1storder} and \ref{subsec:2ndorder} hold, so that such a solution is unique and coincides with the one built through the minimizing movements. 

First, we generalize to our setting a fact that is well known in the context of the mean curvature flow: whenever there is no fattening, the perimeter decreases in time. 
\begin{proposition}\label{propnoflat}
Let $0\le t_1\le t_2\le T$, let $\lambda\in\R$ and assume that  $| \{u(\cdot,t_2) = \lambda\}|=0$. Then, 
$$
J(\{u(\cdot,t_2)>\lambda\}) \le J(\{u(\cdot,t_1)>\lambda\}).
$$
\begin{proof}
Set $\tilde u^0:= d_{\{u(\cdot,t_1)>\lambda\}}$, and let $\tilde u:[0,T-t_1]\to \R$ be the viscosity solution of \eqref{levelsetf} with initial condition $\tilde u^0$. 
By Remark~\ref{intrinsicsuplev}, we get 
$$
\{\tilde u (\cdot, t)>\lambda\} = \{u(\cdot, t+t_1)>\lambda\}
$$ 
for every $t\in [0,T-t_1]$. Let now $\tilde u_h$ be the approximate solution  
defined in \eqref{defuh}. Then, by Proposition~\ref{prop:solmin} we have
$$
J(\{\tilde u_h(\cdot, t_2-t_1))>\lambda\} \le J(\{\tilde u_h(\cdot, 0))>\lambda\}.
$$
Since $\tilde u_h\to\tilde u$ pointwise (indeed, uniformly) and since  
$$
|\{\tilde u(\cdot,t_2-t_1) =\lambda\}|  = |\{u(\cdot,t_2) = \lambda\}|=0,
$$
we easily deduce that 
$$
\{\tilde u_h(\cdot, t_2-t_1))>\lambda\} \to \{\tilde u(\cdot, t_2-t_1))>\lambda\} = \{ u(\cdot, t_2))>\lambda\}
$$
in measure, as $h\to\infty$. By the lower semicontinuity of $J$ we conclude
\begin{multline}
J(\{ u(\cdot, t_2))>\lambda\})\le \liminf_h J(\{\tilde u_h(\cdot, t_2-t_1))>\lambda\} 
\\
\le J(\{\tilde u_h(\cdot, 0))>\lambda\} =
J(\{ u(\cdot, t_1))>\lambda\}. 
\end{multline}
\end{proof}
\end{proposition}

\begin{remark}
{
\rm
A natural question is whether the assumption of Proposition \ref{propnoflat} is satisfied. It seems reasonable to believe that, whenever the initial set  
$\{u^0>\lambda\}$ is smooth enough, then there are no flat levels along the flow. To our knowledge, such a result is not known even for the canonical mean curvature flow. On the other hand, if the initial set $E^0$ is star shaped, one can build $u^0$ such that all its superlevels are homothetic to $E^0$.  
In view of the homogeneity properties of the mean curvature and of  the geometric evolution equation \eqref{levelsetf}, all the superlevels  
evolve staying homothetic to each other. As a consequence, superlevels are never flat, and in turn the perimeter decreases along the flow. 
This is the case whenever a generalized curvature is homogeneous with respect to dilations, i.e.   there exists $\alpha>0$ such that 
$\kappa(x,l E) = l^{-\alpha} \kappa(x,E) $ for every $l>0$ and $E\in\Reg$.
}
\end{remark}

Finally, we introduce a relaxed perimeter, defined on open sets,  that always decreases along the flow. 

\begin{definition}\label{relJ}
For every open set $A\subset \R^N$ with compact boundary set 
$
\tilde J(A) := \inf  \liminf J(A_n)
$
where the infimum is taken among all sequences of open sets { $A_n$ with $\bar A_n\in\Reg$,  $\bar{A_n}\subset A$  and $ \R^N\setminus A_n \to \R^N\setminus A$ in the Hausdorff sense.  }
\end{definition}

\begin{remark}\label{differente}
{\rm 
By the lower semicontinuity property of $J$, we  have  $\tilde J(A) \ge J(A)$ for every open set $A$ with compact boundary. The converse inequality is in general false. For instance let $J$ be the standard perimeter and let $A:= B_1\setminus \{xy =0\}$. Then, $J(A)= J(B_1)$, while it is easy to see that
$\tilde J(A)= J(B)+ 4$. It is well known (see \cite{EvansSpruckI}) that if $u^0= d_A$, then the level-set $\{ u(\cdot, t ) = 0\}$ is fat for every positive time. Moreover, 
$$
\lim_{t\to 0} J (\{ u(\cdot, t ) > 0\}) = \tilde J (A). 
$$  
In particular, the perimeter $J$ (instantaneously) increases along the geometric flow.
The example somewhat motivates Definition~\ref{relJ}.
 As we will see, the relaxed perimeter $\tilde J$ instead is always non increasing. 
}
\end{remark}

\begin{remark}\label{equidefre}
{\rm

Clearly, in Definition \ref{relJ} we can always assume that, { whenever $A$ is bounded, $A_n$ are compactly contained in $A_{n+1}$ for every $n$ (and a similar condition for unbounded sets)}. Moreover, we can remove the regularity assumption on $A_n$ without affecting the notion of $\tilde J$. 
Indeed, let $\hat J$ be defined as in Definition \ref{relJ}, but without the requirement of the $C^{\regu}$-regularity. 
Clearly $\hat J\le \tilde J$. To prove the converse inequality, consider an optimal sequence of open sets $\hat A_n$ such that $J(\hat A_n)\to\hat J(A)$. 
It is enough to regularize each $A_n$ in order to have an optimal sequence $\tilde A_n$ for $\tilde J$. This can be easily done in view of 
Lemma \ref{lemapprox}. The details are left to the reader. 
}
\end{remark}

 We state a lemma which clarifies the role of Definition \ref{relJ} in the viscosity approach to geometric flows. For the reader convenience, we omit its technical but straightforward proof, 

\begin{lemma}\label{provu}
Let $A_0 \subset  \R^N$ be open with compact boundary, and let $\bar{A_n}\in \Reg$ with $\bar{A_n}\subset A_n$ compactly contained in  for every $n$.

Then, there exists a one-Lipschitz function $u_{A_0}$ and a sequence $\lambda_n\to 0$ such that
\begin{itemize} 
\item[1)] $A_0 = \{ u_{A_0} > 0 \} $;
\item[2)] $A_n= \{ u_{A_0} > \lambda_n \}$; 
\item[3)] $u_{A_0} = d_{A_n} + c_n$ in a neighborhood of $\partial A_n$, for some suitable constant $c_n$. 
\end{itemize}
\end{lemma}
 
\begin{proposition}
The relaxed perimeter $\tilde J$ decreases along the geometric flow. More precisely, for every $\lambda\in \R$ the function 
$t\to \tilde J(\{u(\cdot, t) > \lambda \})$ is not increasing.  
\end{proposition} 

\begin{proof}
To easy notations, we will assume $\lambda=0$.  Let $0\le t_1\le t_2\le T$. We have to prove that 
$$
\tilde J(\{ u(\cdot, t_2) > 0 \}) \le \tilde J(\{ u(\cdot, t_1) >0 \}).
$$
Let $(A_n)$ be an optimal sequence  for Definition \ref{relJ}  with $A$ replaced  by  
$A_0:= \{ u(\cdot, t_1) >0) \}$. Clearly, we may assume that  $\bar{A_n}\subset A_{n+1}$ for every $n$. Moreover, let $\lambda_n$, $u_{A_0}$ be as in  Lemma \ref{provu}. By property 3) of Lemma \ref{provu}, we have that the function that associate to $\lambda$ the corresponding superlevel set  of $u_{A_0}$ is continuous from  a neighborhood of each $\lambda_n$ to $\Reg$.
In particular, the function $\lambda \to J(\{ u_{A_0}(\cdot) >\lambda) \} )$ is continuous at each $\lambda_n$. Notice that all except countably many levels of $u_{A_0}$ have null measure.   Therefore, there exists a sequence $\tilde \lambda_n \to 0$ such that
\begin{itemize}
\item[i)] $| \{ u_{A_0}(\cdot) =\tilde \lambda_n) \} | = 0 $ for every $n$;
\item[ii)] $J(\{ u_{A_0}(\cdot) >\tilde \lambda_n \} ) \to \tilde  J(\{ u(\cdot, t_1) >\lambda \})$ as $n\to\infty$.
\end{itemize}

Let $\tilde u:[0,T-t_1]\to\R$ be the solution to \eqref{levelsetf} with initial condition $u_{A_0}$. By Proposition \ref{propnoflat} we have  
$$
J(\{\tilde u(\cdot, t_2 -t_1) >\tilde \lambda_n \}) \le  J(\{ \tilde u(\cdot, 0) >\tilde \lambda_n \}).
$$
Letting $n\to \infty$, by ii)  and by the very definition of $\tilde J$ we get
$$
\tilde J(\{ u(\cdot, t_2) > 0  \})  = \tilde J(\{\tilde u(\cdot, t_2 -t_1 ) > 0  \}) \le  \tilde J(\{ \tilde u(\cdot, 0) > 0 \}) = \tilde J(\{  u(\cdot, t_1) > 0\}),
$$
where the first equality follows by Remark \ref{intrinsicsuplev}.

\end{proof}

\bibliography{NLGF}

\end{document}